              \def\version{13 August 2018}		        	%

\documentclass[reqno,11pt]{amsart} 
\usepackage[utf8]{inputenc}
\usepackage{amsmath} 
\usepackage[mathscr]{eucal}
\usepackage{amssymb}
\usepackage{srcltx} 
\usepackage{dsfont}
\usepackage{hyperref}
\usepackage{color}
\usepackage{enumerate}

\numberwithin{equation}{section}
 
\newcommand{\ProofEnde}{\hfill {$\square$}
}
\DeclareFontFamily{U}{mathx}{\hyphenchar\font45}
\DeclareFontShape{U}{mathx}{m}{n}{
      <5> <6> <7> <8> <9> <10>
      <10.95> <12> <14.4> <17.28> <20.74> <24.88>
      mathx10
      }{}
\DeclareSymbolFont{mathx}{U}{mathx}{m}{n}
\DeclareFontSubstitution{U}{mathx}{m}{n}
\DeclareMathSymbol{\bigtimes}{1}{mathx}{"91}

\def\emptyset{\varnothing} 

\def\d{{\rm d}} 
\def\e{\varepsilon} 
 
\newfam\Bbbfam 
\font\tenBbb=msbm10 
\font\sevenBbb=msbm7 
\font\fiveBbb=msbm5 
\textfont\Bbbfam=\tenBbb 
\scriptfont\Bbbfam=\sevenBbb 
\scriptscriptfont\Bbbfam=\fiveBbb 
 
\def\mmax{m_{\max}}

\newcommand{\R}     {\mathbb{R}} 
\newcommand{\Z}     {\mathbb{Z}} 
\newcommand{\N}     {\mathbb{N}} 
\renewcommand{\P}   {\mathbb{P}}

\newcommand{\smfrac}[2]{\textstyle{\frac {#1}{#2}}}

\def\1{{\mathchoice {1\mskip-4mu\mathrm l}      
{1\mskip-4mu\mathrm l} 
{1\mskip-4.5mu\mathrm l} {1\mskip-5mu\mathrm l}}} 
\newcommand{\ssup}[1] {{{\scriptscriptstyle{({#1}})}}} 
\def\comment#1{} 
\newtheoremstyle{thm}{2ex}{2ex}{\itshape\rmfamily}{} 
{\bfseries\rmfamily}{}{1.7ex}{} 
 
\newtheoremstyle{rem}{1.3ex}{1.3ex}{\rmfamily}{} 
{\itshape\rmfamily}{}{1.5ex}{}

\renewcommand{\theequation}{\thesection.\arabic{equation}} 
 
\newtheorem{theorem}{Theorem}[section] 
\newtheorem{lemma}[theorem]{Lemma} 
\newtheorem{prop}[theorem] {Proposition} 
 
\newtheorem{remark}[theorem]  {Remark} 
\newtheorem{defn}[theorem] {Definition}

\theoremstyle{definition}

\renewcommand{\d}{{\rm d}} 
 
\newcommand{\eps}{\varepsilon} 
\newcommand{\Leb}{{\rm Leb}}

\newcommand{\supp}{{\operatorname {supp}}}

\newcommand{\SIR}{\mathrm{SIR}}
\newcommand{\tensor}{\otimes}

\def\supp{\mathrm{supp}}
\def\kmax{k_{\max}}

\newcommand{\Acal}  {{\mathcal A}}

\newcommand{\Fcal}   {{\mathcal F }} 
 
\newcommand{\Hcal}   {{\mathcal H }}

\newcommand{\Mcal}   {{\mathcal M }} 
 
\newcommand{\Ocal}   {{\mathcal O }}

\newcommand{\Scal}   {{\mathcal S }}

\newcommand\numberthis{\addtocounter{equation}{1}\tag{\theequation}}
\renewcommand{\e}   {{\operatorname e }}

\definecolor{Red}{rgb}{1,0,0}

\begin{document} 
 
\title[A Gibbsian model for message routeing]{A Gibbsian model for message routeing \\ in highly dense multihop networks}
\author[Wolfgang König and András Tóbiás]{}
\maketitle
\thispagestyle{empty}
\vspace{-0.5cm}

\centerline{{\sc Wolfgang K\"onig\footnote{TU Berlin, Straße des 17. Juni 136, 10623 Berlin, and WIAS Berlin, Mohrenstra{\ss}e 39, 10117 Berlin, {\tt koenig@wias-berlin.de}}} and {\sc András Tóbiás\footnote{Berlin Mathematical School, TU Berlin, Straße des 17. Juni 136, 10623 Berlin, {\tt tobias@math.tu-berlin.de}}}}
\renewcommand{\thefootnote}{}
\vspace{0.5cm}
\centerline{\textit{WIAS Berlin and TU Berlin, and TU Berlin}}

\bigskip
\centerline{\small(\version)} 
\vspace{.5cm} 
 
\begin{quote} 
{\small {\bf Abstract:}} 
We investigate a probabilistic model for routeing of messages in relay-augmented multihop ad-hoc networks, where each transmitter sends one message to the origin. Given the (random) transmitter locations, we weight the family of random, uniformly distributed message trajectories by an exponential probability weight, favouring trajectories with low interference (measured in terms of signal-to-interference ratio) and trajectory families with little congestion (measured in terms of the number of pairs of hops using the same relay). Under the resulting Gibbs measure, the system targets the best compromise between entropy, interference and congestion for a common welfare, instead of an optimization of the individual trajectories. 

In the limit of high spatial density of users, we describe the totality of all the message trajectories in terms of empirical measures. Employing large deviations arguments, we derive a characteristic variational formula for the limiting free energy and analyse the  minimizer(s) of the formula, which describe the most likely shapes of the trajectory flow. The empirical measures of the message trajectories well describe the interference, but not the congestion; the latter requires introducing an additional empirical measure. Our results remain valid under replacing the two penalization terms by more general functionals of these two empirical measures.

\end{quote}




\setcounter{tocdepth}{3}


\setcounter{section}{0}


\section{Introduction and main results}\label{sec-Intro} 

\subsection{Background and main goals}\label{sec-background}
In random networks, one of the prominent problems is the question how to conduct a message through the system in an optimal way. Optimality is often measured in terms of determining the shortest path from the transmitter to the recipient, or, if interference is considered, determining the path that yields the least interference. If many messages are considered at the same time, an additional aspect of optimality may be to achieve a minimal amount of congestion. 

Many investigations concern the question just for one single transmitter/recipient pair, which is a question that every single participant faces. However, a strategy found in such a setting may lead to a selfish routeing, and it is quite likely that the totality of all these routeings 
for all the individuals is by far not optimal for the community of all the users \cite[Section 1]{CCS16}. 
Instead, the entire system may work even better if an optimal {\em compromise} is realized, by which we mean a joint strategy that leads to an optimum for the entire system, though possibly not for every participant. 

In this paper, we present a probabilistic ansatz for describing a jointly optimal routeing for an unbounded number of transmitter/recipient pairs, which takes into account the following three crucial properties of the family of message trajectories: entropy (i.e., counting complexity of the number of trajectory families satisfying certain properties), interference and congestion. That is, we consider a situation in which all the messages are directed through the system in a random way, such that each hop prefers a low interference, and such that the total amount of congestion is preferred to be low. Parameters control the strengths of influence of the three effects.

Let us describe our model in words. Let the users be located randomly as the sites of a Poisson point process, which we fix. Each user sends out precisely one message, which arrives at the (unique) base station, which is located at the origin. We consider the entire collection of possible trajectories of the messages through the system. We employ an ad-hoc relaying system with multiple hops, such that all the users act as relays for the handoffs of the messages. The maximal number of hops is $\kmax\in\N$ for each message. Each $k$-hop message trajectory (with $k\in\{1,\dots,\kmax\}$ itself random) is random and {\em a priori} uniformly distributed. The family of all trajectories is {\em a priori} independent. 

Now, the probability distribution of this family is given in terms of a Gibbs ansatz by introducing two exponential weight terms. (That is, we define a quenched measure on trajectories given the locations of the users.) The first one weights the total amount of interference, measured in terms of the signal-to-interference ratio for each hop. The second one weights the total congestion, i.e., the number of times that any two trajectories use the same relay. Under the arising measure,  there is a competition between all the three decisive effects of the trajectory family: entropy, interference and congestion. Furthermore, the users form a random environment for the family, which not only determines the starting sites of all the trajectories, but also has a decisive effect on interference and congestion. While the latter has a smoothing effect on the fine details of the spatial distribution of all the trajectories, the effect of the former is not so clear to predict, as the superposition of signals have a 
very non-local influence. 

Our main interest is in understanding the spatial distribution of the totality of all the message trajectories under the Gibbs distribution.
The measure under consideration is a highly complex object, as it depends on all the user locations and on many detailed properties and quantities. However, we make a substantial step towards a thorough understanding by deriving an asymptotic formula for the logarithmic behaviour of the normalization constant in the limit of a high spatial density of the users. The limiting situation is then described in terms of a large deviations rate function and a variational formula, whose minimizers describe the optimal joint choices of the trajectories. This formula is deterministic and depends only on general spatial considerations, not on the individual users. These are our main results in this paper.

The main objects in terms of which we achieve this description are the empirical measures of the  trajectories of the messages sent out by the users, disintegrated with respect to the lengths and rescaled to finite asymptotic size. These measures turn out to converge in the weak topology in the high-density limit that we consider in this paper. The counting complexity of the statistics of the message trajectories can be written in terms of multinomial expressions and afterwards, in the limit of finer and finer decompositions of the space, approximated in terms of relative entropies, using  Stirling's formula. The interference term can also be handled in a standard way \cite{HJKP15}, since it is a continuous function of the collection of empirical measures of message trajectories.

However, a key finding of our work is that the congestion term is a highly discontinuous function of these empirical measures. Indeed, one cannot express  its limiting behaviour in terms of these measures. Instead, one needs to substantially enlarge the probability space of trajectories and to introduce another collection of empirical measures, the ones of the locations of users (relays) who receive given numbers of incoming messages (counted with multiplicity if a message trajectory hits the same relay multiple times). The congestion expression then turns out to  be a lower semicontinuous function of these empirical measures, and hence the limiting congestion term is still expressible in terms of the weak limits of these measures. Again, using explicit combinatorics and Stirling's formula, we arrive at explicit entropic terms describing the statistics of these measures. These two families of empirical measures together enable us to describe all the properties of the message trajectories that we are 
interested in. We establish a full large deviation principle for the tuple of all 
these measures with an 
explicit rate function and obtain in particular their convergence towards the minimizer(s) of a characteristic variational formula. We also derive their positivity properties and characterize them in terms of Euler-Lagrange equations. Unfortunately, due to the complexity of the congestion term, we are not able to decide about the uniqueness of the minimizer.

Nevertheless, in the special case when congestion is not penalized, the minimizer turns out to be unique,  and we obtain an explicit expression that is amenable to further investigation.  In certain limiting regimes, we can derive a good understanding of decisive quantities of the system, like the typical number of hops, the typical length of a hop and the typical shape of a trajectory as a function of the distance between the transmission site and the origin. We expect that such properties of the system are similar if congestion is also penalized, as the effect of congestion is {\em a priori} not spatial, but combinatorial. We decided to analyse such questions in a separate work \cite{KT17b}, as they have a strongly analytic, rather than probabilistic, nature. The present paper includes a short summary of the results of \cite{KT17b}.

The main purpose of the present paper is to provide the mathematical framework of large deviations for the quenched trajectory distribution, given the user locations, in the high-density limit. We also provide a discussion of the relevance of the model for telecommunication theory. A connection of our work to traffic theory is outlined in \cite{KT17b}, which paper also includes numerical examples for the case when congestion is not penalized. Hence, in the present paper, we will formulate the model in a more general, slightly abstract, way in order to bring its mathematical essence to the surface. That is, we consider a random complete geometric graph in a compact subset of $\R^d$ (where the vertices are the users and the edges are the straight line segments between any two users), and a distribution of trajectories that has an interaction (the interference) with all the locations of the nodes and suppresses local clumping (the congestion).

\subsubsection{The high-density limit for multihop networks}\label{sec-Discussion-highdensity}
The quality of service in large multihop ad-hoc networks has received particular interest in the last years. In order to be able to derive a clear picture, one has to make a certain approximation in limiting settings. Two mathematical settings are frequently used: the high-density limit (sometimes called also a hydrodynamic limit or a mean-field limit), where the number of users in a compact fixed area diverges, and the thermodynamic limit, where the area diverges as well, such that the number of users per space unit remains fixed. The former models a situation like at concerts, demonstrations or sports events, while the latter one models large-area networks with moderate user density.

A number of papers on this subject use large deviations methods. This has several advantages. Indeed, the corresponding large deviation principles often come with a law of large numbers for certain empirical measures, and with exponential bounds on the probabilities of deviations from the limit. This suggests that the qualitative behaviour of the network is close to the limit already for relatively moderate values of the diverging parameter. These methods lead in the high-density setting to much more handy formulas (see e.g.~\cite{HJKP15,HJP16,HJ17}) than in the thermodynamic limit (see e.g.~\cite{HJKP15a}). This is why we decided to analyse our Gibbsian model in the high-density setting.

\subsubsection{Related literature}
Apart from the potential value for understanding a new type of message routeing models in telecommunication, the present paper provides also some interesting mathematical research on topological fine properties of random paths in random a environment in a high-density setting, a subject that received a lot of interest for various types of such processes over the last decades. We remind the reader on a number of investigations of the intersection properties of random walks and Brownian motions (both self-intersections and mutual intersections) in highly dense settings, see the monograph \cite{Ch09} and some particular investigations in \cite{KM02,KM13}; in all these works, one is interested in large deviations properties of suitable empirical measures, and the lack of continuity of the path properties is the main difficulty. Let us mention that the main aspect of the approach in \cite{KM02} is the same as in the present paper: an approximation of combinatorics in 
finer and finer decompositions of the space by entropic terms. Another line of research in which similar questions arise is a mean-field variant of a spatial version of Bose-Einstein statistics, like in \cite{AK08}, where the statistics of the empirical measures of  a diverging number of Brownian bridges with symmetrized initial-terminal condition is analysed in terms of a large deviation principle in the weak topology. While \cite{AK08} works with the same method as we in the present paper (spatial discretization with limiting fineness), \cite{T08} showed that a method based entirely on the notion of entropy is able to derive such results in a more general framework.

\subsubsection{Organization of the remainder of this paper.} We introduce the model and necessary notation in Section~\ref{sec-modeldef}, present our main results in Sections~\ref{sec-freeenergy} (the limiting free energy of the model), \ref{sec-minimizerstatement} (the description of the minimizer(s)), \ref{sec-LDP} (the large deviation principle and the convergence of the empirical measures), and~\ref{sec-beta=0description} (results in case congestion is not penalized), we interpret the minimizer(s) and summarize the results of \cite{KT17b} about their qualitative properties in Section~\ref{Discussion-minimizers}, and we discuss and comment our findings in Section~\ref{sec-Discussion}. The remaining sections are devoted to the proofs: in Section~\ref{sec-proofingredients} we prepare for the proofs by introducing our methods and deriving formulas for the probability terms, in Section~\ref{sec-prooflargelambda} we put all this together to a proof of the limiting free energy, the large deviation principle and 
the convergence of the empirical measures, in Section~\ref{sec-minimizers} we analyse the minimizer(s) of the characteristic 
variational formula, 
and in Section~\ref{sec-beta=0proof} we extend the proofs to the case when congestion is not penalized. 

\subsection{The Gibbsian model}\label{sec-modeldef}

We introduce now the mathematical setting. For any topological space $V$, let $\mathcal{M}(V)$ denote the set of all finite nonnegative Borel measures on $V$, which we equip with the weak topology. We are working in $\R^d$ with some fixed $d\in\N$. 
Our model is defined as follows. Let $W\subset \mathbb R^d$ be compact, the territory of our telecommunication system, containing the origin $o$ of $\R^d$.  

\subsubsection{Users}
Let $\mu \in \mathcal{M}(W)$ be an absolutely continuous measure on $W$ with $\mu(W)>0$. Note that we do not require that $\supp(\mu)=W$. For $\lambda>0$, we denote by $X^\lambda$ a Poisson point process in $W$ with intensity measure $\lambda\mu$. The points $X_i \in X^\lambda$ are interpreted 
as the {\em locations of the users} in the system, while the origin $o$ of $\mathbb R^d$ is the single \emph{base station}. We assume that $X^\lambda=\lbrace X_i \colon i \in I^\lambda \rbrace$ with $I^\lambda=\{1,\dots,N(\lambda)\}$ and $(N(\lambda))_{\lambda>0}$ a homogeneous Poisson process on $\N_0$ with intensity $\mathbb E[N(1)]=\mu(W)$, and $(X_i)_{i \in \mathbb N}$ is an i.i.d.~sequence of $W$-distributed random variables with distribution ${\mu(\cdot)}/{\mu(W)}$ defined on one probability space $(\Omega,\mathcal F,\mathbb P)$. Since $\mu$ has a density, all points $X_i$ are mutually different with probability one. Further, $X^\lambda$ is increasing in $\lambda$, and its {\em empirical measure}, normalized by $1/\lambda$,
\begin{equation}\label{userempirical}
L_\lambda=\frac{1}{\lambda} \sum_{i \in I^\lambda} \delta_{X_i} ,
\end{equation}
converges towards $\mu$ almost surely as $\lambda \to \infty$.

These assumptions on the users can be relaxed, see Section \ref{sec-poisson}.

\subsubsection{Message trajectories}\label{modeldef-messagetraj}

We now introduce the collection of trajectories sent out from the users to $o$, i.e., for uplink communication. For any $i \in I^\lambda$, we call a vector of the form
\[ S^i= (S_{-1}^i=K_i,S_0^i=X_i, S_1^i \in X^\lambda,\ldots,S_{K_i-1}^i \in X^\lambda, S_{K_i}^i =o)\in \bigcup_{k\in\N} \Big( \{ k \} \times \{ X_i \} \times W^{k-1} \times\{ o\} \Big) , \numberthis\label{trajectoryindices} \]
a {\em message trajectory} from $X_i$ to $o$ with $K_i$ hops. That is, $S^i$ starts from $X_i$ and ends in $o$ after $K_i$ hops from user to user in $X^\lambda$. Hence, each user sends exactly one message to $o$, and each user has the function of a relay. We fix a number $\kmax \in\N$ and write $\mathcal{S}_{\kmax}^i(X^\lambda)$ for the set of all possible realizations of  the random variable $S^i$ with $K_i\leq \kmax$, i.e., with no more than $\kmax$ hops. Hence, elements $s^i=(s^i_{-1},s^i_0,s^i_1,\ldots,s^i_{s^i_{-1}-1},s^i_{s^i_{-1}})$ of $\mathcal{S}_{\kmax}^i(X^\lambda)$ satisfy $s^i_{-1}\in\{1,\dots,\kmax\}$, $s^i_0=X_i$ and $s^i_{s^i_{-1}}=o$. We write $\mathcal{S}_{\kmax}(X^\lambda)=\bigtimes_{i\in I^\lambda} \mathcal{S}_{\kmax}^i(X^\lambda)$ for the set of all possible realizations of the families $S= (S^i)_{i\in I^\lambda}$. We use the notation $[k]=\lbrace 1,\ldots,k \rbrace$ for $k\in\N$. The assumption that we choose a finite upper bound $\kmax$ on the number of hops will be discussed 
in Section~\ref{sec-kmaxdiscussion}.

Given $i\in I^\lambda$, we consider each trajectory $S^i$ in \eqref{trajectoryindices} as an $\Scal_{\kmax}^i(X^\lambda)$-valued random variable. Its {\em a priori} measure is defined by the formula
\begin{equation}\label{referencemeasure}
s^i\mapsto \frac1{N(\lambda)^{s_{-1}^i-1}},\qquad s^i\in \Scal_{\kmax}^i(X^\lambda).
\end{equation}
That is, its restriction to $\{ s^i \in \mathcal{S}^i_{\kmax}(X^\lambda) \colon s^i_{-1}=k\}$ is the uniform distribution on the set of all $k$-hop trajectories from $X_i$ to $o$ for any $k \in [\kmax]$, and its total mass is equal to $\kmax$. Recall that it fixes the starting point $X_i$  and the terminal point $o$.

Under our joint {\em a priori} measure, all the trajectories are independent; indeed, it gives the value
\begin{equation}\label{referenceprod}
s= (s^i)_{i\in I^\lambda}\mapsto \prod_{i\in I^\lambda}  \frac1{N(\lambda)^{s_{-1}^i-1}}
\end{equation}
to the configuration $s \in \mathcal{S}_{\kmax}(X^\lambda)$. Thus, it gives a total mass of $\kmax^{N(\lambda)}$ to $\mathcal{S}_{\kmax}(X^\lambda)$.

\subsubsection{Gibbsian trajectory distribution}\label{sec-Gibbsdistribution}

In this section, we define the central object of this study: a Gibbs distribution on the set $\mathcal{S}_{\kmax}(X^\lambda)$ of collections of trajectories. After providing the abstract definitions, in Section \ref{sec-interference&congestion} we sketch the key example that is relevant for application in telecommunication. The general conditions on the ingredients of the Gibbs distribution in this section arise naturally from the properties of this example.

We introduce the following notation. For $k\in\N$, elements of the product space $W^k=W^{\{0,1,\dots,k-1\}}$ are denoted as $(x_0,\ldots,x_{k-1})$. For $l=0,\ldots,k-1$, the $l$-th marginal of a measure $\nu_k \in \mathcal{M}(W^k)$ is denoted by $\pi_l\nu_k\in \mathcal{M}(W)$, i.e., $\pi_l \nu_k(A)=\nu_k(W^{\{0,\dots,l-1\}}\times A\times W^{\{l+1,\dots,k-1\}})$ for any Borel set $A \subseteq W$.

For fixed $k \in [\kmax]$ and for a collection of trajectories $s \in \mathcal{S}_{\kmax}(X^\lambda)$, we define
\begin{equation}\label{Rdef}
R_{\lambda,k}(s)=\frac{1}{\lambda} \sum_{i \in I^\lambda\colon s^i_{-1}=k} \delta_{(s_0^i,\ldots,s_{k-1}^i)},
\end{equation}
the empirical measures of all the $k$-hop trajectories, which is an element of $\mathcal{M}(W^k)$. By the assumption that each user sends out exactly one message, we have \[ \sum_{k=1}^{\kmax} \pi_0 R_{\lambda,k}(s) =L_\lambda. \numberthis\label{idiscrete} \]
For $k \in [\kmax]$, we choose a continuous function $f_k \colon \mathcal M(W) \times W^{k} \to \R$ that is bounded from below. Using \eqref{idiscrete}, we put
\[ \mathfrak S(s)= 
\lambda \sum_{k=1}^{\kmax} \Big\langle R_{\lambda,k}(s)(\cdot),  f_k(L_\lambda,\cdot)\Big\rangle
=\sum_{k=1}^{\kmax} \sum_{i\in I^\lambda\colon s^i_{-1}=k}f_k(L_\lambda,s^i_0,\dots,s^i_{k-1}), \numberthis\label{SIRenergy} \] 
where we write $\langle \nu, f \rangle$ for the integral of the function $f$ against the measure $\nu$.
Moreover, we define
\begin{equation}\label{midef}
m_i(s)=\sum_{j \in I^\lambda} \sum_{l=1}^{s^j_{-1}-1} \mathds 1 \lbrace s^j_l=s^i_0 \rbrace, \quad i\in I^\lambda,
\end{equation}
as the number of incoming hops into the user (relay) $s^i_0=X_i$ of any of the trajectories. 

We pick a function $\eta \colon \N_0 \to \R$, bounded from below such that $\lim_{m \to \infty} \eta(m)/m = \infty$. Then we put
\[ \mathfrak M (s) = \sum_{i \in I^\lambda} \eta(m_i(s)). 
\numberthis\label{Menergy} \]
Now, we define
\begin{equation}\label{Gibbsdistribution}
\mathrm P^{\gamma,\beta}_{\lambda,X^\lambda}(s) := \frac1{Z^{\gamma,\beta}_\lambda(X^\lambda)}
\Big(\prod_{i\in I^\lambda}  \frac1{N(\lambda)^{s_{-1}^i-1}}\Big) \exp\Big\{-\gamma \mathfrak S(s)-\beta\mathfrak M(s)\Big\},
\end{equation}
where $\gamma>0$ and $\beta>0$ are parameters. This is the Gibbs distribution with independent reference measure given in \eqref{referenceprod}, subject to two exponential weights with the terms \eqref{SIRenergy} and \eqref{Menergy}. Here
\[ Z^{\gamma,\beta}_\lambda(X^\lambda)=\sum_{r \in \mathcal{S}_{\kmax}(X^\lambda)}  \Big(\prod_{i\in I^\lambda}\frac1{N(\lambda)^{ r^i_{-1}-1}}\Big) \exp\Big\{-\gamma \mathfrak S(r)-\beta\mathfrak M(r)\Big\} \numberthis\label{zed} \]
is the normalizing constant, which we will refer to as \emph{partition function}. Note that $\mathrm P^{\gamma,\beta}_{\lambda, X^\lambda}(\cdot)$ is random conditional on $X^\lambda$, and it is a probability measure on $\mathcal{S}_{\kmax}(X^\lambda)$. In the jargon of statistical mechanics, it is a {\em quenched} measure, which we will consider almost surely with respect to the process $(X^\lambda)_{\lambda>0}$. In the {\em annealed} setting, one would average out over $(X^\lambda)_{\lambda>0}$, see Section \ref{sec-annealed}.

\subsubsection{The key example: penalization of interference and congestion} \label{sec-interference&congestion}
In this section, we sketch the most important example for the exponents $\mathfrak S$ and $\mathfrak M$ in \eqref{Gibbsdistribution}, where $\mathfrak S$ registers interference and $\mathfrak M$ expresses congestion in a telecommunication network. Analysing the qualitative properties of the network with these choices of $\mathfrak M$ and $\mathfrak S$ is the main topic of our accompanying paper \cite{KT17b}, see Section~\ref{sec-furtherresults}.

Now we introduce interference. We choose a \emph{path-loss function}, which describes the propagation of signal strength over distance. This is a monotone decreasing, continuous function $\ell\colon [0,\infty) \to (0,\infty)$. 
An example used in practice is $\ell(r)=\min \lbrace 1, r^{-\alpha} \rbrace$, for some $\alpha>0$, see e.g.~\cite[Section II.]{GT08}, for further examples see \cite[Section 2.3.1]{BB09}. The \emph{signal-to-interference ratio (SIR)} of a transmission from $X_i \in X^\lambda$ to $x \in W$ in the presence of the users in $X^\lambda$ is given as
\[ \SIR(X_i,x,L_\lambda)= \frac{\ell(|X_i-x|)}{\frac{1}{\lambda} \sum_{j \in I^\lambda} \ell(|X_j-x|)}. \numberthis\label{SIR} \]
We call the denominator of the r.h.s\ of \eqref{SIR} the \emph{interference} at $x$. The definition \eqref{SIR} comes from \cite{HJKP15}. It is adapted to the high-density setting, and it differs from the usual definition \cite{GK00} of $\SIR$ in the following way. The sum in the interference in \eqref{SIR} is multiplied by $1/\lambda$, and it contains also the term $\ell(|X_i-x|)$. For a justification of these differences, see \cite[Section 6.1.1]{KT17b}. 

Now, given a trajectory configuration $s=(s^i)_{i \in I^\lambda}\in \Scal_{\kmax}(X^\lambda)$, we put 
\begin{equation} \label{realSIRenergy}
\mathfrak S(s)=\sum_{i\in I^\lambda}\sum_{l=1}^{s^i_{-1}} \SIR(s^i_{l-1},s^i_l,L_\lambda)^{-1}=\lambda \sum_{k=1}^{\kmax}  \int_{W^k} R_{\lambda,k}(s)(\d x_0,\ldots,\d x_{k-1}) \sum_{l=1}^k\frac{\int_W \ell(|y-x_l|) L_\lambda(\d y)}{\ell(|x_{l-1}-x_l|)},
\end{equation}
where for $k \in [\kmax]$, we write $x_k=o$. Then \eqref{realSIRenergy} is a special case of \eqref{SIRenergy} with \[ f_k(\nu,x_0,\ldots,x_{k-1}) = \sum_{l=1}^k \frac{\int_{W} \ell(|y-x_l|)\nu(\d y)}{\ell(|x_{l-1}-x_l|)}, \quad x_k=o, k \in [\kmax]. \numberthis\label{fkspecial} \]

Next, we introduce congestion. We define $\eta(m)=m(m-1)$, and, as in \eqref{Menergy}, we put
\[ \mathfrak M(s) = \sum_{i \in I^\lambda} \eta(m_i(s)) = \sum_{i \in I^\lambda} m_i(s)(m_i(s)-1), \quad s \in \Scal_{\kmax}(X^\lambda). \numberthis\label{realMenergy} \]
Note that $\eta(m_i(s))=m_i(s)(m_i(s)-1)$ is the number of ordered pairs of hops arriving at the relay $X_i=s^i_0$. 
We will explain and motivate these choices in Section~\ref{sec-SIRpenaltydiscussion1}. 

In the downlink scenario, instead of users sending messages to the base station, the base station sends exactly one message to each of the users, using the same relaying rules. One can define a Gibbsian model analogously, now for trajectories from $o$ to $X_i$ instead of the other way around.  For this, the interference term and the congestion term have to be redefined in an obvious way. We are certain that analogues of all our results are true and can be proved in the same way, hence we abstained from spelling them out.

For possible extensions of this model involving time dependence or users transmitting multiple messages, see Section~\ref{sec-modelextensions}.

\subsection{The limiting free energy}\label{sec-freeenergy}

The main goal of this paper is the description of this model in the limit $\lambda\to\infty$ in the quenched setting. Our first result describes the limiting {\em free energy}, i.e., the exponential behaviour of the partition function $Z^{\gamma,\beta}_\lambda(X^\lambda)$. One expects that this is entirely governed by the large deviations behaviour of the empirical measures $((R_{\lambda,k})(S)_{k \in [\kmax]})_{\lambda>0}$. This expectation is supported by the fact that, for $i \in I^\lambda$ and $s \in \Scal_{\kmax}(X^\lambda)$, we can express $m_i(s)$ defined in \eqref{midef} in terms of $(R_{\lambda,k}(s))_{k \in [\kmax]}$ as follows
\[ m_i(s) = \lambda \sum_{k=1}^{\kmax} \sum_{l=1}^{\kmax} \pi_{l} R_{\lambda,k}(s) (\{ s^i_0 \}). \numberthis\label{miR} \]

Surprisingly, it turns out that the limiting free energy cannot be described entirely in terms of these measures. The reason is that the function in \eqref{miR} that maps them onto $m_i(s)$ is highly discontinuous in the limit $\lambda \to \infty$; even a proper formulation of such continuity would be awkward since both $i$ and $s$ depend on $\lambda$.

One therefore needs to substantially extend the probability space and to choose an additional family of empirical measures such that the congestion term $\mathfrak M(s)$ can be written as a (lower semi-)continuous function of these measures in the limit $\lambda \to \infty$. A natural choice of such a family is the one of the measures
\begin{equation}\label{Pdef}
P_{\lambda,m} (s)=\frac{1}{\lambda} \sum_{i \in I^\lambda\colon m_i(s)=m} \delta_{s^i_0}, \quad m \in \N_0.
\end{equation}
Then for $m \in \N_0$, $P_{\lambda,m} (s) \in \Mcal(W)$ is the empirical measure of the users $s^i_0$ whose number of incoming hops equals $m$. For any $s \in \Scal_{\kmax}(X^\lambda)$ the following hold
\begin{equation}\label{discreteconstraints}
(i)\quad \sum_{k=1}^{\kmax} \pi_0 R_{\lambda,k}(s) = L_\lambda,\qquad (ii)\quad \sum_{m=0}^{\infty} P_{\lambda,m}(s) = L_\lambda,\qquad(iii) \quad \sum_{m=0}^{\infty} m P_{\lambda,m}(s) = \sum_{k=1}^{\kmax} \sum_{l=1}^{k-1} \pi_l R_{\lambda,k}(s).
\end{equation}
Condition (i) expresses our assumption that each user transmits precisely one message, (ii) says that each user serves as a relay for precisely $m$ message trajectories for precisely one $m\in\N_0$, and (iii) says that the relays can be calculated in two ways: according to the number of incoming hops and according to the index of the hop of a trajectory that uses it. Moreover, we can write \eqref{Menergy} in terms of $(P_{\lambda,m}(s))_{m\in\N_0}$ as follows
\[ \mathfrak M(s)= \sum_{i \in I^\lambda} \eta(m_i(s)) = \lambda \sum_{m=0}^{\infty} \eta(m) P_{\lambda,m}(s)(W). \] 
We note that the function $\mathcal M(W)^{\N_0} \to \R\cup \{\infty\}$, $(\xi_m)_{m\in\N_0} \mapsto \sum_{m=0}^{\infty} \eta(m) \xi_m(W)$ is lower semicontinuous, and even continuous on $\{ (\xi_m)_{m\in \N_0} \colon \sum_{m=0}^{\infty} \eta(m) \xi_m(W) \leq \alpha \}$ for any $\alpha \in \R$. 

The limiting free energy will be described in terms of the following kind of families of measures, and it will turn out that all subsequential limits of the families $((R_{\lambda,k}(S))_{k=1}^{\kmax},(P_{\lambda,m}(S))_{m=0}^{\infty})$ in the quenched limit $\lambda \to \infty$ are of this kind.

\begin{defn}\label{def-admtrajet}
An \emph{admissible trajectory setting} is a collection of measures $\Psi=((\nu_k)_{k=1}^{\kmax},(\mu_m)_{m=0}^{\infty})$ with $\nu_k \in \mathcal{M}(W^k)$ for all $k$ and $\mu_m \in \mathcal{M}(W)$ for all $m$, satisfying the following properties.
\begin{equation}\label{constraints}
(i)\quad \sum_{k=1}^{\kmax} \pi_0 \nu_k = \mu,\qquad (ii)\quad \sum_{m=0}^{\infty} \mu_m = \mu,\qquad(iii) \quad M:=\sum_{m=0}^{\infty} m \mu_m = \sum_{k=1}^{\kmax} \sum_{l=1}^{k-1} \pi_l \nu_k.
\end{equation}
\end{defn}
The measure $\nu_k$ is the measure of the $k$-hop trajectories and $\mu_m$ the measure of the users that receive precisely $m$ incoming hops; note that there is no reason that they be normalized (like for $\mu$). Observe that in \eqref{discreteconstraints}, $L_\lambda$, $R_{\lambda,k}(s)$ and $P_{\lambda,m}(s)$ play the role of $\mu$, $\nu_k$ and $\mu_m$, respectively. In particular, after replacing $\mu$ by $L_\lambda$, $((R_{\lambda,k}(s))_{k \in [\kmax]},(P_{\lambda,m}(s))_{m\in\N_0})$ satisfies the definition of an admissible trajectory setting. See Section~\ref{sec-Discussion} for more explanations and interpretations, moreover for a modified version of our model where the assumption (i) is relaxed.
By
\[
\mathcal H_V(\nu\mid \widetilde \nu)=\begin{cases}
\int_V \d \nu\,\log\frac{\d \nu}{\d \widetilde \nu}\,-\nu(V)+\widetilde \nu(V) ,&\text{if the density }\frac{\d \nu}{\d \widetilde \nu}\text{ exists,}\\
+\infty&\text{otherwise,}
\end{cases} \numberthis\label{relativeentropy}
\]
we denote the relative entropy \cite[Section 2.3]{GZ93} of a Borel measure $\nu$ with respect to another Borel measure $\widetilde\nu$ on a measurable subset $V$ of $\R^n$, $n \in \N$.

For an admissible trajectory setting $\Psi=((\nu_k)_{k=1}^{\kmax}, (\mu_m)_{m=0}^{\infty})$ we define
\begin{equation} \label{limitSIR}
\mathrm S(\Psi)=\sum_{k=1}^{\kmax} \int_{W^k} \d \nu_k\, \widetilde f_k,\qquad \text{where }\widetilde f_k( x_0,\ldots,x_{k-1})=f_k(\mu,x_0,\ldots,x_{k-1}),
\end{equation}
and
\begin{equation} \label{limitbottleneck}
\mathrm M(\Psi)=\sum_{m=0}^{\infty} \eta(m)\, \mu_m(W)
\end{equation}
and
\begin{equation}\label{quenchedentropy} 
 \mathrm I(\Psi)=\sum_{k=1}^{\kmax} \mathcal H_{W^k}\big(\nu_k \mid \mu \tensor M^{\tensor (k-1)}\big)+ \sum_{m=0}^{\infty} \mathcal H_{W}(\mu_m \mid \mu c_m)+\mu(W) \Big( 1-\sum_{k=1}^{\kmax} M(W)^{k-1}\Big)-\frac1\e ,
 \end{equation}
 where we recall $M=\sum_{m\in\N_0}m \mu_m$ from Definition~\ref{def-admtrajet}(iii), $\eta$ defined before \eqref{Menergy}, and $c_m=\exp(-1/(\e\mu(W))(\e\mu(W))^{-m}/m! $ are the weights of the Poisson distribution with parameter $1/(\e\mu(W))$. In Section~\ref{sec-Discussion-varformula}, we argue that $\mathrm I(\Psi)$ is well-defined as an element of $(-\infty,\infty]$ and $\Psi\mapsto \mathrm I(\Psi)$ is a lower semicontinuous function that is bounded from below, and we provide an interpretation for $\mathrm I(\cdot)$. A tedious but elementary calculation shows that $\mathrm I$ is convex. In Section~\ref{sec-LDP}, $\mathrm I$ will turn out to govern the large deviations of the trajectory configuration. The terms $\mathrm S(\cdot)$ and $\mathrm M(\cdot)$ are analogues of the penalty terms $\mathfrak S(\cdot)$ and $\mathfrak M(\cdot)$ in the high-density setting, respectively.
 
We fix all the parameters $W, \mu, \kmax, f_k, \eta, \gamma$ and $\beta$ of the model. Our first main result is the following. 

\begin{theorem}[Quenched exponential rate of the partition function] \label{theorem-variation1} For $\P$-almost all $\omega \in \Omega$,
\[ \lim_{\lambda \to \infty} \frac{1}{\lambda} \log Z^{\gamma,\beta}_\lambda(X^\lambda(\omega)) = -\inf_{\Psi \text{ admissible trajectory setting}} \big(\mathrm I(\Psi)+\gamma \mathrm S(\Psi)+\beta \mathrm M(\Psi)\big). \numberthis\label{variation} \]
\end{theorem}

See Section~\ref{sec-Discussion} for a discussion and Section~\ref{sec-varproof} for the proof. 

\subsection{Description of the minimizers}\label{sec-minimizerstatement} 

From the variational formula in \eqref{variation}, descriptive information about the typical behaviour of the network can be deduced, especially in the case of the specific choice of $\mathfrak M$ and $\mathfrak S$ from Section~\ref{sec-interference&congestion}, see Sections~\ref{sec-LDP}, \ref{Discussion-minimizers} and \ref{sec-Discussion}. Hence, it is important to derive the Euler-Lagrange equations and to characterize the minimizers in most explicit terms. Our main results in this respect are the following. Note that the case $\kmax=1$ is trivial.

\begin{prop}[Characterization of the minimizer(s)] \label{prop-minimizer}
Let $\kmax>1$. The infimum in the variational formula in \eqref{variation} is attained, and every minimizer $\Psi=((\nu_k)_{k=1}^{\kmax},(\mu_m)_{m=0}^{\infty})$ has the following form.
\begin{eqnarray}
 \nu_k(\d x_0,\ldots, \d x_{k-1} ) &=& \mu(\d x_0)\, A(x_0)\prod_{l=1}^{k-1} \big(C(x_{l})M(\d x_l)\big) \e^{-\gamma \widetilde f_k(x_0,\ldots,x_{k-1})},\qquad k \in [\kmax], \label{forshow}\\
 \mu_m(\d x) &=& \mu(\d x)\, B(x)\frac{(C(x)\mu(W))^{-m} }{m!}\e^{-\beta \eta(m)},\qquad m\in\N_0,\label{forshowmu}
\end{eqnarray}
where $A,B, C \colon W\to [0,\infty)$ are functions such that the conditions in \eqref{constraints} are satisfied.
\end{prop}

The proof of Proposition~\ref{prop-minimizer} is in Section~\ref{sec-minimizers}.

While explicit formulas for the functions $A$ and $B$ can, given the function $C$, easily be derived from (i) and (ii) in \eqref{constraints} (see \eqref{ABident}), the condition for $C$ coming from (iii) is deeply involved and cannot be easily solved intrinsically; see \eqref{Ccondition} -- \eqref{Fdefinition}. We have no argument for its existence to offer other than via proving the existence of a minimizer $\Psi$ and deriving the Euler-Lagrange equations. By convexity of $\mathrm I$, $\mathrm S$ and $\mathrm M$, every solution $\Psi$ to these equations is a minimizer. Even the uniqueness of $C$ is unknown to us. We will interpret the equations \eqref{nukminimizer}--\eqref{mumminimizer} in Section~\ref{sec-Discussion-minimizers}. The equations \eqref{forshow}--\eqref{forshowmu} become more explicit in case $\beta=0$, and in this case, uniqueness of the minimizer holds, see Section \ref{sec-beta=0description}.

In case $\kmax=1$, the only admissible trajectory setting is $\Psi=(\nu_1,(\mu_m)_{m\in\N_0})$ with $\mu_0=\nu_1=\mu$ and $\mu_m=0$ otherwise, therefore this $\Psi$ minimizes \eqref{variation}.

\subsection{Large deviations for the empirical trajectory measure}\label{sec-LDP}

Actually, the minimizers of the variational formula in \eqref{variation} receive a rigorous interpretation in terms of important objects that describe the network. As we have already mentioned, the family of empirical measures
\begin{equation}\label{Psilambdadef}
\Psi_\lambda(s)= \big((R_{\lambda,k}(s))_{k \in [\kmax]}, (P_{\lambda,m} (s))_{m\in\N_0}\big)
\end{equation} 
satisfies the definition of an admissible trajectory setting, apart from the fact that in Definition \ref{def-admtrajet}, $\mu$ has to replaced by $L_\lambda$ everywhere, where we recall that $L_\lambda$ converges to $\mu$ almost surely as $\lambda \to \infty$. According to our remarks after Definition~\ref{def-admtrajet}, $R_{\lambda,k}(s)$ and $P_{\lambda,m} (s)$ play the roles of $\nu_k$ and $\mu_m$, respectively, in an admissible trajectory setting, which explains this term. Furthermore, for $s \in \mathcal{S}_{\kmax}(X^\lambda)$, we can express the term $\mathfrak M$ as
$$
\mathfrak M(s)=\lambda \mathrm M(\Psi_\lambda(s)).
$$
Moreover, for the continuous penalization term we have
\begin{equation}\label{SIRapprox}
\mathfrak S(s) \approx \lambda {\rm S}(\Psi_\lambda(s)),
\end{equation}
where we typically do not have an identity, because $\mathfrak S(s) = \lambda \sum_{k=1}^{\kmax} \int_{W^k} \d R_{\lambda,k}(s) f_k(L_\lambda,\cdot)$, which is usually not equal to $\lambda {\rm S}(\Psi_\lambda(s))=\lambda \sum_{k=1}^{\kmax} \int_{W^k} \d R_{\lambda,k}(s) f_k(\mu,\cdot)$. However, since $L_\lambda \Longrightarrow \mu$ almost surely, this difference vanishes in the limit, see Proposition \ref{prop-SIRcomputations}.

We consider now the distribution of $\Psi_\lambda(S)$ with $S$ distributed under the product reference measure introduced in \eqref{referenceprod}, normalized to a probability measure, ${\rm P}_{\lambda,X^\lambda}^{0,0}$; note that the normalization $Z_{\lambda}^{0,0}(X^\lambda)$ is equal to $\kmax^{N(\lambda)}$. Our next main result, Theorem \ref{thm-LDP}, is a large deviation principle (LDP; see \eqref{upper}--\eqref{lower}) and the convergence towards the minimizers of the variational formula.

\begin{theorem}[LDP and convergence for the empirical measures]\label{thm-LDP}
The following statements hold almost surely with respect to $(X^\lambda)_{\lambda>0}$.
\begin{enumerate}
 \item[(i)]
The distribution of $\Psi_\lambda(S)$ under ${\rm P}_{\lambda,X^\lambda}^{0,0}$ satisfies an LDP as $\lambda\to\infty$ with scale $\lambda$ on the set 
\begin{equation}\label{Adef}
 \Acal=\Big( \prod_{k=1}^{\kmax} \mathcal{M}(W^k) \Big) \times \mathcal{M}(W)^{\mathbb N_0}
\end{equation}
with rate function given by $\Acal\ni \Psi\mapsto \mathrm I(\Psi)+\mu(W)\log \kmax$, which we define as $\infty$ if $\Psi$ is not an admissible trajectory setting.

\item[(ii)] For any $\gamma,\beta\in(0,\infty)$, the distribution of $\Psi_\lambda(S)$ under ${\rm P}_{\lambda,X^\lambda}^{\gamma,\beta}$ converges towards the set of minimizers of the variational formula in \eqref{variation}.
\end{enumerate}
\end{theorem}
For the reader's convenience, we recall that the LDP states that the rate function $\mathrm I+\mu(W)\log \kmax$ is lower semicontinuous and 
\begin{eqnarray} 
\limsup_{\lambda\to\infty}\frac 1\lambda\log {\rm P}_{\lambda,X^\lambda}^{0,0}(\Psi_\lambda(S)\in F)&\leq& -\inf_F \big( {\rm I}+\mu(W)\log \kmax \big),\label{upper}\\
\liminf_{\lambda\to\infty}\frac 1\lambda\log {\rm P}_{\lambda,X^\lambda}^{0,0}(\Psi_\lambda(S)\in G)&\geq& -\inf_G \big( {\rm I}+\mu(W)\log \kmax \big),\label{lower}
\end{eqnarray}
for any closed set $F$ and any open set $G$ in $\Acal$. See \cite{DZ98} for more on large deviations theory.  On $\Acal$, we consider the product topology that is induced by weak convergence in each factor; this is equal to coordinatewise weak convergence, see Section \ref{sec-admtrajetstandardsetting} for more details. Convergence of a distribution towards a set is defined by requiring that for any neighbourhood of the set, the probability of not being in the neighbourhood vanishes. 

The proof of Theorem~\ref{thm-LDP}(i) is carried out in Section \ref{sec-LDPproof}, using Lemma~\ref{lemma-existence}. Assertion (ii) is a simple consequence of (i), since the functionals $\mathrm S$ and $\mathrm M$ are bounded and continuous on the set $B_C=\{\Psi\in\Acal\colon \mathrm M(\Psi)\leq C  \}$ for any $C$, and $B_C$ is compact in $\Acal$ (see Lemma \ref{lemma-existence}). Denoting the level sets of the rate function $\mathrm I+\mu\log\kmax$ by $\Phi_\alpha=\{ \Psi\in\Acal\colon \mathrm I(\Psi)+\mu(W)\log \kmax \leq \alpha \}$ for $\alpha \in \mathbb R$, we see that $\Phi_\alpha \cap B_C$ is compact for any $\alpha$ and $C$. Thus, Varadhan's lemma can be applied to prove Assertion (ii).

\subsection{Dropping the congestion term}\label{sec-beta=0description}
Proposition~\ref{prop-minimizer} yields a rather implicit description of the minimizers of \eqref{variation} in the case $\beta,\gamma>0$. The cardinality of the set of minimizers is also unclear. However, in the special case $\beta=0$, where the congestion functional $\mathfrak M$ \eqref{Menergy} is absent, the situation is much better. Indeed, it turns out that the minimizer is unique and is explicitly given in terms of the parameters of the model. Especially for the specific choice of Section \ref{sec-interference&congestion} where $\mathfrak S$ penalizes interference \eqref{realSIRenergy}, on base of this knowledge, we are able in \cite{KT17b} to derive a number of relevant qualitative properties of the trajectories, see Section~\ref{sec-furtherresults} for a summary.

In what follows, we call $\Sigma=(\nu_k)_{k \in [\kmax]}$ with $\nu_k \in \Mcal(W^k)$ for all $k\in[\kmax]$ an \emph{asymptotic routeing strategy} if we have $\sum_{k=1}^{\kmax} \pi_0 \nu_k=\mu$. In \eqref{constraints} we see that the first coordinate, $\Sigma$, of an admissible trajectory setting $\Psi$ is an asymptotic routeing strategy, and in \eqref{limitSIR} we see that $\mathrm S(\Psi)$ depends only on $\Sigma$. We will therefore write $\mathrm S(\Sigma)$ for $\mathrm S(\Psi)$. Further, we write $M=\sum_{k=1}^{\kmax} \sum_{l=1}^{k-1} \pi_l \nu_k$, in accordance with \eqref{constraints} but with no regard to the measures $(\mu_m)_{m\in\N_0}$. We define an entropic term $\mathrm J$ for asymptotic routeing strategies as follows.
\[ \mathrm J(\Sigma) = \sum_{k=1}^{\kmax} \Hcal_{W^k} (\nu_k \mid \mu^{\tensor k}) - \sum_{k=2}^{\kmax} \mu(W)^k + M(W) \log \mu(W). \numberthis\label{Jentropy} 
\]
Similarly to $\mathrm I$ in \eqref{quenchedentropy}, $\mathrm J$ describes counting complexity in the high-density limit, but without reference to the measures $(\mu_m)_m$. 

The following proposition summarizes the analogues of Theorem~\ref{theorem-variation1}, Proposition~\ref{prop-minimizer} and Theorem~\ref{thm-LDP} in case $\beta=0$, after dropping all the measures $\mu_m$.

\begin{prop} \label{prop-variationbeta=0}The following statements hold almost surely with respect to $(X^\lambda)_{\lambda>0}$.
\begin{enumerate}
 \item\label{LDP-variationbeta=0}
The distribution of \[ \Sigma_\lambda(S)=(R_{\lambda,k}(S))_{k\in[\kmax]} \numberthis\label{Sigmalambdadef} \] under ${\rm P}_{\lambda,X^\lambda}^{0,0}$ satisfies an LDP as $\lambda\to\infty$ with scale $\lambda$ on the set 
$
 \Acal_0=\prod_{k=1}^{\kmax} \mathcal{M}(W^k) 
$
with rate function given by $\Acal_0\ni \Psi\mapsto \mathrm J(\Sigma)+\mu(W)\log \kmax$, which we define as $\infty$ if $\Sigma$ is not an asymptotic routeing strategy. Further, the rate function has compact level sets.

 \item\label{freeenergy-variationbeta=0} For any $\gamma\in(0,\infty)$,
\begin{align}
 \lim_{\lambda\to\infty}\frac{1}{\lambda}\log Z^{\gamma,0}_\lambda(X^\lambda)
&=-\inf_{\Sigma\text{ asymptotic routeing strategy}} \big(\mathrm J(\Sigma)+\gamma \mathrm S(\Sigma)\big) \label{beta0variation} 
\end{align}
\item\label{minimizer-variationbeta=0} Let $\gamma>0$ and $\kmax>1$. The variational formula on the r.h.s. of \eqref{beta0variation} exhibits a unique minimizer $\Sigma=(\nu_k)_{k \in [\kmax]}$ given as 
\[ \nu_k(\d x_0,\ldots,\d x_{k-1} ) = \mu(\d x_0) A(x_0) \prod_{l=1}^{k-1} \frac{\mu(\d x_l)}{\mu(W)} \e^{-\gamma \widetilde f_k(x_0,\ldots,x_{k-1})}, \quad k \in [\kmax], \numberthis\label{nukminimizerbeta=0} \]
where
\[ \frac1{A(x_0)}=\sum_{k=1}^{\kmax} \int_{W^{k-1}} \prod_{l=1}^{k-1} \frac{\mu(\d x_l)}{\mu(W)} \e^{-\gamma \widetilde f_k(x_0,\ldots,x_{k-1})}, \qquad x_0 \in W. \numberthis\label{Adefnew} \]

\item\label{convergence-variationbeta=0} For any $\gamma\in(0,\infty)$, the distribution of $\Sigma_\lambda(S)$ under ${\rm P}_{\lambda,X^\lambda}^{\gamma,0}$ converges to the minimizer of the variational formula in \eqref{beta0variation}.
\end{enumerate}
\end{prop}
Proposition \ref{prop-variationbeta=0} is proved in Section~\ref{sec-beta=0proof}. An interpretation of the equations \eqref{nukminimizerbeta=0}--\eqref{Adefnew} can be found in Section~\ref{sec-Discussion-minimizers}.

Let us explain in what way Proposition \ref{prop-variationbeta=0} is the special case of the aforementioned results for $\beta=0$ and in what way it differs. It is true that the LDP in Assertion (1) directly follows from the LDP in Theorem~\ref{thm-LDP}(i) via the contraction principle \cite[Theorem 4.2.1]{DZ98} for the projection map $(\Sigma,(\mu_m)_m) \mapsto \Sigma$, however, with rate function given by
\[ \mathrm J(\Sigma)=\inf_{(\mu_m)_{m\in\N_0} \colon \Psi=(\Sigma,(\mu_m)_{m\in\N_0}) \text{ admissible trajectory setting}} \mathrm I(\Psi). \numberthis\label{contractionprinciple} \]
It is an elementary but tedious task to identify this as in \eqref{Jentropy} by identifying
\[ \mu_m (\d x) = \mu(\d x) \frac{\Big(\frac{M(\d x)}{\mu(\d x)}\Big)^{m}}{m!} \e^{-M(\d x)/\mu(\d x)}, \qquad m\in\N_0,\numberthis\label{mumPoisson} \] 
as the unique minimizer on the right-hand side of \eqref{contractionprinciple}, given $M=\sum_{k=1}^{\kmax} \sum_{l=1}^{k-1} \pi_l \nu_k$. However, we chose an alternative route for proving the LDP with explicit identification of $\mathrm J$, which is a variant of the proof of Theorem~\ref{thm-LDP}(i). From \eqref{contractionprinciple} it is clear that the variational formula in \eqref{beta0variation} is indeed the special case of \eqref{variation} for $\beta=0$, i.e.,
\begin{align}\label{beta0variation2}
\inf_{\Sigma\text{ asymptotic routeing strategy}} \big(\mathrm J(\Sigma)+\gamma \mathrm S(\Sigma)\big)
=\inf_{\Psi\text{ admissible trajectory setting}} \big(\mathrm I(\Psi)+\gamma \mathrm S(\Psi)\big).
\end{align}
Note also that the minimizer $\Psi$ is unique. This raises the additional question whether or not the measures $(P_{\lambda,m}(S))_{m\in\N_0}$ converge to the minimizer $(\mu_m)_{m\in\N_0}$ in \eqref{mumPoisson} under ${\rm P}_{\lambda,X^\lambda}^{\gamma,0}$ for $M$ corresponding to the minimal $\nu_k$'s of \eqref{nukminimizerbeta=0}. Since the congestion term, which gave rise to a strong compactness argument, is now absent, this question cannot immediately be decided using large deviations arguments, but we nevertheless believe it is true. Moreover, this compactness property was also used in the proof of Theorem~\ref{theorem-variation1}, which is another reason that we had to redo the proofs of Proposition \ref{prop-variationbeta=0}(2) and (3), given our proof of Assertion (1).

\subsection{Interpretation and qualitative properties of the minimizer(s)}\label{Discussion-minimizers}

\subsubsection{Interpretation of the minimizer(s)} \label{sec-Discussion-minimizers}

In case $\beta,\gamma>0$, Proposition~\ref{prop-minimizer} tells us quite some information about the limiting trajectory distribution and the limiting spatial distribution of users with a given number of incoming hops under the measure ${\rm P}^{\gamma,\beta}_{\lambda,X^\lambda}$. Indeed, both have densities that are $\mu^{\otimes k}$-almost everywhere positive. It is remarkable that the $k$-hop trajectories follow a distribution that comes from choosing independently all the $k$ sites with measures that do not depend on $k$ (the starting point according to $A(x)\,\mu(\d x)$ and all the other $k-1$ sites each according to $C(x)\, M(\d x)$), exponentially weighted with the term $\gamma \widetilde f_k$. Furthermore, all the measures of the users receiving $m$ incoming hops superpose each other on the full set $\supp(\mu)$, and at each space point $x$, this number $m$ is distributed according to some Poisson distribution, exponentially weighted with the term $\beta \eta(m)$.

As for the case $\beta=0,\gamma>0$, we have a unique minimizer, which exhibits all the properties enumerated for $\beta,\gamma>0$. In the $k$-hop trajectories, the starting point is chosen according to $A(x)\,\mu(\d x)$ and all the other $k-1$ sites according to the measure $\mu(\d x)/\mu(W)$, weighted with $\gamma \widetilde f_k$. Moreover, the number of incoming hops at a given relay at the site $x\in W$ is Poisson distributed with parameter equal to $M(\d x)/\mu(\d x)$.

\subsubsection{Qualitative properties of the minimizer}\label{sec-furtherresults}

In our accompanying paper \cite{KT17b}, we analyse the joint routeing behaviour of our Gibbsian system in the high-density limit. We investigate qualitative properties of the minimizer of the variational formula in \eqref{variation}, such as the typical number of hops, the typical length of a hop and the typical shape of a trajectory, in case $\gamma>\beta=0$ and the interference term is chosen as in Section~\ref{sec-interference&congestion}. In this case, the minimizer is unique and has the form (cf.~\eqref{nukminimizerbeta=0})
\[ \nu_k(\d x_0,\ldots,\d x_{k-1} ) = \mu(\d x_0) A(x_0) \prod_{l=1}^{k-1} \frac{\mu(\d x_l)}{\mu(W)} \e^{-\gamma \sum_{l=1}^{k} \frac{\int_W \ell(|y-x_l|) \mu(\d y)}{\ell(|x_{l-1}-x_l|)}}, \quad k \in [\kmax], \numberthis\label{nukminimizerbeta=0special} \]
where
\[ \frac1{A(x_0)}=\sum_{k=1}^{\kmax} \int_{W^{k-1}} \prod_{l=1}^{k-1} \frac{\mu(\d x_l)}{\mu(W)} \e^{-\gamma \sum_{l=1}^{k} \frac{\int_W \ell(|y-x_l|) \mu(\d y)}{\ell(|x_{l-1}-x_l|)}}, \qquad x_0 \in W. \numberthis\label{Adefnewspecial} \]
Further, the empirical measures of trajectories $\Sigma_\lambda(S)=(R_{\lambda,k}(S))_{k \in [\kmax]}$
converge to this minimizer under the Gibbs distribution $\mathrm P^{\gamma,0}_{\lambda,X^\lambda}$, almost surely as $\lambda \to \infty$. Thus, qualitative information about $\Sigma=(\nu_k)_{k \in[\kmax]}$ yields information about these empirical measures for large $\lambda$.

In order to obtain clear pictures and neat results, one needs to analyse the minimizer in extreme regimes. That is, one has to carry out another limit after the high-density limit has been taken. We consider the following regimes: (1) large communication areas, coupled with large transmitter--receiver distances and large numbers of hops, (2) strong penalization of interference, (3) high local density of the intensity measure on a subset of $W$. In these regimes, the probability of deviating from the typical behaviour decays exponentially fast. This indicates that the behaviour of the minimizers is close to the limiting one already for moderate values of the diverging parameter(s). In regime (2), this indication is supported by numerical examples \cite[Section 8]{KT17b}. We now survey the main results of \cite{KT17b} about regimes (1), (2), (3) in words; for further details, see \cite[Sections 3, 4, 5]{KT17b}, respectively. 

In regimes (1) and (2), the typical trajectory quickly approaches the straight line between transmitter and receiver, and the probability of macroscopic deviations from this straight line decays exponentially fast. Interestingly, in regime (1), the typical number of hops tends to infinity in a sublinear way, thus the typical length of a hop tends to infinity. In regime (3), we analyse the global and local repulsive effect of a particularly highly populated subset of $W$. Here, we replace the intensity measure $\mu$ by $\mu^a = \mu + a \Leb|_\Delta$ for some $\Delta \subseteq W$ with $\Leb(\Delta)>0$, and we let $a \to \infty$. The global effect is that if $\ell$ is close to constant on $W$, then $M^a(W)$ tends to zero exponentially fast as $a \to \infty$. Here, $M^a$ is the measure $M$ of all relaying hops (cf.~\eqref{constraints}) corresponding to the minimizer \eqref{nukminimizerbeta=0special}, in case $\mu$ is replaced by $\mu^a$. As for local effects, we discuss under what conditions it becomes unlikely for small $\Delta$ for a user to choose a relay inside a small neighbourhood of $\Delta$ than one outside a larger neighbourhood of $\Delta$.

\subsection{Discussion}\label{sec-Discussion}

\subsubsection{The interference term and the congestion term}\label{sec-SIRpenaltydiscussion1}
The interference term $\mathfrak S(s)$ in \eqref{realSIRenergy} quantifies the joint quality of the transmissions of the messages in terms of a sum of an individual interference term over all the $N(\lambda)$ trajectories and over all of their hops. 
The reason why we choose the {\em reciprocals} of the SIRs is that the \emph{bandwidth} used for a transmission is defined \cite{SPW07} as 
\[ \frac{\varrho}{\log_2 (1+ \mathrm{SIR}(\cdot))}, 
 \numberthis\label{bandwidth}\]
where $\varrho$ is the data transmission rate, and $\mathrm{SIR}$ is defined as in \eqref{SIR} but without the factor of $1/\lambda$ in the denominator. This quantity is of order $1/\lambda$ for $\lambda$ large, under the assumption that $L_\lambda \Longrightarrow \mu$. Thus, in the high-density setting $\lambda\to\infty$, \eqref{bandwidth} can be approached well by (a constant times) the reciprocals of the SIR, since $\log(1+x) \sim x$ as $x \to 0$.

The reason that we took the sum over all the hops of a trajectory is that \cite[Section 3]{SPW07} suggests that in case of multihop communication, the used bandwidth equals the sum of the used bandwidth values corresponding to the individual hops. See \cite{BC12} for another (single hop) setting where the sum of inverse values of SIRs appears as a cost function to be minimized. 

The congestion term $\mathfrak M(s)$ in \eqref{realMenergy} counts the number of ordered pairs of incoming hops arriving at the relays in the system. This is certainly an important characteristics of the quality of service, as too high an accumulation of many messages at relays results in a delay. Hence, it is natural to suppress the occurrence of such events, in order to increase the value of the model for realistic modeling. 

An important property of this term is that it introduces dependence between the trajectories of different messages, unlike the interference term. Indeed, while $\mathfrak S(s)$ can be decomposed into a sum of terms depending on the respective trajectories, each summand in $\mathfrak M(s)$ involves many different trajectories. This is not only true in the special case of penalizing interference and congestion, but in general in our setting introduced in Section~\ref{sec-Gibbsdistribution}.

\subsubsection{The entropy term} \label{sec-Discussion-varformula}

Let us now explain some important properties of the entropy term
\[ \mathrm I(\Psi)=\sum_{k=1}^{\kmax} \mathcal H_{W^k}\big(\nu_k \mid \mu \tensor M^{\tensor (k-1)}\big)+ \sum_{m=0}^{\infty} \mathcal H_{W}(\mu_m \mid \mu c_m)+\mu(W) \Big( 1-\sum_{k=1}^{\kmax} M(W)^{k-1}\Big)-\frac1\e \numberthis\label{quenchedentropyoncemore} \]
defined in \eqref{quenchedentropy}. 

According to (i) and (iii) in \eqref{constraints}, we have that $M(W) \leq (\kmax-1)\mu(W)$, further, the first term on the right-hand side of \eqref{quenchedentropy} is bounded from below. Moreover, since by (ii) in \eqref{constraints}, we have
 \[ \sum_{m \in \N_0} c_m = \sum_{m \in \N_0} \frac{\mu_m(W)}{\mu(W)} = 1, \numberthis\label{mumentropymassidentity} \] it follows that $\sum_{m=0}^{\infty} \mathcal H_{W}(\mu_m \mid \mu c_m)$ is nonnegative. These together with elementary properties of the relative entropy \cite[Section 6.2]{DZ98} imply that $\mathrm I(\Psi)$ is well-defined as an element of $(-\infty,\infty]$ and $\Psi\mapsto \mathrm I(\Psi)$ is a lower semicontinuous function that is bounded from below. More precisely, the LDP in Theorem~\ref{thm-LDP} implies that the infimum of $\mathrm I(\Psi)$ over admissible trajectory settings equals $-\mu(W) \log \kmax$, which equals the almost sure limit of $1/\lambda$ times the logarithm of the total mass $\kmax^{N(\lambda)}$ of the joint \emph{a priori} measure \eqref{referenceprod}.
 
Let us now provide an interpretation of $\mathrm I(\cdot)$. Let $\lambda>0$ and $s \in \Scal_{\kmax}(X^\lambda)$. Recall the empirical measure family $\Psi_\lambda(s)= \big((R_{\lambda,k}(s))_{k \in [\kmax]}, (P_{\lambda,m} (s))_{m\in\N_0}\big)$ from \eqref{Psilambdadef} and the constraints \eqref{discreteconstraints}, which are similar to the ones \eqref{constraints} but with $\mu$ replaced by the rescaled empirical measure $L_\lambda$ everywhere.

Informally speaking, for $\lambda>0$ large, $\mathrm I(\Psi)$ asymptotically describes the following crucial counting term:
\[ \mathrm I(\Psi) \approx -\frac{1}{\lambda} \log \frac{ \# \Big\{ s \in \Scal_{\kmax}(X^\lambda) \colon R_{\lambda,k}(s) \approx  \nu_k,~\forall k \in [\kmax] \text{ and } P_{\lambda,m}(s) \approx \mu_m,~\forall m \in \N_0\Big\}}{N(\lambda)^{\sum_{i \in I^\lambda}  (\widetilde s^i_{-1}-1)}}, \numberthis\label{Iinformal} \]
where $\widetilde s$ in the denominator is an arbitrarily chosen element of the set in the numerator. In this way, it fully describes the distribution of $\Psi_\lambda(s)$ on an exponential level.

A major part of the proof of Theorem~\ref{theorem-variation1} consists in making \eqref{Iinformal} rigorous and verifying the corresponding formal statement. In the beginning of Section~\ref{sec-proofingredients}, we will argue why taking ``$=$'' signs instead of ``$\approx$'' in the numerator of the right-hand side of \eqref{Iinformal} is not applicable. Instead, first, in Section~\ref{sec-discretization}, we will introduce a spatial discretization procedure and formulate a rigorous discrete version of these ``$\approx$'' relations for fixed $\lambda$ and fixed fineness parameter of the discretization. 
In Section~\ref{sec-combinatorics}, we will derive explicit combinatorial formulas for the cardinality of the trajectories in this setting. Next, in Section~\ref{sec-asymptotics}, we will conclude that $1/\lambda$ times the logarithm of the quotient of the obtained counting complexity and the term $N(\lambda)^{\sum_{i \in I^\lambda}  (\widetilde s^i_{-1}-1)}$ tends to $\mathrm I(\Psi)$ in the limit $\lambda \to \infty$ followed by the fineness parameter of the discretization tending to zero.

The characterization of $\mathrm I(\Psi)$ that arises directly from this argument is not exactly \eqref{quenchedentropyoncemore} but another, however identical expression \eqref{I-proof}. The reason why we chose \eqref{quenchedentropyoncemore} as the definition of $\mathrm I(\Psi)$ is that it is given in terms of objects that are commonly used in large deviations theory: relative entropies, multiples of total masses of the corresponding measures plus an additive constant. In particular, we find it natural in \eqref{quenchedentropyoncemore} that each $\mu_m$ is compared to the intensity measure $\mu$ multiplied by the weight of a Poisson distribution at $m$. Indeed, in case $\beta=0$, for the minimizer $\Psi$ of the variational formula \eqref{nukminimizerbeta=0}, $(\smfrac{\d \mu_m}{\d \mu}(x))_{m\in\N_0}$ is the Poisson distribution with parameter $\smfrac{\d M}{\d \mu}(x)$ for each $x \in W$.  Roughly speaking, the relative entropies plus the linear 
term $\mu(W)( 1-\sum_{k=1}^{\kmax} M(W)^{k-1})$ arise from $1/\lambda$ times the logarithmic rates of certain multinomial expressions in the above mentioned discretized counting procedure, after carrying out the limit $\lambda \to \infty$ followed by the fineness parameter tending to zero.

\subsubsection{Rotation symmetry} Let us assume that $\beta=0$ (cf. Section~\ref{sec-beta=0description}), and let us consider the special setting of Section~\ref{sec-interference&congestion} where $\mathfrak S$ penalizes interference. If $W=\overline{B_r(o)} \subset \R^d$ is a closed ball and $\mu$ is invariant under rotations, then the measures $(\nu_k)_k$ in \eqref{nukminimizerbeta=0} are also invariant under rotations of the entire trajectory, i.e., for any orthogonal $d\times d$-matrix $O$, we have that $\nu_k(\d x_0,\ldots,\d x_{k-1})=\nu_k^O(\d x_0,\ldots,\d x_{k-1}) \equiv \nu_k(\d (O x_0), \ldots, \d (O{x_{k-1}}))$ for any $k \in [\kmax]$. This is easily seen by an inspection of the formulas for the entropy term $\mathrm J$ in \eqref{Jentropy} and for the interference term $\mathrm S$ in \eqref{limitSIR}, as the function $(x,y)\mapsto \int_{W} \ell(|z-y|)\,\mu(\d z)/\ell(|x-y|)$ is invariant under multiplication of both arguments with the same orthogonal matrix.

\subsubsection{Non-Poissonian users} \label{sec-poisson}

In fact, the main results of this paper hold for any collection of (random or non-random) point processes $((X_i)_{i=1,\ldots,N(\lambda)})_{\lambda>0}$ on $W$ for which $L_\lambda=\frac{1}{\lambda} \sum_{i=1}^{N(\lambda)} \delta_{X_i}$ converges weakly (almost surely, if random) to $\mu$ as $\lambda \to \infty$. Neither the independence or monotonicity in $\lambda$, nor the Poissonity of $(N(\lambda))_{\lambda>0}$ is used for the proofs. For example, our results remain also true for the deterministic set $X^\lambda=W\cap(\frac 1\lambda \Z^d)$ and $\mu$ the Lebesgue measure on $W$.

\subsubsection{The annealed setting}\label{sec-annealed}

Of mathematical interest might also be the annealed setting, where we average also over the locations of the users. In order to get an interesting result, we have to assume that $L_\lambda$ satisfies a large deviation principle on the set $\Mcal(W)$ with some good rate function $\mathrm H$. (In the case of a Poisson point process with intensity measure $\lambda \mu$, $\mathrm H$ would be \cite[Proposition 3.6]{HJP16} the relative entropy with respect to $\mu$, see \eqref{relativeentropy}.) Then the large-$\lambda$ exponential rate of the annealed free energy should be equal to the negative infimum over $\mu_0\in\Mcal(W)$ of $\mathrm H(\mu_0)$ plus the quenched rate function terms from the right-hand side of \eqref{variation} with $\mu$ replaced by $\mu_0$ everywhere. Also our other results on the LDP and the form of the minimizer(s) should have some analogue, which we do not spell out.

\subsubsection{Extensions of the model of Section~\ref{sec-interference&congestion}}\label{sec-modelextensions}
The main goal of this paper is to set up a model where message routeing happens probabilistically, with a uniform \emph{a priori} distribution weighted by two penalization terms of different nature. Our results demonstrate how the interplay between entropy and energy leads to an orderly behaviour of the system in the high-density limit on its own.

Section~\ref{sec-interference&congestion} provides an example of the two penalization terms that can be interpreted in terms of well-known objects in telecommunication. In Section~\ref{sec-furtherresults}, we summarized qualitative properties of the minimizer of the variational formula in this special setting for $\beta=0$. This special case can be viewed as a snap-shot type model, where there is no time-dependence but all hops of all transmissions happen simultaneously at the same time. Further, we assume that each user submits exactly one message. However, one can relax these assumptions and make the model more realistic in various ways.

First, one easily sees from the proofs in Sections~\ref{sec-proofingredients}--\eqref{sec-beta=0proof} that Theorems~\ref{theorem-variation1} and \ref{thm-LDP} as well as Proposition~\ref{prop-variationbeta=0} can be extended to cases where users send out no message or multiple messages. This models the standard situation in which large messages are cut into many smaller ones, who independently find their ways through the system. For this, the trajectory probability space has to be enlarged: to each user $X_i \in X^\lambda$, we attach the number $P_i\in\N_0$ of transmitted messages, and for each $j\in\{1,\dots,P_i\}$, there is an independent trajectory $X_i\to o$. The empirical trajectory measure $R_{\lambda,k}(\cdot)$ must be augmented by these trajectories. The main additional assumption then is that $\sum_{k=1}^{\kmax} \pi_0 R_{\lambda,k}(S)$ converges to some measure $\mu_0 \in \mathcal{M}(W)$ with $0 \neq \mu_0 \ll \mu$. (Also the case that $\mu_0$ is not absolutely continuous with respect to $\mu$ is 
interesting, but will need additional work.) The interference term $\mathfrak S(\cdot)$ introduced in \eqref{realSIRenergy} also has to be changed. According to \cite[Sections 2.3.1, 5.1]{BB09}, the SIR of the transmission of one of the $P_i$ messages from $X_i$ to $x \in W$ should be defined as 
\[ \frac{\ell(|X_i-x|)}{\frac{1}{\lambda} \sum_{j \in I^\lambda} \ell(|X_j-x|)P_j}. \]
One could also incorporate (possibly random) sizes of the messages, which would require an additional enlargement of the trajectory space.

Second, the model of Section~\ref{sec-interference&congestion} can be made time-dependent. If one, e.g., introduces $\kmax$ discrete time slots indexed by $[\kmax]$, and assumes the $l$th hop of any message trajectory to happen at time $l$ for any $l \in [\kmax]$, then the interference of a transmission at time $l$ is obtained from the starting points of all hops that happen at the same time, see \cite[Section I.A]{GK00}. The SIR is defined analogously to \eqref{SIR} but with this notion of interference, which depends on the entire message trajectories rather than only on the users. Further, the congestion term can also be adapted to the time-dependent situation via counting numbers of incoming hops at each time step separately. This variant of the model is indeed mathematically not far from the model treated in the present paper; for example the entropy term does not change.

More realistic and mathematically much more demanding time-dependent versions of our model can be set up in various ways; for example, one could allow for a much longer time horizon (for example, of order $\lambda$, and then dropping the factor of $\lambda$ in the interference term), which must come with the possibility of messages standing still for many separate time units. Furthermore, one could allow users to transmit multiple messages over time. One could also introduce mobility of users similarly to \cite{HJKP15}. The new notion of SIR comes with significant changes in the behaviour of the system in the high-density limit, and we decided to defer such investigations to a later work.

\subsubsection{Allowing an unbounded number of hops}\label{sec-kmaxdiscussion} 

If the upper bound $\kmax$ for the length of the trajectories is dropped, then the {\em a priori} measure defined in  \eqref{referenceprod} has infinite total mass, and therefore the entropy function $\rm I$ is not bounded from below. However, since the function $f_k$ in \eqref{SIRenergy} is bounded from below, for any $\gamma>0,\beta \geq 0$, the total probability mass of all the $k$-hop trajectories under the Gibbs distribution is upper bounded by some geometrically decaying term in $k$. Hence, the definition of the model is no problem for $\kmax= \infty$ and $\gamma>0,\beta \geq 0$. We believe that all our results of Section~\ref{sec-Intro} about its limiting behaviour as $\lambda\to\infty$ remain essentially true (apart from the LDP under the {\em a priori} measure). However, proofs will require an additional cutting argument, which might become rather nasty. Our belief that the results remain unchanged is supported by the fact that also the minimizing objects $\Psi=((\nu_k)_{k=1}^{\kmax},(\mu_m)_{m=0}^{\infty})$ defined in Proposition~\ref{prop-minimizer} enjoy a geometric upper bound for $\nu_k(W^k)$ in $k$. Thus, these measures are also well-defined and form the set of minimizers of the variational formula \eqref{variation} in case $\kmax=\infty$. The assertions of \cite{KT17b} corresponding to the large-distance limit, which we explained in Section~\ref{sec-furtherresults}, are proved also for $\kmax=\infty$, see \cite[Section 3.3.3]{KT17b}. 

\subsubsection{Relation to an optimization problem via Monte Carlo Markov chains}\label{sec-MCMC} 

In the light of the Section~\ref{sec-SIRpenaltydiscussion1}, it is certainly interesting to minimize the cost function $s \mapsto \gamma \mathfrak S(s) + \beta \mathfrak M(s)$ for fixed $\gamma,\beta\in(0,\infty)$. (For game-theoretic properties of this optimization problem, we refer the reader to \cite[Section 7]{KT17b}.) Computationally, this is in general a hard problem for high densities $\lambda$ because the cardinality of $\mathcal S_{\kmax}(X^\lambda)$ increases super-exponentially in $N(\lambda) \asymp \lambda$. Thus, computing all values of $s \mapsto \gamma \mathfrak S(s) + \beta \mathfrak M(s)$  and then extracting the maximum is only feasible for small $\lambda$.

Now, our Gibbs distribution opens the possibility to optimize this cost function via the well-known approach of {\em simulated annealing}. Furthermore, for $\lambda$ large, it is substantially less complex to realize the Gibbs distribution using Monte Carlo Markov chains than to directly minimize the cost function. Thus, our Gibbsian ansatz provides an easier implementable joint strategy for the routeing of all messages, preferring trajectory collections having low cost function values, already for moderate values of $\gamma,\beta$. 

Indeed, the recent master's thesis of Morgenstern \cite{M18} investigates the computational complexity of realizing the Gibbs distribution $\mathrm P^{\gamma,\beta}_{\lambda,X^\lambda}$ numerically using Monte Carlo Markov chains, for $\lambda,\beta,\gamma>0$. The author finds irreducible and aperiodic Markov chains on the state space $\mathcal S_{\kmax}(X^\lambda)$, both of Gibbs sampler and Metropolis types, having the Gibbs distribution as their stationary distribution. These chains converge towards $\mathrm P^{\gamma,\beta}_{\lambda,X^\lambda}$ as the number of Markovian steps tends to infinity. Using these chains, the number of operations needed in order to simulate the Gibbs distribution up to a given error $\eps>0$ in total variation distance is at most exponential in $\lambda$. This is much more efficient than evaluating all the trajectory collections. In a variant of the Gibbsian model where each user can receive at most a given number $\mmax\in \N_0$ incoming hops, the number of necessary 
operations is even polynomial in $\lambda$.

These Monte Carlo Markov chains can also be used in order to find the optimum of the cost function $s \mapsto \gamma \mathfrak S(s) + \beta \mathfrak M(s)$ for a fixed $\lambda$ and a fixed realization of $X^\lambda$, using simulated annealing. Here, one lets the transition probability of the $t$-th step of the chain depend on $t$ via replacing $(\gamma,\beta)$ by $(\gamma_t,\beta_t)$ such that $\gamma_t,\beta_t \to \infty$ sufficiently slowly as $t \to \infty$. \cite[Theorem 7.1]{M18} shows that if one chooses $\beta_t=\smfrac{\beta}{\gamma} \gamma_t \leq c_0 \log t$ for a suitably chosen $c_0=c_0(\lambda) \asymp \lambda/N(\lambda)^2$, then the Markov chain converges to the uniform distribution on the set of minimizers of the cost function. 

\section{The distribution of the empirical measures} \label{sec-proofingredients}

Having seen in Section~\ref{sec-LDP} that the Gibbsian model can be entirely described in terms of the  trajectory setting $\Psi_\lambda(s)$,  i.e., of the crucial empirical measures $R_{\lambda,k}(s)$ and $P_{\lambda,m}(s)$ defined in \eqref{Rdef}--\eqref{Pdef}, we now consider the question how to describe their distributions. We have to quantify the number of message trajectory families $s$ that give the same family of empirical measures. The plain and short (but wrong) answer is
\begin{equation}\label{wishfulthinking}
\sum_{s\in \mathcal S_{\kmax}(X^\lambda)\colon R_{\lambda,k}(s)=\nu_k \,\forall k,~P_{\lambda,m} (s)=\mu_m\,\forall m} \quad \prod_{i\in I^\lambda}\frac1{N(\lambda)^{s^i_{-1}-1}}\approx \e^{-\lambda \mathrm I(\Psi)},
\end{equation}
where we recall $\mathrm I(\Psi)$ from \eqref{quenchedentropy} and recall that $\Psi=((\nu_k)_{k \in [\kmax]},(\mu_m)_{m\in\N_0})$. From such an assertion, it is indeed not far to conclude Theorem \ref{theorem-variation1}, but the problem is that this statement is not true like this. Actually, there are very many $\Psi$'s such that the left-hand side is equal to zero, for example if any of the $\nu_k$'s or $\mu_m$'s has values outside $\frac 1\lambda \N_0$. However, if we do not consider single $\Psi$'s, but open sets of $\Psi$'s, then the idea behind \eqref{wishfulthinking} is sustainable. Therefore, we proceed in a standard way by decomposing the area $W$ into finitely many subsets and count the message trajectories only according to the discretization sets that they visit. In Section~\ref{sec-discretization} we introduce necessary notation for carrying out this strategy, and we comment on the relevance of the discretization procedure. Next, in Section~\ref{sec-combinatorics}, we derive explicit formulas 
for the distribution of the empirical measures in this discretization.

For the purpose of the present paper, where we consider the high-density limit $\lambda\to\infty$, we later need to take this limit and afterwards the limit as the fineness parameter $\delta$ of the decomposition of $W$ goes to zero. 
The outcome of these parts of the procedure is formulated in Proposition~\ref{prop-combinatorics}. In Proposition~\ref{prop-SIRcomputations} the consequences for the interference term and for the congestion term are formulated. 

\subsection{Our discretization procedure}\label{sec-discretization}

Let us now head towards the formulation of the discretization procedure. We proceed by triadic spatial discretization of the Poisson point process $(X^\lambda)_{\lambda>0}$, similarly to the approach of \cite{HJKP15}. To be more precise, we perform the following discretization argument. Note that we may assume that our communication area $W$ is taken as $W=[-r,r]^d$, by accordingly extending $\mu$ trivially. We write $\mathbb B=\lbrace 3^{-n} \vert n \in \mathbb N_0 \rbrace$.  For $\delta \in \mathbb B$, we define the set
\[ W_\delta = \lbrace [x-r\delta, x+r\delta]^d \colon x \in (2r\delta\mathbb Z)^d \cap W \rbrace \]
of congruent sub-cubes of $W$ of side length $2 r \delta$ and centres in $(2r\delta \Z)^d$. Note that $W_\delta$ is a finite set, $o$ is a centre of an element of $W_\delta$ and any intersection of two distinct elements of $W_\delta$ has zero Lebesgue measure. Elements of $W_\delta$ will be called \emph{$\delta$-subcubes}. We will assume that for all $\delta \in \mathbb B$, the $\delta$-subcubes are canonically numbered as $W^\delta_1,\ldots,W^\delta_{\delta^{-d}}$, which can be done e.g. according to the increasing lexicographic order of the midpoints of the subcubes. Now, for Lebesgue-almost every $x \in W$, for all $\delta \in \mathbb B$ there exists a unique $W^\delta_j$ that contains $x$; let us denote this $W^\delta_j$ by $W^x_\delta$, and the set of all $x \in W$ for which $W^x_\delta$ is well-defined by $W_{\mathbb B}$. 

Now, if $\nu \in \mathcal{M}(W)$, then for any $\delta \in \mathbb B$, we define $\nu^{\delta}(\cdot)=\nu(\cdot\mid \Fcal_\delta)\in \mathcal{M}(W)$ as the conditional version of $\nu$ given $\Fcal_\delta=\sigma(W_\delta)$, that is, the measure on $W$ that has in each box $W^\delta_i$ a constant Lebesgue density and mass equal to $\nu(W^\delta_i)$. Since $\Fcal_\delta\subset\Fcal_{\delta'}$ for $\delta,\delta'\in\mathbb B$ with $\delta'<\delta$, we see that $(\nu^{\delta})^{\delta'}=(\nu^{\delta'})^{\delta}=\nu^\delta$ by the tower property. We also write $L_\lambda^\delta:=(L_\lambda)^\delta$ for $\lambda>0$ and $\delta \in \mathbb B$, where the empirical measure $L_\lambda$ was defined in \eqref{userempirical}. We proceed analogously for $W^k$, $k \in [\kmax]$ instead of $W$. Note that $\nu^\delta\Longrightarrow \nu$ as $ \delta\downarrow 0$, which can be shown by a martingale convergence argument, since the union of all the $\Fcal_\delta$ generate the Borel-$\sigma$-field on $W$.

Now we are able to define what a \emph{standard setting} is, the interpretation of which will be given right after the definition. Roughly speaking, the measures $\nu_k$ and $\mu_m$ appearing in its definition will later play the role of the measures appearing in \eqref{wishfulthinking}, their $\delta$-approximations are defined as above, and their $(\delta,\lambda)$-versions approach them in the limit $\lambda\to\infty$, followed by $\delta\downarrow 0$. The latter ones satisfy the constraints of \eqref{discreteconstraints} restricted to $\Fcal_\delta^{\otimes k}$ respectively $\Fcal_\delta$, hence, the $\nu_k$ and $\mu_m$ will later turn out to be eligible for the variational problem in  \eqref{variation} (under some mild additional assumption, see Lemma~\ref{lemma-constraints}). 

\begin{defn} \label{defn-standardsetting}
A \emph{standard setting} is a collection of measures  
\begin{equation}\label{standardsetting}
\begin{aligned}
\underline{\Psi}&=\Big( (\nu_k)_{k=1}^{\kmax}, ((\nu_k^{\delta})_{k=1}^{\kmax})_{\delta \in \mathbb B}, ((\nu_k^{\delta,\lambda})_{k=1}^{\kmax})_{\delta \in \mathbb B,\lambda>0},\\
&\quad (\mu_m)_{m=0}^{\infty},((\mu_m^{\delta})_{m=0}^{\infty})_{\delta\in\mathbb B},((\mu_m^{\delta,\lambda})_{m=0}^{\infty})_{\delta\in\mathbb B,\lambda>0},
(\mu^{\delta,\lambda})_{\delta \in \mathbb B,\lambda>0}\Big)
\end{aligned}
\end{equation}
with the following properties: for any $\delta,\delta'\in\mathbb B$, $\lambda>0$, $k\in [\kmax]$, $m \in \mathbb N_0$ and $i,i_0,\ldots,i_{k-1}=1,\ldots,\delta^{-d}$, respectively, $\nu_k\in\Mcal(W^k)$ and $\mu_m\in\Mcal(W)$, and 

\begin{enumerate}
\item \label{muadmissible-standardsetting} $\mu^{\delta,\lambda}=L_\lambda^\delta$. 

\item \label{i-standardsetting} $\nu_k^{\delta,\lambda} \in \mathcal{M}(W^k)$. Further, $\sum_{k=1}^{\kmax} \pi_0 \nu_k^{\delta,\lambda}=\mu^{\delta,\lambda}$, moreover $\lambda \nu_k^{\delta,\lambda}(W^\delta_{i_0} \times \ldots \times W^\delta_{i_{k-1}}) \in \mathbb N_0$. 

\item \label{nukconsistency-standardsetting}  If $\delta' \leq \delta$, then $\nu_k^{\delta',\lambda} (\cdot \mid \Fcal_\delta^{\tensor k})=\nu_k^{\delta,\lambda}(\cdot)$.

\item \label{nukconv-standardsetting} $\nu_k^{\delta,\lambda} \overset{\lambda \to \infty}{\Longrightarrow}\nu_k^{\delta}$.

\item \label{ii-standardsetting} $\mu_m^{\delta,\lambda}\in \mathcal{M}(W)$ with the property that $\sum_{m=0}^{\infty} \mu_m^{\delta,\lambda}=\mu^{\delta,\lambda}$, moreover $\lambda \mu_m^{\delta,\lambda}(W^\delta_i) \in \mathbb N_0$. 

\item \label{iii-standardsetting} $\sum_{m=0}^{\infty} m \mu_m^{\delta,\lambda} = \sum_{k=1}^{\kmax} \sum_{l=1}^{k-1} \pi_l \nu_k^{\delta,\lambda}$.

\item \label{mumconsistency-standardsetting} If $\delta' \leq \delta$, then $\mu_m^{\delta',\lambda} (\cdot \mid \Fcal_\delta) =\mu_m^{\delta,\lambda}(\cdot)$.

\item \label{mumconv-standardsetting}  $\mu_m^{\delta,\lambda}\overset{\lambda \to \infty}{\Longrightarrow}\mu_m^{\delta}$.

\end{enumerate}
\end{defn}

\begin{remark}\label{remark-standardsetting}
Immediate properties of  a standard setting $\Psi$ are the following. 
\begin{enumerate}[(A)]
 \item\label{muconsistency-standardsetting} If $\delta' \leq \delta$, then $\mu^{\delta',\lambda}(\cdot \mid \Fcal_\delta) =\mu^{\delta,\lambda}(\cdot)$.
 
\item \label{muconv-standardsetting} $\mu^{\delta,\lambda} \overset{\lambda \to \infty}{\Longrightarrow}\mu^{\delta}$ since $L_\lambda \Longrightarrow \mu$ as $\lambda \to \infty$.

\item \label{mudelta-standardsetting} $\mu^{\delta}(\cdot )=\mu (\cdot \mid \Fcal_\delta)$. In particular, $\mu^\delta\overset{\delta\downarrow0}{\Longrightarrow}\mu$.

\item \label{nukdelta-standardsetting} $\nu_k^{\delta}(\cdot)=\nu_k (\cdot \mid \Fcal_\delta^{\tensor k})$. In particular, $\nu_k^\delta\overset{\delta\downarrow0}{\Longrightarrow}\nu_k$.

\item \label{mumdelta-standardsetting} $\mu_m^\delta(\cdot)= \mu_m (\cdot \mid \Fcal_\delta)$. In particular, $\mu_m^\delta\overset{\delta \downarrow 0}{\Longrightarrow}\mu_m$.
\end{enumerate}

\end{remark}

\begin{remark}[Standard settings and message trajectories]\label{remark-setttraj}
The properties of Remark~\ref{remark-standardsetting} explain the meaning of the $\delta$-indexed coordinates of a standard setting. Let us now interpret how one can obtain the $(\delta,\lambda)$-dependent coordinates of a standard setting starting from a fixed trajectory collection $s \in \Scal_{\kmax}(X^\lambda)$. Let
\[ P_\lambda(s)=\frac{1}{\lambda} \sum_{i \in I^\lambda} \delta_{s^i_0} \numberthis\label{mu-explanation} \] 
denote the empirical measure of the starting sites of the trajectories, then $P_\lambda(s)=L_\lambda$, by our assumption that each user is picked precisely once in such a configuration. Hence, its $\delta$-discretized version $P_\lambda^\delta(s)$ equals $\mu^{\delta,\lambda}$. Let us now choose $\nu_k^{\delta,\lambda}=R_{\lambda,k}^\delta(s)$ (recall \eqref{Rdef}). Then the requirement $\sum_{k=1}^{\kmax} \pi_0 \nu_k^{\delta,\lambda}=\mu^{\delta,\lambda}=L_\lambda^\delta$ in \eqref{i-standardsetting} holds by \eqref{discreteconstraints}. Further, let us choose $\mu_m^{\delta,\lambda}=P_{\lambda,m}(s)$ (recall \eqref{Pdef}). Then the constraints $\sum_{m=0}^{\infty} \mu_m^{\delta,\lambda}=\mu^{\delta,\lambda}$  in \eqref{ii-standardsetting} and $\sum_{m=0}^{\infty} m \mu_m^{\delta,\lambda} = \sum_{k=1}^{\kmax} \sum_{l=1}^{k-1} \pi_l \nu_k^{\delta,\lambda}$ in \eqref{iii-standardsetting} also hold by \eqref{discreteconstraints}. Note that all the other requirements of Definition~\ref{defn-standardsetting} are 
satisfied because of the tower property respectively of the convergence of $L_\lambda$ towards $\mu$.
\end{remark}

In the proof of Theorem~\ref{theorem-variation1}, it will be essential to verify that, for certain standard settings, $\sum_{m=0}^{\infty} m\mu_m^\delta(W)$ converges to $\sum_{m=0}^{\infty} m\mu_m(W)$ as $\delta \downarrow 0$, which is not implied by Definition~\ref{defn-standardsetting}. However, similarly to the de la Vallée Poussin theorem about uniform integrability, the super-linear increase of $m \mapsto \eta(m)$ yields the following handy criterion.

\begin{defn} \label{defn-controlledstandardsetting}
A \emph{controlled standard setting} is a standard setting $\underline{\Psi}$ as in \eqref{standardsetting}  with the following extra property:
\begin{equation} \label{third-controlleddef}
\lim_{\lambda \to \infty} \sum_{m=0}^{\infty} \eta(m) \mu_m^{\delta,\lambda}(W)=\sum_{m=0}^{\infty} \eta(m) \mu_m^{\delta}(W)<\infty, \quad \text{for all } \delta \in \mathbb B.
\end{equation}
\end{defn}

Note that by part \eqref{nukdelta-standardsetting} of Remark \ref{remark-standardsetting}, we have $\sum_{k=1}^{\kmax} k \nu_k^\delta(W^k)=\sum_{k=1}^{\kmax} k \nu_k(W^k)$ for any standard setting. Using this, we verify the following lemma.

\begin{lemma} \label{lemma-constraints}
Let $\underline{\Psi}$ be a controlled standard setting as in \eqref{standardsetting}. Then $\Psi=((\nu_k)_{k=1}^{\kmax},(\mu_m)_{m=0}^{\infty})$ is an admissible trajectory setting.
\end{lemma}

\begin{proof}
Part \eqref{i-standardsetting} of Definition \ref{defn-standardsetting} claims that for all $\delta \in \mathbb B$ and $\lambda>0$ we have $\sum_{k=1}^{\kmax} \pi_0 \nu_k^{\delta,\lambda}=\mu^{\delta,\lambda}$. By parts \eqref{muconv-standardsetting} and \eqref{mudelta-standardsetting} of Remark \ref{remark-standardsetting}, we have $\lim_{\delta \downarrow 0} \lim_{\lambda \to \infty} \nu_k^{\delta,\lambda}=\nu_k$ in the weak topology of $\mathcal{M}(W^k)$, for any fixed $k \in [\kmax]$. Similarly, by part \eqref{nukconv-standardsetting}  of Definition \ref{defn-standardsetting} and part \eqref{nukdelta-standardsetting} of Remark \ref{remark-standardsetting}, we have $\lim_{\delta \downarrow 0} \lim_{\lambda \to \infty} \mu^{\delta,\lambda}=\mu$ in the weak topology of $\mathcal{M}(W)$.
Moreover, since taking marginals is a continuous operation, also $\lim_{\delta \downarrow 0} \lim_{\lambda \to \infty} \pi_0 \nu_k^{\delta,\lambda}=\pi_0 \nu_k$ for all $k$ in the weak topology of $\mathcal{M}(W)$. Thus, we have (i) in \eqref{constraints} for $(\nu_k)_{k=1}^{\kmax}$.
In order to see that (ii) holds for $(\mu_m)_{m=0}^{\infty}$, one can additionally use part \eqref{ii-standardsetting} of Definition \ref{defn-standardsetting}, together with \eqref{third-controlleddef} and dominated convergence. 
Finally, by part \eqref{iii-standardsetting} of Definition \ref{defn-standardsetting}, \eqref{third-controlleddef} in Definition \ref{defn-controlledstandardsetting}, the fact that $\lim_{m \to \infty} \eta(m)/m=\infty$ and dominated convergence, we see that for any controlled setting $\underline{\Psi}$, we also have
\[\sum_{m=0}^{\infty} m \mu_m= \lim_{\delta \downarrow 0} \sum_{m=0}^{\infty} m \mu_m^{\delta}= \lim_{\delta \downarrow 0} \lim_{\lambda \to \infty} \sum_{m=0}^{\infty} m \mu^{\delta,\lambda}_m = \lim_{\delta \downarrow 0} \lim_{\lambda \to \infty} \sum_{k=1}^{\kmax} \sum_{l=1}^{k-1} \pi_l \nu_k^{\delta,\lambda} = \sum_{k=1}^{\kmax} \sum_{l=1}^{k-1} \pi_l \nu_k \numberthis\label{second-controlleddef} \]
in the weak topology of $\mathcal{M}(W)$.
This implies (iii) in \eqref{constraints} for $\Psi$. Hence, $\Psi$ is an admissible trajectory setting.
\end{proof}

A converse of Lemma~\ref{lemma-constraints} also holds, in the following sense: for any admissible trajectory setting $\Psi=((\nu_k)_{k=1}^{\kmax},(\mu_m)_{m=0}^{\infty})$, there exists a standard setting containing it, which can be chosen controlled if $\sum_{m \in \N_0} \eta(m) \mu_m(W)<\infty$. This will be the content of Proposition~\ref{prop-standardsettingconstruction}, a preliminary result for Theorems~\ref{theorem-variation1} and \ref{thm-LDP}(i). 

\subsection{The distribution of the empirical measures}\label{sec-combinatorics}

In this section, we describe the combinatorics of the system. For a standard setting $\underline \Psi$ as in Definition \ref{defn-standardsetting}, let us introduce the configuration set
\begin{equation}\label{numberofconf}
J^{\delta,\lambda}(\underline{\Psi})=\Big\{ s \in \mathcal{S}_{\kmax}(X^\lambda) ~\Big|~ R_{\lambda,k}^\delta(s) = \nu_k^{\delta,\lambda}~\forall k,\quad P_{\lambda,m}^\delta(s)=\mu_m^{\delta,\lambda}~\forall m \Big\} 
\end{equation}
for fixed $\delta \in \mathbb B$ and $\lambda>0$. In words, $ J^{\delta,\lambda}(\underline{\Psi})$ is the set of families of trajectories such that the $\delta$-coarsenings of the empirical measures of the trajectories and the hop numbers are given by the respective measures in the setting $\underline \Psi$. Note that $J^{\delta,\lambda}(\underline{\Psi})$ depends only on the $\delta$-$\lambda$ depending measures in the collection $\underline \Psi$.

In case $\mu^{\delta,\lambda}(W)>0$, we will refer to the entity $s^i_0$, $i=1,\ldots,\lambda \mu^{\delta,\lambda}(W)$, as the $i$th user or $i$th transmitter, the entity $s^i$, $i=1,\ldots,\lambda \mu^{\delta,\lambda}(W)$, as the trajectory of the $i$th user, $s^i_{-1}$ as the length (number of hops) of $s^i$, $s^i_l$ as the $l$-th relay of $s^i$ (for $l=1,\ldots,s^i_{-1}-1$), finally $m_i(s)$ as the number of incoming hops at the relay $s^i_0$.

The combinatorics of computing $\# J^{\delta,\lambda}(\underline{\Psi})$ is given as follows. 

\begin{lemma}[Cardinality of $J^{\delta,\lambda}(\underline{\Psi})$]\label{lem-cardinality}
For any $\delta,\lambda>0$, and for any standard setting $\underline\Psi$, 
\begin{equation}\label{Jcard} 
\# J^{\delta,\lambda}(\underline{\Psi})=N^1_{\delta,\lambda}(\underline{\Psi})\times N^2_{\delta,\lambda}(\underline{\Psi}) \times N^3_{\delta,\lambda}(\underline{\Psi}), 
\end{equation}
where
\begin{eqnarray}
N^1_{\delta,\lambda}(\underline\Psi)&=&\prod_{i=1}^{\delta^{-d}} \binom{\lambda \mu^{\delta,\lambda}(W^\delta_i)}{((\lambda \nu_k^{\delta,\lambda}(W^\delta_i \times W^\delta_{i_1} \times \ldots \times W^\delta_{i_{k-1}}))_{i_1,\ldots,i_{k-1}=1}^{\delta^{-d}})_{k=1}^{\kmax}}, \label{comb-multinomial}\\
N^2_{\delta,\lambda}( \underline\Psi)&=&\prod_{i=1}^{\delta^{-d}} \binom{\lambda \mu^{\delta,\lambda}(W^\delta_i)}{(\lambda \mu_m^{\delta,\lambda}(W^\delta_i))_{m \in \mathbb N_0}},\label{comb-littlemultinomial} \\
 N^3_{\delta,\lambda}(\underline\Psi)&=&\prod_{i=1}^{\delta^{-d}} \frac{\left(\lambda \sum_{k=1}^{\kmax} \sum_{l=1}^{k-1} \pi_l \nu_k^{\delta,\lambda}(W^\delta_{i})\right)!}{\prod_{m=0}^{\infty} m!^{\lambda \mu_m^{\delta,\lambda}(W^\delta_i)}}
 =\prod_{i=1}^{\delta^{-d}} \frac{\left(\lambda \sum_{m=0}^{\infty} m \mu_m^{\delta,\lambda} (W^\delta_{i})\right)!}{\prod_{m=0}^{\infty} m!^{\lambda \mu_m^{\delta,\lambda}(W^\delta_i)}}.\label{comb-lambdaloglambdaterm}
\end{eqnarray}
\end{lemma}

\begin{proof} We proceed in three steps by counting first the trajectories, registering only the partition sets $W_i^\delta$ that they travel through, second, for each $m\in\N_0$, the sets of relays in each partition set that receive precisely $m$ ingoing hops and finally the choices of the relays for each hop in each partition set. Since every choice in the three steps can be freely combined with the other ones, the product of the three cardinalities is equal to the number of all trajectory configurations with the requested coarsened empirical measures.

\begin{enumerate}[(A)]
\item \label{first-combproof} \emph{Number of the transmitters of trajectories passing through given sequences of $\delta$-subcubes.} For each configuration $s \in J^{\delta,\lambda}(\underline{\Psi})$ defined in \eqref{numberofconf}, in each $\delta$-subcube $W^\delta_{i}$, $i=1,\ldots,\delta^{-d}$, there are $\lambda \mu^{\delta,\lambda}(W^\delta_i)$ users. Out of them exactly $\lambda \nu_k^{\delta,\lambda}(W^\delta_i \times W^\delta_{i_1} \times \ldots W^\delta_{i_{k-1}})$ take exactly $k$ hops, having their first relay in $W^\delta_{i_1}$, their second in $W^\delta_{i_2}$ etc.\ and their $(k-1)$st relay in $W^\delta_{i_{k-1}}$, for any $k \in [\kmax]$ and $i_1,\ldots,i_{k-1}=1,\ldots,\delta^{-d}$. Such choices in different sub-cubes $W^\delta_i$ corresponding to the transmitters are independent. Thus, the total number of such choices equals the number $N^1_{\delta,\lambda}(\underline\Psi)$ defined in \eqref{comb-multinomial}. Note that for $i=1,\ldots,\delta^{-d}$, 
\[ \sum_{k=1}^{\kmax} \sum_{i_1,\ldots,i_{k-1}=1}^{\delta^{-d}} \nu_k^{\delta,\lambda}(W^\delta_i \times W^\delta_{i_1} \times \ldots \times W^\delta_{i_{k-1}}) =\sum_{k=1}^{\kmax} \pi_0 \nu_k^{\delta,\lambda}(W^\delta_i)=\mu^{\delta,\lambda}(W^\delta_i), \] 
where we used part \eqref{i-standardsetting} of Definition \ref{defn-standardsetting}; hence the multinomial expressions in \eqref{comb-multinomial} are well-defined.

\item \label{second-combproof} \emph{Number of incoming hops.} In this step, for any $\delta$-subcube $W^\delta_i$, we count all the possible ways to distribute the incoming hops among the relays ($=$ users) $X_j \in W^\delta_i$, under the two constraints that in $W^\delta_i$ there are $\lambda \mu^{\delta,\lambda}(W^\delta_i)$ potential relays, and for any $m \in \mathbb N_0$, exactly $\lambda \mu_m^{\delta,\lambda}(W^\delta_i)$ receive precisely $m$ incoming hops. Such choices are clearly independent of each other for different $\delta$-subcubes. Hence, the total number of such choices equals the number $N^2_{\delta,\lambda}( \underline\Psi)$ defined in \eqref{comb-littlemultinomial}. Again, the constraint \eqref{ii-standardsetting} from Definition \ref{defn-standardsetting} implies that the multinomial expression \eqref{comb-littlemultinomial} is well-defined. Clearly, all choices in this part are independent of the choices in part \eqref{first-combproof}.

\item \label{third-combproof} \emph{Number of assignments of the hops to the relays.} 
Assume that we have chosen one possible choice in part \eqref{first-combproof} and one possible choice in part \eqref{second-combproof}. We now derive the number of possible ways of distributing, for any $i$, all the incoming hops in $W^\delta_i$ among the relays in $W^\delta_i$. Let us call this number $M_i$, then we know from part \eqref{first-combproof} that $M_i=\lambda \sum_{k=1}^{\kmax} \sum_{l=1}^{k-1} \pi_l \nu_k^{\delta,\lambda} (W^\delta_{i}) $, since each such hop is the $l$-th of some of the trajectories for some $l$. The cardinality of the set of relays in $W_i^\delta$ is equal to $\lambda  \mu^{\delta,\lambda}(W_i^\delta)=\lambda \sum_{m=0}^{\infty} \mu_m^{\delta,\lambda}(W_i^\delta)$, and in part \eqref{second-combproof} we decomposed it into sets, indexed by $m$, in which each relay receives precisely $m$ ingoing hops. Let us call such a relay an $m$-relay. Think of each such relay as being replaced by precisely $m$ copies (in particular those with $m=0$ are discarded), then we have $\lambda \sum_{m=0}^{\infty} m \mu_m^{\delta,\lambda}(W^\delta_i)$ virtual relays in $W_i^\delta$. (Note that 
this is equal to $M_i$ by \eqref{iii-standardsetting}.) Now, if all these $m$ copies of the $m$-relays were distinguishable, then the number of ways to distribute the $M_i$ ingoing hops to the relays would be simply equal to $M_i!$. However, since these $m$ copies are identical, we overcount by a factor of $m!$ for any $m$-relay.  This means that the number of hops into $W_i^\delta$ is equal to $M_i!/\prod_{m=0}^{\infty}(m!)^{\lambda\mu_m^{\delta,\lambda}(W_i^\delta)}$. Since all these cardinalities can freely be combined with each other, we have deduced that the number of possible choices is equal to the number $N^3_{\delta,\lambda}(\underline\Psi)$ defined in \eqref{comb-lambdaloglambdaterm}. 
\end{enumerate} 

We also see that all  the choices in the three parts are independent of each other, i.e., can be freely combined with each other and yield different combinations. Hence, we arrived at the assertion.
\end{proof}

\section{The limiting free energy and the LDP: proof of Theorems~\ref{theorem-variation1} and \ref{thm-LDP}}\label{sec-prooflargelambda}

In this section, we prove Theorems~\ref{theorem-variation1} and \ref{thm-LDP}(i), that is, we derive the variational formula in \eqref{variation} for the high-density (i.e., $\lambda\to\infty$) exponential rate of the partition function, and we verify the LDP for the empirical measures. Our first step is to derive the large-$\lambda$ exponential rate of the combinatorial formulas for  the empirical measures of Lemma~\ref{lem-cardinality} in Section~\ref{sec-asymptotics}. Furthermore, in Section~\ref{sec-SIRproof} we formulate and prove how the interference term and the congestion term behave in the limits $\lambda\to\infty$, followed by $\delta\downarrow0$. In Section~\ref{sec-admtrajetstandardsetting}, given an admissible trajectory setting, we construct a standard setting containing it. Using all these, in Section~\ref{sec-varproof} we prove Theorem~\ref{theorem-variation1}, and in Section~\ref{sec-LDPproof}, we complete the proof of Theorem~\ref{thm-LDP}(i).

For the rest of this section, we fix the set $\Omega_1 \subset \Omega$ of full $\P$-measure on which we do our quenched investigations:
\begin{equation}\label{Omega1def}
\begin{aligned}
\Omega_1 &= \Big\{\omega \in \Omega \colon X_i(\omega) \in W_{\mathbb B}\ \forall i\in\N,\\
&\qquad \lim_{\lambda \to \infty} \frac{\#\lbrace i \in I^\lambda(\omega) \colon X_i(\omega) \in W^\delta_j \rbrace}{\lambda}=\mu(W^\delta_j),~\forall j=1,\ldots,\delta^{-d},~\forall \delta \in \mathbb B\Big\}.
\end{aligned}
\end{equation}
That $\mathbb P(\Omega_1)=1$ holds follows immediately from the Restriction Theorem \cite[Section 2.2]{K93} combined with the Poisson Law of Large Numbers \cite[Section 4.2]{K93} and the fact that $\mu$ is absolutely continuous.

\subsection{The asymptotics of the combinatorics}\label{sec-asymptotics}

Let us fix a controlled standard setting $\underline{\Psi}$ as in \eqref{standardsetting}. Fix any $\omega \in \Omega_1$, and let the quantities $I^\lambda$ and $X^\lambda$ refer to this $\omega$. 
Denote
\begin{equation}\label{N0def}
N^0_{\delta,\lambda}(\underline\Psi)=\prod_{i=1}^{\delta^{-d}} \prod_{k=1}^{\kmax} \prod_{l=1}^{k-1}  N(\lambda)^{\lambda \pi_l \nu_k^{\delta,\lambda}(W^\delta_i)}.
\end{equation}
For a measurable subset $V$ of $\R^d$ and $\nu,\widetilde \nu \in \Mcal(V)$,  let us write $H_V(\nu|\widetilde \nu) = \int_V \d \nu \log \smfrac{\d \nu}{\d \widetilde \nu}$ if the density $\smfrac{\d \nu}{\d \widetilde \nu}$ exists and $H_V(\nu|\widetilde \nu) =\infty$ otherwise. (The difference between $H_V(\nu|\widetilde \nu)$ and the relative entropy $\Hcal_V(\nu|\widetilde \nu)$ defined in \eqref{relativeentropy} is the additive term $\widetilde \nu(V)-\nu(V)$.) Let us recall $M = \sum_{k=1}^{\kmax} \sum_{l=1}^{k-1} \pi_l \nu_k = \sum_{m=0}^{\infty} m \mu_m$ from \eqref{constraints} and $c_m=\exp(-1/(\e\mu(W))(\e\mu(W))^{-m}/m! $ from \eqref{quenchedentropy}. Note that the rate function $\mathrm I$ defined in \eqref{quenchedentropy} has also the representation
\begin{equation} \label{I-proof}
 \mathrm I(\Psi)= \sum_{k=1}^{\kmax}H_{W^k}(\nu_k\mid\mu^{\otimes k}) -H_{W}( M | \mu )+\sum_{m=0}^{\infty} H_{W}(\mu_m \mid \mu c_m)-\frac1\e,
\end{equation}
which we are going to use here. The equivalence between \eqref{I-proof} and \eqref{quenchedentropy} will be verified in Section~\ref{sec-entropyrepresentation} of the Appendix. Recall \eqref{mumentropymassidentity}, which implies that the third term in \eqref{I-proof} is invariant under replacing $H$ by $\mathcal H$. We now identify the large-$\lambda$ exponential rate of  the cardinality of $J^{\delta,\lambda}(\underline{\Psi})$ both on the scale $\lambda\log\lambda$ and $\lambda$:

\begin{prop}[Exponential rates of counting terms]\label{prop-combinatorics}
Let $\underline{\Psi}$ be a controlled standard setting. Let us write $\Psi=((\nu_k)_{k=1}^{\kmax},(\mu_m)_{m=0}^{\infty})$. We have
\[ \lim_{\delta \downarrow 0} \lim_{\lambda \to \infty} \frac{1}{\lambda} \log \frac{\# J^{\delta,\lambda}(\underline{\Psi})}{N^0_{\delta,\lambda}(\underline\Psi)} =-\mathrm I(\Psi), \]
as an identity in $[0,\infty]$. Moreover if $\mathrm I(\Psi)<\infty$, then
\begin{align*} & \lim_{\delta \downarrow 0} \lim_{\lambda \to \infty} \frac{1}{\lambda\log\lambda} \log \# J^{\delta,\lambda}(\underline{\Psi}) = M(W)<\infty, 
\end{align*}
almost surely.
\end{prop}

\begin{proof}
Recall that $\Psi$ is an admissible trajectory setting, according to Lemma \ref{lemma-constraints}. In particular, $\mathrm I(\Psi) \in (-\infty,\infty]$ is well-defined.

We use Stirling's formula $\lambda !=(\lambda/\e)^\lambda \e^{o(\lambda)}$ in the limit $\N \ni \lambda \to \infty$, which leads to
\begin{equation}\label{Stirling}
\lim_{\lambda\to\infty}\frac 1\lambda\log \binom {a^{\ssup\lambda}}{a^{\ssup\lambda}_1,\dots, a^{\ssup\lambda}_n}=-\sum_{i=1}^n a_i\log \frac {a_i}a,
\end{equation}
for any integers $a^{\ssup\lambda}_1,\dots, a^{\ssup\lambda}_n$ that sum up to $a^{\ssup\lambda}$ and satisfy $\frac1\lambda a_i^{\ssup\lambda}\overset{\lambda\to\infty}{\to}a_i$ for $i=1,\dots,n$ with positive numbers $a_1,\dots,a_n$ satisfying  $\sum_{i=1}^n a_i=a$.

From \eqref{comb-multinomial} we obtain that
\begin{align*}
   I^1_\delta(\underline\Psi)&=-  \lim_{\lambda\to\infty}\frac 1\lambda\log N^1_{\delta,\lambda}( \underline\Psi) \\ 
   & =\sum_{i=1}^{\delta^{-d}} \sum_{k=1}^{\kmax} \sum_{i_1,\ldots,i_{k-1}=1}^{\delta^{-d}}  \nu_k^{\delta} (W^\delta_i \times W^\delta_{i_1} \times \ldots \times W^\delta_{i_{k-1}}) \log \frac{\nu_k^{\delta} (W^\delta_i\times W^\delta_{i_1} \times \ldots \times W^\delta_{i_{k-1}})}{\mu^{\delta}(W^\delta_i) },
\end{align*}
where we also used that all the measures $ \nu_k^{\delta,\lambda}$ and $\mu^{\delta,\lambda}$ converge as $\lambda\to\infty$ to  $ \nu_k^{\delta}$ and $\mu^{\delta}$, respectively.

Now we add the term $\prod_{l=1}^{k-1} \mu^{\delta}(W^\delta_{i_{l}})$  both in the numerator and the denominator under the logarithm and separate these two terms. In the former, we write its logarithm as $\sum_{l=1}^{k-1} \log \mu^{\delta}(W^\delta_{i_{l}})$, interchange this sum on $l$ with all the other sums on the $i_0,\dots,i_{k-1}$ and write the sums over $i_0,\ldots,i_{l-1}, i_{l+1},\ldots,i_{k-1}$ in terms of the $l$-th marginal measure of $\nu_k^{\delta}$. This gives
\begin{align*}
I^1_\delta(\underline\Psi)&= \sum_{i=1}^{\delta^{-d}} \sum_{k=1}^{\kmax} \sum_{i_1,\ldots,i_{k-1}=1}^{\delta^{-d}} \nu_k^{\delta} (W^\delta_i \times W^\delta_{i_1} \times \ldots \times W^\delta_{i_{k-1}})   \log \frac{\nu_k^{\delta} (W^\delta_i\times W^\delta_{i_1} \times \ldots \times W^\delta_{i_{k-1}})}{\mu^{\delta}(W^\delta_i) \prod_{l=1}^{k-1} \mu^{\delta}(W^\delta_{i_{l}})}  \\ 
& \quad +\sum_{i=1}^{\delta^{-d}} \sum_{k=1}^{\kmax} \sum_{l=1}^{k-1} \pi_l \nu_k^{\delta}(W^\delta_i) \log \mu^{\delta}(W^\delta_i)  . \numberthis\label{comb-largelambdastartingpoint}
\end{align*}
In the same way as for $I_1^\delta$, we obtain
\begin{align*}
I^2_\delta(\underline\Psi)&=-  \lim_{\lambda\to\infty}\frac 1\lambda\log N^2_{\delta,\lambda}(\underline\Psi) =\sum_{i=1}^{\delta^{-d}} \sum_{m=0}^{\infty} \mu_m^\delta(W^\delta_i) \log \frac{ \mu_m^\delta(W^\delta_i)}{\mu^\delta(W^\delta_i)}. 
\numberthis\label{N2exprate}
\end{align*}

Using \eqref{Omega1def}, on $\Omega_1$ we have that the asymptotic behaviour of \eqref{N0def} is the following
\[ N^0_{\delta,\lambda}(\underline\Psi)=N(\lambda)^{\lambda \sum_{i=1}^{\delta^{-d}} \sum_{k=1}^{\kmax} \sum_{l=1}^{k-1} \pi_l \nu_k^{\delta,\lambda}(W^\delta_i)} 
= {(\lambda \mu(W))}^{\lambda (1+o(1)) \sum_{i=1}^{\delta^{-d}} \sum_{k=1}^{\kmax} \sum_{l=1}^{k-1} \pi_l \nu_k^{\delta,\lambda}(W^\delta_i)}. \] 
On the other hand, also by Stirling's formula, we can identify the large-$\lambda$ rate of the quotient of the counting terms in \eqref{comb-lambdaloglambdaterm} and \eqref{N0def} as follows:
\begin{equation}\label{comb-largelambdarelay}
\begin{aligned}
I^{3,0}_\delta(\underline \Psi)&=-\lim_{\lambda\to\infty}\frac 1\lambda\log \frac{N^3_{\delta,\lambda}(\underline\Psi)}{N^0_{\delta,\lambda}(\underline\Psi)}\\
&=-\lim_{\lambda\to\infty}\frac 1\lambda\log \prod_{i=1}^{\delta^{-d}} \frac{\big(\smfrac 1{\e \mu(W)}\sum_{k=1}^{\kmax} \sum_{l=1}^{k-1} \pi_l \nu_k^{\delta,\lambda} (W^\delta_{i})\big)^{\lambda \sum_{k'=1}^{\kmax} \sum_{l'=1}^{k'-1} \pi_{l'} \nu_{k'}^{\delta,\lambda} (W^\delta_{i})}}{\prod_{m=0}^{\infty} m!^{\lambda \mu_m(W^\delta_i)}}\\ 
&=-\sum_{i=1}^{\delta^{-d}} \sum_{k'=1}^{\kmax} \sum_{l'=1}^{k'-1} \pi_{l'} \nu_{k'}^{\delta} (W^\delta_{i}) \left(\log\sum_{k=1}^{\kmax} \sum_{l=1}^{k-1} \pi_l \nu_k^{\delta} (W^\delta_{i}) -(1+\log \mu(W)) \right) \\ 
& \qquad + \sum_{i=1}^{\delta^{-d}} \sum_{m=0}^{\infty}  \mu_m^{\delta}(W^\delta_i)\,\log(m!),
\end{aligned}
\end{equation}
where for the last term we used the fact that $\underline{\Psi}$ is controlled (see also Lemma \ref{lemma-constraints}), together with dominated convergence. We can summarize the sum of the terms in \eqref{comb-largelambdastartingpoint}, \eqref{N2exprate} and \eqref{comb-largelambdarelay} as
\begin{equation}\label{Jexprate}
\begin{aligned}
-\lim_{\lambda\to\infty}\frac1\lambda\log \frac{\# J^{\delta,\lambda}(\underline{\Psi})}{N^0_{\delta,\lambda}(\underline\Psi)}
&=I^{1}_\delta(\underline \Psi)+I^{2}_\delta(\underline \Psi)+I^{3,0}_\delta(\underline \Psi)\\
&= \sum_{k=1}^{\kmax} \sum_{i_0,\ldots,i_{k-1}=1}^{\delta^{-d}} \nu_k^{\delta} (W^\delta_{i_0} \times \ldots \times W^\delta_{i_{k-1}})   \log \frac{\nu_k^{\delta} (W^\delta_{i_0} \times \ldots \times W^\delta_{i_{k-1}})}{\prod_{l=0}^{k-1} \mu^{\delta}(W^\delta_{i_{l}})}  
\\ &\quad +\sum_{i=1}^{\delta^{-d}} \sum_{m=0}^{\infty} \mu_m^\delta(W^\delta_i) \Big( \log \frac{ \mu_m^\delta(W^\delta_i)}{\mu^\delta(W^\delta_i)} + m(1+\log \mu(W))+\log(m!) \Big) \\
&\quad- \sum_{i=1}^{\delta^{-d}} \left( \sum_{k=1}^{\kmax} \sum_{l=1}^{k-1} \pi_{l} \nu_{k}^{\delta} (W^\delta_{i})\right) \log \frac{\sum_{k=1}^{\kmax} \sum_{l=1}^{k-1} \pi_l \nu^\delta_k (W^\delta_{i})}{\mu^{\delta}(W^\delta_i)},
\end{aligned}
\end{equation}
where in the first line on the right-hand side we changed the summing index $i$ into $i_0$. Since we have \[ \sum_{m=0}^{\infty} \mu_m^\delta(W)=\sum_{m=0}^{\infty}\mu_m(W)=\mu(W), \] and thus
\[ \sum_{i=1}^{\delta^{-d}} \sum_{m=0}^{\infty} \mu_m^\delta(W^\delta_i) \Big( \log \frac{ \mu_m^\delta(W^\delta_i)}{\mu^\delta(W^\delta_i)} + m(1+\log \mu(W))+\log(m!) \Big) = \sum_{m=0}^{\infty} \mu_m^\delta(W^\delta_i) \log \frac{ \mu_m^\delta(W^\delta_i)}{c_m \mu^\delta(W^\delta_i)} -\frac1\e, \] 
we obviously arrived at the discrete version of the entropy terms in \eqref{I-proof}, more precisely, the entropy of the measures in \eqref{I-proof} with respect to the $\sigma$-field $\Fcal_\delta$, respectively $\Fcal^{\otimes k}_\delta$. Now, according to \cite[Proposition (15.6)]{G11}, the limit of these entropies as $\delta\downarrow 0$ is equal to their corresponding continuous version, i.e., the right-hand side of \eqref{Jexprate} converges to ${\rm I}(\Psi)$. The first part of Proposition \ref{prop-combinatorics} follows.

Moreover, if $\mathrm I(\Psi)<\infty$, then we have by continuity
\begin{align*} 
& \lim_{\delta \downarrow 0} \lim_{\lambda \to \infty} \frac{1}{\lambda \log \lambda} \log \# J^{\delta,\lambda}( \underline{\Psi} ) = \lim_{\delta \downarrow 0} \lim_{\lambda \to \infty} \frac{1}{\lambda \log \lambda} \log N^0_{\delta.\lambda}( \underline\Psi )\\ 
& =  \lim_{\delta \downarrow 0} \lim_{\lambda \to \infty}   \sum_{k=1}^{\kmax} \sum_{l=1}^{k-1} \sum_{i=1}^{\delta^{-d}} \pi_l \nu_k^{\delta,\lambda}(W^\delta_i) = \sum_{k=1}^{\kmax} \sum_{l=1}^{k-1} \pi_l \nu_k (W)=\sum_{k=1}^{\kmax} (k-1) \nu_k(W^k) \in [0,\infty),
 \end{align*}
where in the last identity we used that by Fubini's theorem, $\pi_0\nu_k(W)=\nu_k(W^k)$ holds for all $k$. 
Hence, using that $\Psi$ is an admissible trajectory setting, we conclude the second part of Proposition \ref{prop-combinatorics}. 
\end{proof}

\subsection{Approximations for the penalization terms} \label{sec-SIRproof} 

The limiting relations between the penalization terms depending on the numbers of incoming hops in \eqref{Menergy} and \eqref{limitbottleneck}, and between the continuous penalization terms in \eqref{SIRenergy} and \eqref{limitSIR} are given as follows.

\begin{prop} \label{prop-SIRcomputations}
Let $\underline{\Psi}$ be a controlled standard setting. Let us write $\Psi=((\nu_k)_{k=1}^{\kmax},(\mu_m)_{m=0}^{\infty})$ for the admissible trajectory setting contained in $\underline\Psi$.  Then, almost surely,
\[ \lim_{\delta\downarrow0}\lim_{\lambda \to \infty}\sup_{s\in J^{\delta,\lambda}(\underline{\Psi})}  \Big|\frac{1}{\lambda} \mathfrak{M}(s) - \mathrm M(\Psi)\Big| =0,\numberthis\label{limitBottlemu0} \]
and
\[ \lim_{\delta\downarrow0}\lim_{\lambda \to \infty}\sup_{s\in J^{\delta,\lambda}(\underline{\Psi})}  \Big|\frac{1}{\lambda} \mathfrak{S}(s) - \mathrm S(\Psi)\Big| =0.\numberthis\label{limitSIRmu0} \]
\end{prop}

\begin{proof} Throughout the proof, we perform our analysis on $\Omega_1$. First, we show \eqref{limitBottlemu0}. Consider some $s\in J^{\delta,\lambda}(\underline{\Psi})$ for $\lambda>0$ and $\delta\in\mathbb B$. Additionally assume that $s^i_l \in W_{\mathbb B}$ for all $i \in I^\lambda$ and $l=0,\ldots,k$ (which is always the case for $s=S=(S^i)_{i \in I^\lambda}$ on $\Omega_1$).

Then $P_\lambda^\delta(s)=\mu^{\delta,\lambda}$ and $P_{\lambda,m}^{\delta}(s)=\mu_m^{\delta,\lambda}$ for all $m\in\N_0$, see the definition \eqref{numberofconf} of $J^{\delta,\lambda}(\underline{\Psi})$,  \eqref{standardsetting} and \eqref{mu-explanation}. Recall that $m_i(s)$ is the number of ingoing messages at relay $X_i$ for the trajectory configuration $s$. Hence we have
$$
\begin{aligned}
\mathfrak{M}(s) & = \sum_{i \in I^\lambda} \eta(m_i(s)) = \sum_{m=0}^{\infty}\eta(m) \#\{i \in I^\lambda\colon m_i(s)=m\}  =  \sum_{m=0}^{\infty}\eta(m) P_{\lambda,m}(s)(W)\\
&= \sum_{m=0}^{\infty}\eta(m) P_{\lambda,m}^\delta(s)(W) = \lambda \sum_{m=0}^{\infty} \eta(m) \mu_m^{\delta,\lambda}(W),  
\end{aligned} 
$$
for all such $s$. Note that by part \eqref{mumconv-standardsetting} of Definition~\ref{defn-standardsetting} and part \eqref{mumdelta-standardsetting} of Remark \ref{remark-standardsetting}, we obtain that $\mu_m^{\delta,\lambda}$ tends to $\mu_m$ as $\lambda \to \infty$ followed by $\delta \downarrow 0$.  Now, \eqref{third-controlleddef} in Definition \ref{defn-controlledstandardsetting}, together with the fact that the total mass of $\mu_m^\delta$ equals the one of $\mu_m$ for any $m$, implies the assertion in \eqref{limitBottlemu0}.

We continue with verifying \eqref{limitSIRmu0}. Let us fix an arbitrary controlled standard setting $\underline{\Psi}$. Our goal is to prove that \eqref{limitSIRmu0} holds for this $\underline{\Psi}$. Using  that,  for an admissible trajectory setting $\Psi=((\nu_k)_{k=1}^{\kmax},(\mu_m)_{m=0}^{\infty})$, $\mathrm S(\Psi)$ depends only on $(\nu_k)_{k=1}^{\kmax}$, we have for any $\lambda>0$, $\delta\in \mathbb B$, $s \in J^{\delta,\lambda}(\underline \Psi)$ and $k \in [\kmax]$ 
\[ \frac{1}{\lambda} \mathfrak S(s) - \mathrm S(\Psi) = \langle  R_{\lambda,k}(s) , f_k(L_\lambda,\cdot) \rangle - \langle \nu_k, f_k(\mu,\cdot) \rangle.  \]

In the rest of Section~\ref{sec-prooflargelambda}, we will often have to verify convergence of certain (sequences of) measures in the (coordinatewise) weak topology. In order to keep our arguments clear and short, for $k \in \mathbb N$, we fix a metric $d_k(\cdot,\cdot)$ on $\mathcal{M}(W^k)$ that generates the weak topology on this space. It turns out to be convenient to choose $d_k$ to be the {\em Lipschitz bounded metric} \cite[Section D.2]{DZ98} on $\mathcal{M}(W^k)$, that is,
\[ d_k(\nu_k^1, \nu_k^2) = \sup \lbrace | \langle \nu_k^1, f \rangle - \langle \nu_k^2, f \rangle | \colon f \in \mathrm{Lip}_1(W^k) \rbrace \numberthis\label{Lipschitzboundedmetric} \]
for all $k$, where $\mathrm{Lip}_1(W^k)$ is the set of Lipschitz continuous functions taking $W^k$ to $\mathbb R$ with Lipschitz parameter less than or equal to 1 and with uniform bound 1. We have for $k \in [\kmax]$
\begin{align*} 
\Big| \frac{1}{\lambda} \mathfrak S(s) - \mathrm S(\Psi) \Big| &= \Big| \langle R_{\lambda,k}(s) , f_k(L_\lambda,\cdot) \rangle - \langle \nu_k , f_k(\mu,\cdot) \rangle \Big| \\ 
 &\leq  \Big| \langle R_{\lambda,k}(s), f_k(L_\lambda,\cdot) \rangle - \langle \nu_k^{\delta,\lambda}, f_k(L_\lambda,\cdot) \rangle \Big| + \Big|  \langle \nu_k^{\delta,\lambda}, f_k(L_\lambda,\cdot) \rangle  -  \langle \nu_k^{\delta,\lambda}, f_k(L_\lambda^\delta,\cdot) \rangle  \Big| \\ 
 & \qquad +  \Big| \langle \nu_k^{\delta,\lambda}, f_k(L_\lambda^\delta,\cdot) \rangle - \langle \nu_k^{\delta,\lambda}, f_k(\mu,\cdot) \rangle \Big| + \Big| \langle \nu_k^{\delta,\lambda}, f_k(\mu,\cdot) \rangle-\langle \nu_k, f_k(\mu,\cdot) \rangle  \Big|. \numberthis\label{fourterms} 
 \end{align*}
Now, we claim that all the four terms on the right-hand side tend to 0 in the limit $\lambda \to \infty$ followed by $\delta \downarrow 0$. Indeed, for the first term, let $g \in \mathrm{Lip}^1(W^k)$. Then we have
\begin{align*} 
\Big| \langle R_{\lambda,k}(s), & g \rangle - \langle \nu_k^{\delta,\lambda}, g \rangle
 = \Big| \int_{W^k} g(y) R_{\lambda,k}(s)(\d y)-\int_{W^k} g(y) R_{\lambda,k}^\delta(s) (\d y)\Big| \\ 
 &\leq  \sum_{i_0,\ldots,i_{k-1}=1}^{\delta^{-d}} \sup_{y,z \in W^\delta_{i_0} \times\ldots\times W^\delta_{i_{k-1}}} |g(y)-g(z)| R_{\lambda,k}(s)(W^\delta_{i_0} \times \ldots \times W^\delta_{i_{k-1}}) \\
 &\leq \delta \sqrt{dk} R_{\lambda,k}(s)(W) \leq  \delta \sqrt{dk} L_\lambda (W), 
 \end{align*}
which tends to 0 as $\lambda \to \infty$ followed by $\delta \downarrow 0$. It follows that $R_{\lambda,k}(s)-\nu_k^{\delta,\lambda}$ tends weakly to 0 as $\lambda \to \infty$ followed by $\delta \downarrow 0$. We note that for any $\alpha>0$, the restriction of $f_k$ to $\Mcal_{\leq \alpha}(W) \times W^k$ is bounded, where we wrote $\Mcal_{\leq \alpha}(V)$ for the set of measures on the space $V$ with total mass $\leq \alpha$. Indeed, since $W$ is compact, $\Mcal_{\leq \alpha}(W)$ with the weak topology is also a compact, metrizable space by Prohorov's theorem. Thus, the continuous function $f_k \colon \Mcal_{\leq \alpha}(W) \times W^k \to \R$ is uniformly continuous, and therefore it is bounded. Now, since eventually $L_\lambda \in  \Mcal_{\leq 2 \mu(W)}(W)$, the first term on the right-hand side of \eqref{fourterms} tends to 0.

As for the second term, note that for any $\delta\in \mathbb B$, $L_\lambda-L_\lambda^\delta$ tends to $\mu-\mu^\delta$ as $\lambda \to \infty$, which tends to 0 as $\delta \downarrow 0$. Thus, by the fact that $f_k$ is continuous and bounded on $\Mcal_{\leq 2 \mu(W)}(W)  \times W^k$ and eventually $L_\lambda,L_\lambda^\delta,\nu_k^{\delta,\lambda} \in \Mcal_{\leq 2 \mu(W)}(W)$ for all $\delta \in \mathbb B$, the second term also tends to 0 as first $\lambda\to\infty$ and afterwards $\delta\downarrow 0$. An analogous argument applies for the third term, using that $L_\lambda^\delta$ converges to $\mu$ as first $\lambda \to \infty$ and then $\delta \downarrow 0$ by part \eqref{muconv-standardsetting} of Remark \ref{remark-standardsetting} and the definition of $\mu^\delta$, $\delta \in \mathbb B$. The fourth term tends to zero since it easily follows from part \eqref{nukconv-standardsetting} of Definition~\ref{defn-standardsetting} and part \eqref{nukdelta-standardsetting} of Remark~\ref{remark-standardsetting} that $\nu_k^{\delta,\lambda}$ converges weakly to $\nu_k$, and $f_k$ is 
continuous as bounded on $\Mcal_{\leq 2 \mu(W)}(W)  \times W^k$. We conclude \eqref{limitSIRmu0} and Proposition \ref{prop-SIRcomputations}.
\end{proof}

\subsection{Existence of standard settings}\label{sec-admtrajetstandardsetting}
Recall that we equip $\mathcal{A}$ defined in \eqref{Adef} with the product topology of the weak topologies of the factors $\Mcal(W^k)$, $\mathcal M(W)$, and that this is the topology of coordinatewise weak convergence. For $k \in \mathbb N$, let $d_k(\cdot,\cdot)$ be the Lipschitz bounded metric  \eqref{Lipschitzboundedmetric} on $\mathcal{M}(W^k)$, which generates the weak topology on this space. Then, 
\[ d_0(\Psi^1,\Psi^2) = \sum_{k=1}^{\kmax} d_k(\nu_k^1,\nu_k^2) + \sum_{m=0}^{\infty} 2^{-m} d_1(\mu_m^1,\mu_m^2), \quad \Psi^1, \Psi^2 \in \mathcal A, \numberthis\label{dmetric}\]
is a metric on $\mathcal{A}$ that generates the product topology. For $\varrho>0$ and $\Psi \in \Acal$, let us write $B_\varrho(\Psi)=\lbrace \Psi' \in \mathcal{A} \colon d_0(\Psi',\Psi)<\varrho \rbrace $ for the open $\varrho$-ball around $\Psi$. We have the following.

\begin{prop}\label{prop-standardsettingconstruction}
On $\Omega_1$, for any admissible trajectory setting (see Definition~\ref{def-admtrajet}), $\Psi=((\nu_k)_k,(\mu_m)_m)$, there exists a standard setting $\underline{\Psi}$ containing it. If $\sum_m \eta(m) \mu_m(W)<\infty$, then $\underline{\Psi}$ can be chosen to be a controlled standard setting.
\end{prop}
\begin{proof}
We fix an admissible trajectory setting $\Psi$ and construct $\underline \Psi$ as follows. As is required in Definition \ref{defn-standardsetting}, the measures $\nu_k^\delta$ for $k \in [\kmax]$ and $\mu_m^\delta$ for $m \in \mathbb N_0$ are the $\delta$-coarsenings of the measures $\nu_k$ and $\mu_m$, respectively, and $\mu^{\delta,\lambda}=L_\lambda^{\delta}$. Now for $\delta \in \mathbb B$ and $\lambda>0$, pick some measures $\nu_k^{\delta,\lambda}$ and $\mu_m^{\delta,\lambda}$ with values in $\frac 1\lambda\N_0$ such that the requirements \eqref{i-standardsetting} $\sum_{k=1}^{\kmax} \pi_0 \nu_k^{\delta,\lambda}=\mu^{\delta,\lambda}$, \eqref{ii-standardsetting} $\sum_{m=0}^{\infty} \mu_m^{\delta,\lambda}=\mu^{\delta,\lambda}$ and \eqref{iii-standardsetting} $\sum_{m=0}^{\infty} m \mu_m^{\delta,\lambda} = \sum_{k=1}^{\kmax} \sum_{l=1}^{k-1} \pi_l \nu_k^{\delta,\lambda}$ of Definition~\ref{defn-standardsetting} are met, such that $\nu_k^{\delta,\lambda}\Longrightarrow \nu_k^{\delta}$ 
and $\mu_m^{\delta,\lambda}\Longrightarrow \mu_m^{\delta}$ as $\lambda\to\infty$ and such that the collection $\underline \Psi$ of all these measures is a standard setting containing $\Psi$, which is controlled if $\sum_m \eta(m) \mu_m(W)<\infty$. 

We claim that this can be done by taking suitable up- and downroundings of the numbers
\[ \nu_k'^{\delta,\lambda}(W^\delta_{i_0}\times\ldots\times W^\delta_{i_{k-1}}) =\nu_k^\delta(W^\delta_{i_0}\times\ldots\times W^\delta_{i_{k-1}}) \frac{L_{\lambda}^{\delta}(W^\delta_{i_0})}{\mu^{\delta}(W^\delta_{i_0})} \mathds 1 \lbrace \mu^{\delta}(W^\delta_{i_0}) >0 \rbrace, \quad k \in [\kmax],\numberthis\label{constructionagain} \] 
for all $i_0,\ldots,i_{k-1}=1,\ldots,\delta^{-d}$, and dividing by $\lambda$, analogously for the $\mu_m$'s. Now, using the $d$-metric defined in \eqref{dmetric}, we prove that the convergences required in Definition~\ref{defn-standardsetting} hold for such $\underline \Psi$. 

First, we prove the convergence of the $\delta$-coarsenings $\Psi^\delta=((\nu_k^\delta)_k,(\mu_m^\delta)_m)$ to $\Psi$ in the $d_0$-metric. We claim that for any $\varrho>0$, there exists $\delta_0 \in \mathbb B$ such that $\Psi^\delta\in B_\varrho(\Psi)$ for all $\mathbb B \ni \delta \leq \delta_0$. Indeed, for $k \in [\kmax]$, $\nu_k \in \mathcal{M}(W^k)$ and $\delta \in \mathbb B$ we see that the distance between $\nu_k$ and its $\delta$-coarsening is of order $\delta$ with respect to the Lipschitz bounded metric:
\[
\begin{aligned} 
d_k(\nu_k,\nu_k^\delta)
= &  \sup_{f \in \mathrm{Lip}_1(W^k)} \Big|  \sum_{i_0,\ldots,i_{k-1}=1}^{\delta^{-d}} \Big( \int_{W^\delta_{i_0} \times \ldots \times W^\delta_{i_{k-1}}} f(x) \nu_k(\d x)-\int_{W^\delta_{i_0} \times \ldots \times W^\delta_{i_{k-1}}} f(x) \nu_k^{\delta}(\d x) \Big) \Big| 
\\ \leq & \sup_{f \in \mathrm{Lip}_1(W^k)} \sum_{i_0,\ldots,i_{k-1}=1}^{\delta^{-d}}  \sup_{x,y\in W^\delta_{i_0} \times \ldots \times W^\delta_{i_{k-1}}} |f(x)-f(y)| \nu_k(W^\delta_{i_0} \times \ldots \times W^\delta_{i_{k-1}}) \leq \delta\nu_k(W^k) \sqrt{kd},
\end{aligned}
\]
where we wrote $x=(x_0,\ldots,x_{k-1})$; and analogously for $\mu_m$. Thus, we have
\[ d_0(\Psi,\Psi^\delta) \leq \delta \sqrt{d}\Big[ \sum_{k=1}^{\kmax} \nu_k(W^k) \sqrt{k} + \sum_{m=0}^{\infty} \mu_m(W) 2^{-m} \Big]. \]
Since $\sum_{m=0}^{\infty} \mu_m(W)<\infty$ by (ii) in \eqref{constraints}, there exists a constant $C$, only depending on $\Psi$, such that $\Psi^\delta \in B_\varrho(\Psi)$ for any $\delta\leq C \varrho$.

Second, we ignore the up- or downroundings in the construction of $\Psi$ and prove the following. For $\delta \in \mathbb B$ and $\lambda>0$, let $\Psi'^{\delta,\lambda}$ be the collection of the measures introduced in \eqref{constructionagain}. We claim that on $\Omega_1$, we have
\[ \limsup_{\lambda \to \infty} d_0(\Psi^\delta,\Psi'^{\delta,\lambda})=0. \]
Indeed, for any $k \in [\kmax]$ and $i_0,\ldots,i_{k-1}=1,\ldots,\delta^{-d}$, $d_k(\nu_k^{\delta},\nu_k'^{\delta,\lambda}) $ is bounded from above by
\[ \sup_{f \in \mathrm{Lip}_1(W^k)}  \sum_{i_0,\ldots,i_{k-1}=1}^{\delta^{-d}} \nu_k^\delta(W^\delta_{i_0}\times\ldots\times W^\delta_{i_{k-1}}) \Big| \frac{L_{\lambda}^{\delta}(W^\delta_{i_0})}{\mu^{\delta}(W^\delta_{i_0})} -1 \Big| \Vert f \Vert_{\infty} \leq \nu_k^{\delta}(W^k) \max_{i_0=1}^{\delta^{-d}} \Big| \frac{L_{\lambda}^{\delta}(W^\delta_{i_0})}{\mu^{\delta}(W^\delta_{i_0})} -1 \Big|.  \numberthis\label{standardagain} \]
Similarly, for any $\delta \in \mathbb B$ and $m \in \N_0$, $d_1(\mu_m^\delta,\mu_m'^{\delta,\lambda})$ vanishes in the limit $\lambda \to \infty$. Thus,
\[ d_0(\Psi^\delta,\Psi'^{\delta,\lambda}) \leq \Big( \sum_{k=1}^{\kmax}   \nu_k^{\delta}(W^k) + \sum_{m=0}^{\infty} 2^{-m} \mu_m^{\delta}(W) \Big) \max_{i_0=1}^{\delta^{-d}} \Big| \frac{L_{\lambda}^{\delta}(W^\delta_{i_0})}{\mu^{\delta}(W^\delta_{i_0})} -1 \Big|, \]
which tends to 0 on $\Omega_1$ as $\lambda \to \infty$, according to \eqref{Omega1def}.

Now, if we add the suitable up- and downroundings, we only change distances in the $d$-metric by an error term of order $1/\lambda$, which vanishes as $\lambda \to \infty$. This implies that $\underline \Psi$ is a standard setting. It also follows easily that if $\sum_m \eta(m) \mu_m(W)<\infty$, then $\underline \Psi$ is controlled.
\end{proof}

\subsection{Proof of Theorem \ref{theorem-variation1}}\label{sec-varproof}
Abbreviate  
\[ 
\mathfrak{Y}(r)=\Big(\prod_{i\in I^\lambda}N(\lambda)^{- (r^i_{-1}-1)}\Big) \exp\Big\{-\gamma \mathfrak S(r)-\beta\mathfrak M(r)\Big\},\qquad \lambda>0, r \in \mathcal{S}_{\kmax}(X^\lambda),
\]and note that the partition function is given as 
\begin{equation}\label{Zrepres}
Z^{\gamma,\beta}_\lambda(X^\lambda) =\sum_{r \in \mathcal{S}_{\kmax}(X^\lambda)}\mathfrak Y(r).
\end{equation}
Then Theorem \ref{theorem-variation1} says that its large-$\lambda$ negative exponential rate is given as the infimum of $\mathrm I(\Psi)+\gamma \mathrm S(\Psi)+\beta \mathrm M(\Psi)$, taken over all admissible trajectory settings $\Psi$. Throughout the proof, we assume that the configuration $X^\lambda=X^\lambda(\omega)$ comes from some $\omega\in\Omega_1$ defined in \eqref{Omega1def}.

Having proved Propositions \ref{prop-combinatorics}, \ref{prop-SIRcomputations} and \ref{prop-standardsettingconstruction}, our strategy to prove Theorem \ref{theorem-variation1} is the following. First, Proposition \ref{prop-standardsettingconstruction} gives a standard way of constructing from an admissible trajectory setting $\Psi$ satisfying $\mathrm I(\Psi)+\gamma \mathrm S(\Psi)+\beta \mathrm M(\Psi)<\infty$ a controlled standard setting $\underline\Psi$ that contains $\Psi$. Then the lower bound for the partition function is easily given in terms of the objects that are contained in any such $\underline \Psi$ and using the logarithmic asymptotics for their combinatorics from Propositions \ref{prop-combinatorics} and \ref{prop-SIRcomputations} and finally taking the infimum over all such $\Psi$, respectively $\underline \Psi$. The upper bound needs more care, since the entire sum over $r$ has to be handled. First of all, we show that the sum can be restricted for all $\lambda>0$, modulo some error term 
that is negligible on the exponential scale, to the sum of those configurations whose 
congestion exponent is at most $C\lambda$ for some appropriate large constant $C>0$. This sum can be decomposed, for any $\delta\in \mathbb B$, to sums on configurations coming from a particular choice of empirical measures on the $\delta$-partitions of $W$. The number of these empirical measures and the sum on the partitions is negligible in the limit $\lambda\to\infty$, and the asymptotics of the sums on $r$ in these partitions can be evaluated with the help of our spatial discretization procedure, using arguments of the proofs of Propositions \ref{prop-combinatorics} and \ref{prop-SIRcomputations} in the limit $\lambda\to\infty$, followed by $\delta\downarrow 0$. Using these, we arrive at the formula \eqref{variation}.

Let us give the details. We start with the proof of the lower bound. For any admissible trajectory setting $\Psi$ such that $\mathrm I(\Psi)+\gamma \mathrm S(\Psi)+\beta \mathrm M(\Psi)<\infty$, we pick a controlled standard setting $\underline \Psi$ as in Proposition \ref{prop-standardsettingconstruction} and recall the configuration class $J^{\delta,\lambda}(\underline{\Psi})$ from \eqref{numberofconf}. Then, for any $\lambda>0$ and $\delta\in\mathbb B$,
\begin{equation}\label{afterconstruction}
Z^{\gamma,\beta}_\lambda(X^\lambda) \geq \sum_{r \in J^{\delta,\lambda}(\underline{\Psi})}  \mathfrak{Y}(r) 
\geq \frac{\# J^{\delta,\lambda}(\underline{\Psi})}{\sup_{r\in J^{\delta,\lambda}(\underline{\Psi})}\prod_{i\in I^\lambda}N(\lambda)^{- (r^i_{-1}-1)}}
\exp\Big\{- \sup_{r\in J^{\delta,\lambda}(\underline{\Psi})}\big(\gamma\mathfrak S(r)+\beta\mathfrak M(r)\big)\Big\}.
\end{equation}
Hence,
\begin{equation} \label{lowerbound-theorem1}
\begin{aligned}
\liminf_{\lambda\to\infty}\frac 1\lambda\log Z^{\gamma,\beta}_\lambda(X^\lambda) 
&\geq 
\liminf_{\delta\downarrow0}\liminf_{\lambda\to\infty}\frac 1\lambda\log \frac{\# J^{\delta,\lambda}(\underline{\Psi})}{\sup_{r\in J^{\delta,\lambda}(\underline{\Psi})}\prod_{i\in I^\lambda}N(\lambda)^{- (r^i_{-1}-1)}}\\
&\quad -\gamma \limsup_{\delta\downarrow0}\limsup_{\lambda\to\infty}\sup_{r\in J^{\delta,\lambda}(\underline{\Psi})}\frac 1\lambda \mathfrak S(r)-\beta \limsup_{\delta\downarrow0}\limsup_{\lambda\to\infty}\sup_{r\in J^{\delta,\lambda}(\underline{\Psi})}\frac 1\lambda \mathfrak M(r)\\
&=-\mathrm I(\Psi)-\gamma \mathrm S(\Psi)-\beta \mathrm M(\Psi).
\end{aligned}
\end{equation}
In the last step we also used  Propositions \ref{prop-combinatorics} and \ref{prop-SIRcomputations} together with the fact that $\underline{\Psi}$ is controlled. Now take the supremum over all such $\Psi$ on the r.h.s.~of \eqref{lowerbound-theorem1} to conclude that the lower bound in \eqref{variation} holds.

The upper bound of Theorem \ref{theorem-variation1} requires some additional work. We start from \eqref{Zrepres}. For $C>0$ we have 
\begin{equation}\label{Zdecompose}
Z^{\gamma,\beta}_\lambda(X^\lambda) = \sum_{r \in \mathcal{S}_{\kmax}(X^\lambda) \colon\mathfrak{M}(r) \leq \lambda C} \mathfrak Y(r) + \sum_{r \in \mathcal{S}_{\kmax}(X^\lambda) \colon\mathfrak{M}(r) > \lambda C} \mathfrak Y(r).
\end{equation}
Since the total mass of our {\it a priori} measure has a bounded large-$\lambda$ exponential rate (see Section \ref{modeldef-messagetraj}) and $\mathfrak S$, $\mathfrak M$ are bounded from below, we see that
\[ \limsup_{C \to \infty}\limsup_{\lambda \to \infty} \frac{1}{\lambda} \log \sum_{r \in \mathcal{S}_{\kmax}(X^\lambda) \colon\mathfrak{M}(r) > \lambda C} \mathfrak Y(r) =-\infty.  \]
Thus, for $C$ sufficiently large, the exponential rate of $Z^{\gamma,\beta}_\lambda(X^\lambda)$ is equal to the one of the first term on the right-hand side of \eqref{Zdecompose}. We additionally require $C$ so large that
\[ \inf_{\Psi\text{ adm. traj. setting},~\mathrm M(\Psi) \leq C} (\mathrm I(\Psi)+\gamma \mathrm S(\Psi)+\beta \mathrm M(\Psi)) = \inf_{\Psi\text{ adm. traj. setting}} (\mathrm I(\Psi)+\gamma \mathrm S(\Psi)+\beta \mathrm M(\Psi)). \numberthis\label{largeCinf} \]
Let us write $\mathcal{S}_{\kmax,C}(X^\lambda)=\lbrace r \in \mathcal{S}_{\kmax}(X^\lambda) \colon\mathfrak{M}(r) \leq \lambda C \rbrace$ and $Z^{\gamma,\beta,C}_\lambda(X^\lambda)=\sum_{r \in \mathcal{S}_{\kmax,C}(X^\lambda) } \mathfrak Y(r)$. The upper bound of Theorem \ref{theorem-variation1} follows as soon as we show that
\[  \limsup_{\lambda \to \infty} \frac{1}{\lambda} \log Z^{\gamma,\beta,C}_\lambda(X^\lambda)\leq -\inf_{\Psi\text{ admissible trajectory setting},~\mathrm M(\Psi) \leq C} (\mathrm I(\Psi)+\gamma \mathrm S(\Psi)+\beta \mathrm M(\Psi)). \numberthis\label{esistgenug} \]

For fixed $\lambda>0$ and $\delta \in \mathbb B$, let us say that a collection of measures $\Psi^{\delta,\lambda}=((\nu_k^{\delta,\lambda})_{k=1}^{\kmax},(\mu_m^{\delta,\lambda})_{m=0}^{\infty})$ lies in $G(\delta,\lambda)=G(\delta,\lambda)(X^\lambda)$ if all these measures take values in $\frac 1\lambda\N_0$ only and satisfy the constraints $\sum_{k=1}^{\kmax} \pi_0 \nu_k^{\delta,\lambda}=L_\lambda^\delta$, $\sum_{m=0}^{\infty} \mu_m^{\delta,\lambda}=L_\lambda^\delta$ and $\sum_{k=1}^{\kmax}\sum_{l=1}^{k-1}\pi_l \nu_k^{\delta,\lambda}=\sum_{m=0}^\infty m\mu_m^{\delta,\lambda}$. We will write $ J^{\delta,\lambda}(\Psi^{\delta,\lambda})$ for the set $J^{\delta,\lambda}(\underline\Psi)$ defined in \eqref{numberofconf}. Then the union of $J^{\delta,\lambda}(\Psi^{\delta,\lambda})$ over all $\Psi^{\delta,\lambda}$ with $\sum_{m=0}^{\infty} \eta(m) \mu_m^{\delta,\lambda}(W) \leq C$ is equal to
\[ \big\{ (R_{\lambda,k}^\delta(r))_{k\in[\kmax]}, (P^\delta_{\lambda,m}(r))_{m\in\N_0}) \colon r \in \mathcal{S}_{\kmax,C}(X^\lambda) \big\} ,\]  
since these three equations characterize the tuple of the measures $(R_{\lambda,k}^\delta(S))_{k=1}^{\kmax}$ and $(P_{\lambda,m}^{\delta}(S))_{m=0}^{\infty}$ if $(S^i)_{i \in I^\lambda} \in \mathcal{S}_{\kmax,C}(X^\lambda)$. 

Using this, we can estimate, for any $\delta\in\mathbb B$,
\begin{equation}
{Z}^{\gamma,\beta,C}_\lambda(X^\lambda) =\sum_{\Psi^{\delta,\lambda}\in G(\delta,\lambda) \colon \mathrm M(\Psi^{\delta,\lambda}) \leq C}~ \sum_{ r \in J^{\delta,\lambda}(\Psi^{\delta,\lambda})} \mathfrak{Y}(r)
\leq \#G(\delta,\lambda)\sup_{\Psi^{\delta,\lambda}\in G(\delta,\lambda) \colon \mathrm M(\Psi^{\delta,\lambda}) \leq C}\sum_{ r \in J^{\delta,\lambda}(\Psi^{\delta,\lambda})}\mathfrak Y(r).
\end{equation} 
Hence, 
\begin{equation}
\begin{aligned} \label{supinside}
\limsup_{\lambda\to\infty}&\frac 1\lambda\log Z^{\gamma,\beta,C}_\lambda(X^\lambda)\\
&\leq \limsup_{\delta\downarrow0}\limsup_{\lambda\to\infty}\frac 1\lambda\log \#G(\delta,\lambda)\\
&\quad+\limsup_{\delta\downarrow0}\limsup_{\lambda\to\infty}\frac 1\lambda\log 
\sup_{\Psi^{\delta,\lambda} \in G(\delta,\lambda) \colon \mathrm M(\Psi^{\delta,\lambda}) \leq C }\Big[\frac{\# J^{\delta,\lambda}(\Psi^{\delta,\lambda})}{\inf_{r\in J^{\delta,\lambda}(\Psi^{\delta,\lambda})}\prod_{i\in I^\lambda}N(\lambda)^{- (r^i_{-1}-1)}}\\
&\quad -\gamma \liminf_{\delta\downarrow0}\liminf_{\lambda\to\infty}\inf_{r\in J^{\delta,\lambda}(\Psi^{\delta,\lambda})}\frac 1\lambda \mathfrak S(r)-\beta \liminf_{\delta\downarrow0}\liminf_{\lambda\to\infty}\inf_{r\in J^{\delta,\lambda}(\Psi^{\delta,\lambda})}\frac 1\lambda \mathfrak M(r)\Big].
\end{aligned}
\end{equation}
According to Lemma \ref{lemma-Gstillnegligible} below, the first term on the right-hand side is equal to zero. Now pick a sequence $(\delta_n)_n$ and for each $n$ a sequence $(\lambda_{n,j})_j$ along which the superior limits as $n\to\infty$, respectively $j\to\infty$, are realized. Now pick, for any $n$ and $j$, a maximizer $\widetilde\Psi^{\delta_n,\lambda_{n,j}}$. Pick $\lambda_0$ so large that $N(\lambda)\leq 2\mu(W) \lambda$ for all $\lambda\geq\lambda_0$. Hence,
\[  \bigcup_{\lambda>\lambda_0,\delta \in \mathbb B} G(\delta,\lambda) \subseteq  \Big( \prod_{k=1}^{\kmax} \mathcal{M}_{\leq 2\mu(W)}(W^k) \Big) \times \mathcal{M}_{\leq 2\mu(W)}(W)^{\mathbb N_0}, \numberthis\label{leq} \] 
where we recall that $\Mcal_{\leq \alpha}(V)$ is the set of measures on a space $V$ with total mass $\leq \alpha$.  Note that $\mathcal{M}_{\leq 2\mu(W)}(W^k)$ is compact in the weak topology of $\mathcal{M}(W^k)$ for any $k$, according to Prohorov's theorem. 

Without loss of generality (using two diagonal sequence arguments), we can assume that for all $n \in \mathbb N$, $\widetilde\Psi^{\delta_n,\lambda_{n,j}}$ converges coordinatewise weakly to a collection of measures $\widetilde\Psi^{\delta_n}=((\widetilde \nu_k^{\delta_n})_{k=1}^{\kmax}, (\widetilde \mu_m^{\delta_n})_{m=0}^{\infty})$ as $j \to \infty$, and $\widetilde\Psi^{\delta_n}$ converges coordinatewise weakly to a collection of measures $\widetilde \Psi$ as $n \to \infty$. Then, it is clear that $\widetilde \Psi$ satisfies (i) from \eqref{constraints}, and also that 
\[ \lim_{n \to \infty} \lim_{j \to \infty} \sum_{k=1}^{\kmax} \sum_{l=1}^{k-1} \pi_l \widetilde \nu_k^{\delta_n,\lambda_{n,j}} =  \sum_{k=1}^{\kmax} \sum_{l=1}^{k-1} \pi_l \widetilde \nu_k. \] 
In order to see that (iii) holds for $\widetilde\Psi$, it remains to show that $ \lim_{n \to \infty} \lim_{j \to \infty} \sum_{m=0}^{\infty} m \widetilde \mu_m^{\delta_n,\lambda_{n,j}} = \sum_{m=0}^{\infty} m \widetilde \mu_m$. For $N \in \mathbb N$ 
and for any continuous function $f \colon W \to \mathbb R$, we estimate
\[ \left| \left\langle \sum_{m=0}^{\infty} m(\widetilde \mu_m^{\delta_n,\lambda_{n,j}}-\widetilde \mu_m), f \right\rangle \right| \leq  \sum_{m=0}^{N} m \left|\langle \widetilde \mu_m^{\delta_n,\lambda_{n,j}}-\widetilde \mu_m, f \rangle \right| + \sum_{m=N+1}^{\infty} \Vert f \Vert_\infty m \left| \widetilde \mu_m^{\delta_n,\lambda_{n,j}}(W)-\widetilde \mu_m(W) \right|. \]
The first term on the r.h.s. clearly tends to 0 as $j \to \infty$, followed by $n\to\infty$, for any fixed $N$. The second term can further be estimated from above as follows
\[ \Vert f \Vert_\infty \sum_{m>N} \eta(m) \Big(\sup_{\widetilde m>N} \frac{\widetilde m}{\widetilde \eta(m)}\Big) (\widetilde \mu_m^{\delta_n,\lambda_{n,j}}(W) + \widetilde \mu_m(W)) \leq 2 \Vert f \Vert_\infty   \Big(\sup_{\widetilde m>N} \frac{\widetilde m}{\widetilde \eta(m)}\Big) C. \]
By the assumption that $(\eta(N)/N) \to\infty$ as $N\to\infty$, the right-hand side tends to 0. One can analogously show that $\sum_{m=0}^{\infty} \widetilde \mu_m^{\delta_n,\lambda_{n,j}}$ tends to $\sum_{m=0}^{\infty} \widetilde \mu_m$ as $j \to \infty$ followed by $n\to\infty$, and hence condition (ii) from \eqref{constraints} holds. Also we have $\sum_{m=0}^{\infty} \eta(m)\widetilde \mu_m (W) \leq C$. Altogether, $\widetilde \Psi$ is an admissible trajectory setting.

Now, using the arguments of the proofs of Propositions \ref{prop-combinatorics} and \ref{prop-SIRcomputations} (which also involve the coarsened limits $\widetilde \Psi^{\delta_n}$ for fixed $n \in \mathbb N$) for the subsequential limits $j \to \infty$ followed by $n \to \infty$, we conclude that \[ \lim_{n \to \infty} \lim_{j \to \infty} \frac{\# J^{\delta_n,\lambda_{n,j}}(\widetilde \Psi^{\delta_n,\lambda_{n,j}})}{\inf_{r\in J^{\delta_n,\lambda_{n,j}}(\widetilde \Psi^{\delta_n,\lambda_{n,j}})}\prod_{i\in I^{\lambda_{n,j}}}N(\lambda_{n,j})^{- (r^i_{-1}-1)}}=\mathrm I(\widetilde \Psi) \] and, using the boundedness and continuity of each $f_k$ on $\Mcal_{\leq 2 \mu(W)}(W) \times W^k$, \[ \lim_{n \to \infty} \lim_{j \to \infty} \inf_{r\in J^{\delta_n,\lambda_{n,j}}(\widetilde \Psi^{\delta_n,\lambda_{n,j}})}\frac 1{\lambda_{n,j}} \mathfrak S(r) = \mathrm S(\widetilde \Psi). \] 
Furthermore, the lower semicontinuity of $\mathcal{M}(W)^{\N_0} \to (-\infty,\infty]$, $(\nu_m)_{m\in\N_0} \mapsto \sum_{m\in\N_0}\eta(m) \nu_m(W)$, together with Fatou's lemma implies that
\[ -\beta \liminf_{n \to \infty}\liminf_{j \to\infty}\inf_{r\in J^{\delta_n,\lambda_{n,j}}(\widetilde \Psi^{\delta_n,\lambda_{n,j}})}\frac 1{\lambda_{n,j}} \mathfrak M(r) \leq -\beta \mathrm M (\widetilde \Psi).  \numberthis\label{M-upper} \]

Thus, we conclude that \eqref{esistgenug} (and therefore the upper bound in  Theorem \ref{theorem-variation1}) holds, as soon as Lemma \ref{lemma-Gstillnegligible} is formulated and verified. This we do now. Recall that we are working with fixed $\omega \in \Omega_1$, and that the notion of $G(\delta,\lambda)$ depends on $\omega$ via $G(\delta,\lambda)=G(\delta,\lambda)(X^\lambda(\omega))$.

\begin{lemma} \label{lemma-Gstillnegligible}
For any $\delta \in \mathbb B$ and $\omega \in \Omega_1$, we have
\[ \limsup_{\lambda \to \infty} \frac{1}{\lambda} \log \# G(\delta,\lambda) = 0. \]
\end{lemma}

\begin{proof}
For $\lambda>0$, let $G_1(\delta,\lambda)$ denote the set of $(\nu_k^{\delta,\lambda})_{k=1}^{\kmax}$ satisfying part  \eqref{i-standardsetting} from Definition \ref{defn-standardsetting}. It is easily seen that its cardinality increases only polynomially in $\lambda$. Now, given $(\nu_k^{\delta,\lambda})_{k=1}^{\kmax}\in G_1(\delta,\lambda)$, we will give an upper bound for the number of $(\mu_m^{\delta,\lambda})_{m=0}^{\infty})$ such that the pair of these tuples is in $G(\delta,\lambda)$. This is much more demanding, since there is {\it a priori} no upper bound for $m$. We will provide a $\lambda$-dependent one.

For any $\lambda>0$, $\Psi^{\delta,\lambda} \in G(\delta,\lambda)$ and $j=1,\ldots,\delta^{-d}$ we have that
\[ \lambda \sum_{m=0}^{\infty} m \mu_m^{\delta,\lambda} (W^\delta_j) = \lambda \sum_{k=1}^{\kmax} \sum_{l=1}^{k-1} \pi_l \nu_k^{\delta,\lambda}(W^\delta_j) \leq (\kmax-1) N(\lambda), \]
and the numbers $\mu_0^{\delta,\lambda}(W^\delta_j)$, $\ldots, \mu^{\delta,\lambda}_{(\kmax-1) N(\lambda)}(W^\delta_j)$, are $\frac{1}{\lambda}$ times nonnegative integers.
In particular, $\mu_m^{\delta,\lambda}(W^\delta_j)=0$ for $m>(\kmax-1) N(\lambda)$. 

Let $\varepsilon>0$ be fixed. We claim that for all sufficiently large $\lambda>0$, there are not more than $\varepsilon N(\lambda) \sim \varepsilon \lambda \mu(W) $ nonzero ones out of these quantities. Indeed, if there were at least $\lceil \varepsilon N(\lambda) \rceil$ nonzero ones, denoted $\mu_{m_0}^{\delta,\lambda}(W^\delta_j), \ldots, \mu_{m_{\lceil \varepsilon N(\lambda)\rceil -1}}^{\delta,\lambda}(W^\delta_{j})$ with $0 \leq m_0<m_1  <\ldots < m_{\lceil \varepsilon N(\lambda) \rceil -1} \leq (\kmax-1)N(\lambda)$, then we could estimate
\begin{align*} 
& (\kmax-1)N(\lambda) \geq \sum_{m=0}^{(\kmax-1) N(\lambda)} \lambda m \mu_m^{\delta,\lambda} (W^\delta_j) \geq \sum_{i=0}^{\lceil \varepsilon N(\lambda) \rceil -1} \lambda m_i \mu_{m_i}^{\delta,\lambda} (W^\delta_j) \1 \big\{ \mu_{m_i}^{\delta,\lambda} (W^\delta_j) > 0 \big\} \\ &  = \sum_{i=0}^{\lceil \varepsilon N(\lambda) \rceil -1} \lambda m_i \mu_{m_i}^{\delta,\lambda} (W^\delta_j) \1 \big\{ \mu_{m_i}^{\delta,\lambda} (W^\delta_j) \geq \frac{1}{\lambda} \big\} \geq \sum_{i=0}^{\lceil \varepsilon N(\lambda) \rceil-1} m_i \geq \sum_{m=0}^{\lceil \varepsilon N(\lambda) \rceil -1} m \sim \frac{1}{2} (\varepsilon N(\lambda))(\varepsilon N(\lambda)-1),
\end{align*}
which is a contradiction for all $\lambda>0$ sufficiently large. 

Now, $\# G(\delta,\lambda)$ can be estimated as follows. Let us first fix $(\nu_k^{\delta,\lambda})_{k=1}^{\kmax}\in G_1(\delta,\lambda)$, i.e., satisfying part \eqref{i-standardsetting} from Definition \ref{defn-standardsetting}, and let us count the number of $(\mu_m^{\delta,\lambda})_{m=0}^{(\kmax-1) N(\lambda)}$ such that $((\nu_k^{\delta,\lambda})_{k=1}^{\kmax},(\mu_m^{\delta,\lambda})_{m=0}^{(\kmax-1) N(\lambda)}))$ lies in $G(\delta,\lambda)$. Out of the $\kmax \delta^{-d} N(\lambda)$ quantities $\mu_0^{\delta,\lambda}(W^\delta_j), \ldots, \mu^{\delta,\lambda}_{(\kmax-1) N(\lambda)}(W^\delta_j)$, $j=1,\ldots,\delta^{-d}$, at most $\lceil \varepsilon N(\lambda) \rceil\delta^{-d}$ are nonzero. The number of ways to choose them equals $\binom{\kmax N(\lambda)\delta^{-d}}{\lceil \varepsilon N(\lambda) \rceil\delta^{-d}}$. Having chosen $\lceil \varepsilon N(\lambda) \rceil\delta^{-d}$ potentially nonzero 
ones so that the remaining $\kmax \delta^{-d} N(\lambda)-\lceil \varepsilon N(\lambda) \rceil\delta^{-d}$ ones are equal to zero, according to part \eqref{ii-standardsetting} of Definition \ref{defn-standardsetting} we note that the potentially nonzero ones sum up to $N(\lambda)$, and each one has a value in $\frac{1}{\lambda}\N_0$. For this, there are at most $\binom{N(\lambda)+\lceil \varepsilon N(\lambda) \rceil\delta^{-d}-1}{\lceil \varepsilon N(\lambda) \rceil\delta^{-d}-1}$ combinations, for any choice of the set of the potentially nonzero ones. Since $\omega \in \Omega_1$, we have $N(\lambda)=N(\lambda)(\omega)=\lambda(\mu(W)+o(1))$ as $\lambda \to \infty$ (where the $o(1)$ term depends on $\omega$). Therefore, using Stirling's formula as in  \eqref{Stirling}, we have the following estimate in the limit $\lambda \to \infty$
\begin{align*} 
 \# G(\delta,\lambda) &\leq \# G_1(\delta,\lambda) \binom{\kmax  N(\lambda)\delta^{-d}}{\lceil \varepsilon N(\lambda) \rceil\delta^{-d}} \binom{N(\lambda)+\lceil \varepsilon N(\lambda) \rceil\delta^{-d}-1}{\lceil \varepsilon N(\lambda)  \rceil\delta^{-d}-1} \\ 
& = \e^{o(\lambda)}\exp\Big(-\lambda \mu(W)\Big((\kmax -\varepsilon)\delta^{-d} \log \frac{(\kmax-\varepsilon)\delta^{-d}}{\kmax \delta^{-d}}+\varepsilon \delta^{-d}\log \frac{\varepsilon\delta^{-d}}{\kmax \delta^{-d}}\Big)\Big)\\ 
& \qquad\times \exp\Big(-\lambda \mu(W) \Big( \varepsilon\delta^{-d} \log \frac{\varepsilon\delta^{-d}}{1+\varepsilon\delta^{-d}}+\log \frac{1}{1+\varepsilon\delta^{-d}} \Big) \Big)
\end{align*}
Letting $\varepsilon\downarrow 0$, we conclude that $\limsup_{\lambda \to \infty} \frac{1}{\lambda} \log \# G(\delta,\lambda)=0$. 
\end{proof}

\subsection{The large deviation principle: proof of Theorem \ref{thm-LDP}(i)}\label{sec-LDPproof}

In this section, we prove Theorem \ref{thm-LDP}(i). The combinatorial essence of this theorem has already been proved in Proposition \ref{prop-combinatorics}, including the relations with $\delta$-coarsenings. What remains to be done is to relate this to the coordinatewise weak convergence on $\mathcal A$. We will be able to use some of the arguments of Section \ref{sec-varproof}. 

The lower semicontinuity of $\mathrm I+\mu(W) \log \kmax$ was already discussed in Section \ref{sec-freeenergy}, the nonnegativity in Section \ref{sec-LDP}. These together mean that $\mathrm I+\mu(W) \log \kmax$ is a rate function.


We proceed with the proof of the lower bound. Let $G \subseteq \mathcal A$ be open. If $\inf_G \mathrm I=\infty$, then there is nothing to show, therefore let us assume that there exists $\Psi \in G$ with $\mathrm I(\Psi)<\infty$.  According to Proposition~\ref{prop-standardsettingconstruction}, there is a standard setting $\underline\Psi$ containing $\Psi$. Since $G$ is open, there exists $\varrho>0$ such that $B_{\varrho}(\Psi) \subseteq G$. Let us choose $\delta_0 \in \mathbb B$ and, for any $\mathbb B \ni \delta \leq \delta_0$, some $\lambda_0=\lambda_0(\delta)>0$  such that $\Psi^\delta, \Psi^{\delta,\lambda} \in B_{\varrho}(\Psi)$ for any $\lambda>\lambda_0$. Now we can estimate, for these $\delta$ and $\lambda$,
$$
\begin{aligned}
\mathrm P^{0,0}_{\lambda,X^\lambda}(\Psi_\lambda (S) \in G)&\geq \mathrm P^{0,0}_{\lambda,X^\lambda}(\Psi_\lambda(S) \in B_{\varrho}(\Psi)) \geq \mathrm P^{0,0}_{\lambda,X^\lambda}\big((\Psi_\lambda(S))^\delta=\Psi^{\delta,\lambda}\big)\\
&=\frac 1{Z^{0,0}_{\lambda}(X^\lambda)} \sum_{r\in J^{\delta,\lambda}(\Psi^{\delta,\lambda})}\frac1{\prod_{i\in I^\lambda}N(\lambda)^{r^i_{-1}-1}}
\geq \frac{\# J^{\delta,\lambda}(\Psi^{\delta,\lambda})}{\kmax^{N(\lambda)} \sup_{r\in J^{\delta,\lambda}(\Psi^{\delta,\lambda})}\prod_{i\in I^\lambda}N(\lambda)^{r^i_{-1}-1}}.
\end{aligned}
$$
Now, using Proposition \ref{prop-combinatorics} and the fact that $N(\lambda)/\lambda\to\mu(W)$, we obtain 
$$
\liminf_{\lambda \to \infty} \frac{1}{\lambda} \log \mathrm P^{0,0}_{\lambda,X^\lambda}(\Psi_\lambda (S) \in G)
\geq -\mu(W)\log \kmax-\mathrm I(\Psi).
$$
Note that $\underline \Psi$ is not necessarily controlled because $\mathrm M(\Psi)<\infty$ is not guaranteed. However, since for all $\delta \in \mathbb B$, $s=1,\ldots,\delta^{-d}$, $\lambda>0$,  $\mu_m^{\delta,\lambda}(W^\delta_s)/\mu_m^{\delta}(W^\delta_s)$ does not depend on $m$, we easily see that Proposition \ref{prop-combinatorics} holds for this $\underline \Psi$ as well. Now, take the supremum over $\Psi \in G \cap \lbrace \mathrm I<\infty \rbrace$ to conclude that the lower bound holds.

We continue with the upper bound. Let $F \subseteq \mathcal A$ be closed. Let us choose an increasing sequence $(\lambda_n)_{n \in \mathbb N}$ of positive numbers along which the limit superior in \eqref{upper} is realized. For $\lambda>0$, let  us put 
$$
O(\lambda) =\big\{ \Psi \in \mathcal{A} \colon \mathrm P^{0,0}_{\lambda,X^\lambda}(\Psi_{\lambda}(S) =\Psi)>0 \big\}.
$$
If for all but finitely many $n \in \mathbb N$ we have $F \cap O(\lambda_n)=\emptyset$, then 
\[ \limsup_{\lambda \to \infty} \frac{1}{\lambda} \log \mathrm P^{0,0}_{\lambda,X^\lambda}(\Psi_\lambda (S) \in F) = -\infty. \numberthis\label{minusinfty} \] 
Therefore, without loss of generality, we can assume that $O(\lambda_n) \cap F$ is non-empty for all $n \in \mathbb N$. For $\delta \in \mathbb B$ and $A\subset\Acal$, let us write $A^\delta= \lbrace \Psi^\delta \colon \Psi \in A \rbrace$, where $\Psi^\delta$ is the coordinatewise $\delta$-coarsened version of $\Psi$. Then we have
\[
\begin{aligned}
\mathrm P^{0,0}_{\lambda_n,X^{\lambda_n}}\big(\Psi_{\lambda_n}(S) \in F) &= \mathrm P^{0,0}_{\lambda_n,X^{\lambda_n}}\big(\Psi_{\lambda_n}(S) \in F \cap O(\lambda_n)\big) = \mathrm P^{0,0}_{\lambda_n,X^{\lambda_n}}\big((\Psi_{\lambda_n}(S))^{\delta} \in (F \cap O(\lambda_n))^{\delta})\\
&\leq\# (F \cap O(\lambda_n))^\delta\sup_{\Psi \in F \cap O(\lambda_n)} \frac{\# J^{\delta,\lambda_n}(\Psi^\delta)}{\kmax^{N(\lambda_n)} \inf_{r\in J^{\delta,\lambda_n}(\Psi^\delta)}\prod_{i\in I^{\lambda_n}} N(\lambda_n)^{r^i_{-1}-1}}. 
\end{aligned} \numberthis\label{supinside2}
\]
It is clear that $(F \cap O(\lambda_n))^{\delta} \subseteq G(\delta,\lambda_n)=(O(\lambda_n))^\delta$ for all $n \in \mathbb N$ and $\delta \in \mathbb B$, where $G(\delta,\lambda_n)$ was defined in Section \ref{sec-varproof}. Hence, by Lemma \ref{lemma-Gstillnegligible},
\[ \limsup_{\delta \downarrow 0} \limsup_{n \to \infty} \frac{1}{\lambda_n} \log \# (F \cap O(\lambda_n))^{\delta} = 0.  \] 
It remains to show that
\[ \limsup_{\delta \downarrow 0} \limsup_{n \to \infty} \frac{1}{\lambda_n} \log \Big[ \sup_{\Psi \in F \cap O(\lambda_n)} \frac{\# J^{\delta,\lambda_n}(\Psi^\delta)}{\inf_{r\in J^{\delta,\lambda_n}(\Psi^\delta)}\prod_{i\in I^{\lambda_n}} N(\lambda_n)^{r^i_{-1}-1}} \Big] \leq -\inf_{\Psi \in F} \mathrm I(\Psi). \numberthis\label{upperwish} \]
One can do this analogously to the proof of the upper bound of Theorem \ref{theorem-variation1} starting from \eqref{supinside}. 
Indeed, using Prohorov's theorem together with a diagonal sequence argument, we find $\Psi^\ast \in \mathcal A$ that the maximizer in \eqref{supinside2} converges to along a subsequence of $\delta$'s and $\lambda_n$'s. The limit lies in $F$ because $F$ is closed. Using the lower semicontinuity of $\mathrm I$ together with Fatou's lemma, we conclude that the left-hand side of \eqref{upperwish} is not larger than $-\mathrm I(\Psi^\ast) $, which itself is not larger than $-\inf_F \mathrm I$. This finishes the proof of the upper bound in Theorem \ref{thm-LDP}(i).

\section{Analysis of the minimizers} \label{sec-minimizers}
This section is devoted to the proof of Proposition \ref{prop-minimizer}. In particular, in Section \ref{sec-positivity}, we show that the infimum in \eqref{variation} is attained and, for any minimizer $\Psi=((\nu_k)_{k=1}^{\kmax},(\mu_m)_{m=0}^{\infty})$, for any $k\in[\kmax]$, $\mu^{\tensor k}$ is absolutely continuous with respect to $\nu_k$ and $\mu$ is absolutely continuous with respect to each $\mu_m$. Further, for all $k \in [\kmax]$ and $m \in \N_0$, there exist constants $c_k>0$, $k \in [\kmax]$, and $c'_m>0$, $m \in \N_0$ such that $\nu_k(A) \geq c_k \mu^{\tensor k}(A)$ for all $A \subseteq W^k$ measurable and $\mu_m(A') \geq c_m'(A')$ for all $A' \subseteq W$ measurable. This is a prerequisite for perturbing the minimizer in many admissible directions. In Section \ref{sec-minimizerproof} we finish the proof of Proposition \ref{prop-minimizer} by deriving the Euler--Lagrange equations. For the rest of the section, we fix all parameters $W, \mu,\gamma,\beta$ and $\kmax$. Moreover, we use the 
following 
representation of $\mathrm I$ from \eqref{quenchedentropy}.
\[ \mathrm I(\Psi)= \sum_{k=1}^{\kmax} H_{W^k}(\mu \tensor M^{\tensor (k-1)})+\sum_{m=0}^{\infty} H_W(\mu_m \mid \mu)-\mu_m(W)\log \frac{(\e\mu(W))^{-m}}{m!}. \numberthis\label{I-minimizer} \]

\subsection{Existence and positivity of the minimizers} \label{sec-positivity}

We start with the following lemma, which follows almost immediately from the arguments of the proof of the upper bound of Theorem \ref{theorem-variation1} in Section \ref{sec-varproof}.

\begin{lemma} \label{lemma-existence}
The set of minimizers of the variational formula in \eqref{variation} is non-empty, compact and convex.
\end{lemma}

\begin{proof} Recall that the three functionals I, S, M are lower semicontinuous and convex. Furthermore, it is clear that we can restrict the infimum in \eqref{variation} to those $\Psi$ that satisfy also ${\rm M}(\Psi)\leq C$ for any sufficiently large $C$. But, as we have seen in Section \ref{sec-varproof}, this set of $\Psi$'s is compact. From this, all our assertions easily follow.
\end{proof}

Now we prove that, for each minimizer $\Psi$,  $\mu^{\tensor k}$ is absolutely continuous with respect to $\nu_k$ and $\mu$ is absolutely continuous with respect to each $\mu_m$, and the corresponding Radon--Nikodym derivatives are even bounded away from 0. (Note that the opposite absolute continuities are true by finiteness of the relative entropies in \eqref{quenchedentropy}.) We need to show this only for $\kmax>1$, as we explained after Proposition \ref{prop-minimizer}. Let us start with verifying the absolute continuities.

\begin{lemma} \label{lemma-everywherepositive}
If $\kmax>1$ and  $\Psi=((\nu_k)_{k=1}^{\kmax},(\mu_m)_{m=0}^{\infty})$ is a minimizer of \eqref{variation}, then $\mu^{\tensor k}\ll\nu_k$ for any $k \in [\kmax]$, and $\mu\ll\mu_m$ for any $m \in \N_0$. 
\end{lemma}

\begin{proof} 
The essence of the proof is the following. The functionals $\mathrm M(\cdot)$ and $\mathrm S(\cdot)$ are linear in each $\mu_m$ respectively $\nu_k$, as well as the third term in $\mathrm I(\cdot)$ in \eqref{quenchedentropy} in each $\mu_m$. On the other hand, the function $x \mapsto x \log x$ has the slope $-\infty$ at $x \downarrow 0$. We show the following assertions about the minimizer $\Psi$ step by step as follows. Recall that $M=\sum_{m \in \mathbb N_0} m \mu_m=\sum_{k \in [\kmax]} \sum_{l=1}^{k-1} \pi_l \nu_k$. We write $\geq$ and $>$, respectively, between measures in $\Mcal(W^k)$ if their difference lies in $\Mcal(W^k)$, respectively in $\Mcal(W^k)\setminus\{0\}$.

Fix a measurable set $A\subset W$ such that $\mu(A)>0$. Then we have:
\begin{enumerate}
\item \label{Mnon0} $M(A)>0$.
\item\label{mumsandwich} for any $m_1 <m_0 <m_2$ such that $\mu_{m_1}(A)>0$ and $\mu_{m_2}(A) > 0$, also  $\mu_{m_0}(A) >0$.

\item \label{mu0} $\mu_0(A)> 0$.

\item \label{mukmax} $\mu_{m}(A)>0$ for any $m\geq \kmax$.

\item \label{nuksandwich} $\nu_k(A^k)>0$ for any $k\in[\kmax]$.
\end{enumerate}

Indeed, these steps are verified respectively as follows. In each of the steps, for $\varepsilon \in (0,1)$, we construct an admissible trajectory setting $\Psi^{\varepsilon}=((\nu_k^{\varepsilon})_{k=1}^{\kmax},(\mu_m^{\varepsilon})_{m=0}^{\infty})$ such that $\mathrm I(\Psi^\varepsilon)+\gamma \mathrm S(\Psi^\varepsilon)+\beta \mathrm M(\Psi^\varepsilon)<\mathrm I(\Psi)+\gamma \mathrm S(\Psi)+\beta \mathrm M(\Psi)$ for sufficiently small $\varepsilon>0$, and therefore $\Psi$ is not a minimizer of \eqref{variation}. 
\begin{enumerate}
\item If $M(A)=0$, then in particular $\mu_0(A)=\nu_1(A)=\mu(A)$ and $\mu_m(A)=0$ for all $m>0$. Also, $\pi_1\nu_2(A)=\nu_2(W \times A)=0$, according to the definition of $M$.
    
Let us define $\Psi^\varepsilon$ as follows:  $\nu_2^\eps=(1-\varepsilon)\nu_2
+\varepsilon (\mu^{\tensor 2}) /\mu(W) $, $\nu_k^\varepsilon=(1-\varepsilon) \nu_k$ for $k \neq 2$, $\mu_1^\eps=(1-\varepsilon)\mu_1+\varepsilon \mu $ and $\mu_m^\varepsilon=(1-\varepsilon)\mu_m$ for $m \neq 1$. Then we compute and estimate the three terms of the entropy ${\rm I}(\Psi)$ as follows.
\begin{equation*}
\begin{aligned}
&\sum_{k=1}^{\kmax} H_{W^k}\big(\nu_k^\varepsilon \mid \mu \tensor (M^\eps)^{\tensor (k-1)}\big) \\ & \leq \sum_{k=1}^{\kmax} H_{W \times (W\setminus A)^{k-1}} ((1-\varepsilon)\nu_k \mid \mu \tensor (M^\varepsilon)^{\tensor (k-1)})+H_{W\times A}\Big( \frac{\varepsilon \mu^{\tensor 2}}{\mu(W)} \mid \varepsilon \mu^{\tensor 2}\Big)+\Ocal(\varepsilon) \\  & \leq \sum_{k=1}^{\kmax} H_{W^k}(\mu \tensor M^{\tensor (k-1)})+\Ocal(\varepsilon), 
        \end{aligned}
    \end{equation*}
furthermore
\begin{equation}\label{epslogeps1}
\begin{aligned}
&\sum_{m=0}^{\infty} H_W(\mu_m^\varepsilon \mid \mu)-\mu_m^\eps(W)\log \frac{(\e\mu(W))^{-m}}{m!} \\ &\leq 
H_W((1-\varepsilon)\mu_m \mid \mu)-\mu_m(W)\log \frac{(\e\mu(W))^{-m}}{m!} + \mu(A) \varepsilon \log \varepsilon+\Ocal(\varepsilon).
\end{aligned}
\end{equation}

For the second term we used the convexity of the relative entropy in the form
\[ H_W((1-\varepsilon) \mu_1+\varepsilon \mu \mid \mu) \leq (1-\eps)H_W(\mu_1\mid\mu)\leq H_W(\mu_1\mid\mu) +\Ocal(\eps). \numberthis\label{entropyJensen} \]
This in turn follows from \cite[Lemmas 3.10, 3.11]{HJKP15}, which implies that, for any $k \in \mathbb N$, $\xi,\eta \in \mathcal{M}(W^k)$ with $\eta \neq 0$ and $\xi \ll \eta$, 
\[ \Big| H_{W^k}(\xi \mid \eta)- H_{W^k}((1-\varepsilon) \xi \mid \eta) \Big| \underset{\varepsilon \downarrow 0}{\asymp} \varepsilon.  \]
It follows that, as $\varepsilon\downarrow0$,
\begin{equation}\label{Ismall}
{\rm I}(\Psi^\varepsilon)+\gamma \mathrm S(\Psi^\varepsilon)+\beta \mathrm M(\Psi^\varepsilon)-\big[{\rm I}(\Psi)+\gamma \mathrm S(\Psi)+\beta \mathrm M(\Psi)\big]\leq \Ocal(\eps)+ \mu(A) \varepsilon \log \varepsilon,
\end{equation}
which is negative for all sufficiently small $\eps>0$. Thus, $\Psi$ is not a minimizer.

\item If $M(A)>0$ but $\mu_{m_1}(A)>0$, $\mu_{m_2}(A)>0$ and $\mu_{m_0}(A)=0$ for some $m_1<m_0<m_2$, then let $\nu_k^\varepsilon=\nu_k$ for all $k \in [\kmax]$ and let $\mu_{m_0}^\varepsilon=(1-\varepsilon)\mu_{m_0}+\varepsilon(\alpha_1 \mu_{m_1}+\alpha_2\mu_{m_2})$, $\mu_{m_1}^\varepsilon=(1-\alpha_1\varepsilon)\mu_{m_1}$, $\mu_{m_2}^\varepsilon=(1-\varepsilon\alpha_2)\mu_{m_2}$, where $\alpha_1,\alpha_2 \in (0,1)$ are such that $\alpha_1+\alpha_2=1$ and $m_1\alpha_1+m_2\alpha_2=m_0$. 
Then, $\Psi^\varepsilon$ is an admissible trajectory setting with $M^\varepsilon=M$. It follows similarly to the previous computation that $\mathrm I(\Psi^\varepsilon)+\gamma \mathrm S(\Psi^\varepsilon)+\beta \mathrm M(\Psi^\varepsilon)<\mathrm I(\Psi)+\gamma \mathrm S(\Psi)+\beta \mathrm M(\Psi)$ 
for all sufficiently small $\varepsilon>0$. However, instead of \eqref{epslogeps1}, we have
\[
\begin{aligned}
& \sum_{m=0}^{\infty} H_W(\mu_m^\varepsilon \mid \mu)-\mu_m^\eps(W)\log \frac{(\e\mu(W))^{-m}}{m!} \\ & \leq \sum_{m=0}^{\infty} H_{W}(\mu_m \mid \mu) -\mu_m(W)\log \frac{(\e\mu(W))^{-m}}{m!}+(\alpha_1\mu_{m_1}(A)+\alpha_2\mu_{m_2}(A))\eps\log\eps+\Ocal(\eps),
\end{aligned}
\]
as $\eps\downarrow0$.

\item If $M(A)>0$ but $\mu_0(A)=0$, let $\nu_k^\varepsilon=(1-\varepsilon)\nu_k$ for all $1<k\leq \kmax$, $\mu_m^\varepsilon=(1-\varepsilon)\mu_m$ for all $m>0$, $\mu_0^\varepsilon=\varepsilon\mu+ (1-\varepsilon)\mu_0$ and $\nu_1^\varepsilon=(1-\varepsilon)\nu_1+\varepsilon\mu$. 
It is again sufficient to consider the entropy terms in $\mathrm I$. 
The summands on $k>1$ can be estimated as follows.
\[ \begin{aligned}
\sum_{k=2}^{\kmax} H_{W^k}(\nu_k^\varepsilon \mid \mu \tensor (M^\eps)^{(k-1)})&=\sum_{k=2}^{\kmax} H_{W^k}((1-\varepsilon) \nu_k \mid (1-\varepsilon)^{k-1} \mu \tensor M^{k-1})\\
&\leq \sum_{k=2}^{\kmax} H_{W^k}(\nu_k \mid \mu \tensor M^{(k-1)}) +\Ocal(\eps). 
    \end{aligned}
   \] 
The summand for $k=1$ can be estimated with the help of \eqref{entropyJensen}.
For the summand for $m=0$, we have
\begin{align*}  
H_W(\mu_0^\varepsilon \mid \mu) & = H_{W\setminus A}((1-\varepsilon)\mu_0+\varepsilon \mu \mid \mu) + \mu(A)\eps \log \eps \\& \leq H_{W \setminus A} ((1-\varepsilon) \mu_0 \mid \mu)+\mu(A)\eps\log\eps  + \Ocal(\varepsilon)=H_W(\mu_0\mid\mu)+\mu(A) \eps\log\eps  + \Ocal(\varepsilon). 
\end{align*}
while the remaining sum is handled as follows.
\[ \begin{aligned} &\sum_{m=1}^{\infty} H_W(\mu_m^\varepsilon \mid \mu)-\mu_m^\eps(W)\log \frac{(\e\mu(W))^{-m}}{m!}\\ &=\sum_{m=1}^{\infty} H_W((1-\varepsilon)\mu_m \mid \mu)-\mu_m(W)\log \frac{(\e\mu(W))^{-m}}{m!}+\Ocal(\eps). \end{aligned} \] 
Thus, \eqref{Ismall} holds also here, which implies the claim.

\item If $M(A)>0$ but $\mu_{m_0}(A)=0$ for some $m_0 \geq \kmax$, let $\mu_{m_0}^{\varepsilon}=(1-\varepsilon)\mu_{m_0}+ \varepsilon M /m_0$, $\mu_m^\eps=(1-\varepsilon)\mu_m$ for $m \notin \{ 0,m_0 \}$, and $\nu_k^\varepsilon=\nu_k$ for all $k \in[\kmax]$. Then,
\[ \sum_{m=1}^{\infty} m \mu_m^{\varepsilon} = (1-\varepsilon) \sum_{m=1}^{\infty} m \mu_m+ \frac{\varepsilon m_0}{m_0} \sum_{k=1}^{\kmax} \sum_{l=1}^{k-1} \pi_l \nu_k =\sum_{k=1}^{\kmax} \sum_{l=1}^{k-1} \pi_l \nu_k. \] 
On the other hand, we have \[ \mu-\sum_{m=1}^{\infty} \mu_m^{\varepsilon} \geq \mu-(1-\varepsilon)\sum_{m=1}^{\infty} \mu_m -\frac{\varepsilon(\kmax-1)}{m_0}\mu \geq (1-\varepsilon)\mu -(1-\varepsilon)\sum_{m=1}^{\infty} \mu_m=(1-\varepsilon) \mu_0. \]
Therefore, if we put $\mu_0^{\varepsilon} = \mu - \sum_{m=1}^{\infty} \mu_m^{\varepsilon}$, then $\mu_0^{\varepsilon} \geq (1-\varepsilon) \mu_0 \geq 0$, in particular $\Psi^\varepsilon$ is an admissible trajectory setting.  Now we can proceed analogously to \eqref{mu0} to conclude that $\mathrm I(\Psi^\varepsilon)+\gamma \mathrm S(\Psi^\varepsilon)+\beta \mathrm M(\Psi^\varepsilon)<\mathrm I(\Psi)+\gamma \mathrm S(\Psi)+\beta \mathrm M(\Psi)$ for sufficiently small $\varepsilon>0$. 
\end{enumerate}
The proof of \eqref{nuksandwich} is very similar to the ones of \eqref{mumsandwich}, \eqref{mu0} and \eqref{mukmax}, thus we leave it to the reader.
\end{proof}

The proof of the following lemma is similar to the one of Lemma~\ref{lemma-everywherepositive}, therefore we omit its proof.
\begin{lemma}\label{lemma-densitybounds}
If $\kmax>1$ and $\Psi=((\nu_k)_{k=1}^{\kmax},(\mu_m)_{m=0}^{\infty})$ is a minimizer of \eqref{variation}, then for each $k \in [\kmax]$, 
there exists $c_k>0$ such that for all $A \subseteq W^k$ measurable, $\nu_k(A) \geq c_k \mu^{\tensor k}(A)$ holds. Similarly, for each $m \in \N_0$, there exists $c'_m>0$ such that for all $A' \subseteq W$ measurable, $\mu_m(A') \geq c'_m \mu(A')$ holds.
\end{lemma}

\subsection{Deriving the Euler--Lagrange equations} \label{sec-minimizerproof} 

In this section, we finish the proof of Proposition \ref{prop-minimizer}. According to the results of Section \ref{sec-positivity}, now we see that \eqref{variation} exhibits at least one minimizer, and all minimizers have almost everywhere positive Lebesgue density on the corresponding powers of $\supp~\mu$. Knowing this, we now carry out the perturbation analysis for the minimizer(s) of the optimization problem in \eqref{variation} and derive the shape of the minimizers in most explicit terms.

We use the method of Lagrange multipliers in the framework of a perturbation argument. Let $\Psi=((\nu_k)_{k=1}^{\kmax},(\mu_m)_{m=0}^{\infty})$ minimize \eqref{variation}. Fix any collection of signed measures $\Phi=((\tau_k)_{k=1}^{\kmax},(\sigma_m)_{m=0}^{\infty} )$ such that only finitely many $\sigma_m$'s are different from zero, the Radon--Nikodym derivative $\smfrac{\d \tau_k}{\d \mu^{\tensor k}}$ is a simple function for each $k$, also $\smfrac{\d \sigma_m}{\d \mu}$ is a simple function for each $m$, further they satisfy the following constraints:
\begin{equation}\label{perturbator}
(i) \quad\sum_{k=1}^{\kmax} \pi_0\tau_k=0,\qquad (ii)\quad \sum_{m=0}^{\infty} \sigma_m=0,\qquad (iii)\quad M_\Phi:=\sum_{m=0}^{\infty} m\sigma_m=\sum_{k=1}^{\kmax}\sum_{l=1}^{k-1}\pi_l\tau_k.
\end{equation}
Then it follows from Lemma~\ref{lemma-densitybounds} that, for any $\varepsilon\in\R$  with sufficiently small $|\eps|$, $\Psi+\varepsilon\Phi=((\nu_k+\eps \tau_k)_{k=1}^{\kmax},(\mu_m+\eps\sigma_m)_{m=0}^{\infty})$ is a collection of (non-negative!) measures that satisfies \eqref{constraints} and is therefore admissible in the variational formula in \eqref{variation}. That \eqref{constraints} is satisfied follows easily from \eqref{perturbator}. Furthermore, using the notation of Section~\ref{sec-positivity}, the non-negativity follows from the fact that each $\tau_k$ respectively each $\sigma_m$ is a finite linear combination of measures of the form $\1_A\,\d \mu^{\tensor k}$ with $A\subset W^k$ respectively of the form $\1_B\, \d \mu$ with $B \subset W$, and we have $\1_A\,\d \mu^{\tensor k} \leq c_k^{-1} \1_A\,  \nu_k$ respectively $\1_B\, \d \mu \leq c_m'^{-1} \1_B\, \d \mu_m $. Since only finitely many such summands are involved, there is a constant $C>0$ such that $|\tau_k|\leq C\nu_k$ and $|\sigma_m|\leq C\mu_m$ for any $k\in[\kmax]$ and $m\in\N_0$, and hence it suffices to take $|\eps|<1/C$.

From minimality, we deduce that
\begin{equation}\label{ELcondition}
     0=\frac{\partial}{\partial \varepsilon} \Big|_{\varepsilon=0} \Big(\mathrm I(\Psi+\varepsilon\Phi)+\gamma \mathrm S(\Psi +\varepsilon \Phi)+\beta \mathrm M(\Psi+\varepsilon \Phi) \Big).
\end{equation}
We calculate the latter two terms as 
$$
\frac{\partial}{\partial \varepsilon} \Big|_{\varepsilon=0} \Big(\gamma \mathrm S(\Psi +\varepsilon \Phi)+\beta \mathrm M(\Psi+\varepsilon \Phi) \Big)
=\gamma\sum_{k \in [\kmax]}\langle \tau_k, \widetilde f_k\rangle+\beta \sum_{m\in\N_0} \eta(m)\sigma_m(W),
$$
where, as before, we used  the notation $\langle \nu,f\rangle$ for the integral of a function $f$ with respect to a measure $\nu$. Abbreviating $M=\sum_{k \in [\kmax]}\sum_{l=1}^{k-1} \pi_l\nu_k$ and using the representation \eqref{I-proof} of $\mathrm I(\cdot)$, we see that
\begin{equation}\label{orthogonality1}
\begin{aligned}
 \frac{\partial}{\partial \varepsilon} \Big|_{\varepsilon=0} \mathrm I(\Psi+\varepsilon\Phi) &= \sum_{k\in[\kmax]} \Big\langle \tau_k, 1+\log\frac{\d\nu_k}{\d\mu^{\otimes k}}\Big\rangle- \Big\langle M_\Phi,1+\log \frac{\d M}{\d \mu}\Big\rangle +\sum_{m\in\N_0} \Big\langle \sigma_m,1+\log \frac{\d\mu_m}{\d(c_m \mu)}\Big\rangle,
\end{aligned}
\end{equation}
where we recall that $c_m=\e^{-1/(\e\mu(W))} (\e\mu(W))^{-m}/m!$.
Summarizing, we obtain from \eqref{ELcondition} that 
\begin{equation}
 0=\Big\langle \Phi,\big((h_k)_{k \in [\kmax]},(g_m)_{m\in\N_0}\big)\Big\rangle,
 \end{equation}
where
$$
h_k=\gamma \widetilde f_k+2-k+\log \frac{\d\nu_k}{\d(\mu\otimes M^{\otimes(k-1)})}\qquad\mbox{and}\qquad g_m=\beta \eta(m)+1+\log \frac{\d\mu_m}{\d\mu}-\log \frac{(\e \mu(W))^{-m}}{m!}.
$$
We conceive $\Phi$ as an element of the vector space
$$
\mathcal A_\pm =\prod_{k \in [\kmax]}\Mcal_\pm(W^k)\times \Mcal_\pm(W)^{\N_0}
$$
where $\Mcal_\pm$ is the set of signed measures equipped with the weak topology, and $((h_k)_{k\in [\kmax]},(g_m)_{m\in\N_0})$ as a function on $\prod_{k \in [\kmax]}W^k\times W^{\N_0}$. 
The condition in \eqref{perturbator} means that $\Phi$ is perpendicular to any function in
$$
\begin{aligned}
\Fcal&=\Big\{((\varphi_k)_{k\in[\kmax]},(\psi_m)_{m\in\N_0})\colon \varphi_k\colon W^k\to\R, \psi_m\colon W\to\R \text{ bounded and measurable for any }k,m,\\
&\qquad \exists \widetilde A,\widetilde B,\widetilde C\colon W\to\R\colon \varphi_k(x_0,\dots,x_{k-1})=\widetilde A(x_0)+\sum_{l=1}^{k-1}\widetilde C(x_l),\\
&\qquad\text{and }\psi_m(x)=\widetilde B(x)-m\widetilde C(x)\text{ for }x,x_0,\dots,x_{k-1}\in W\Big\}.
\end{aligned}
$$
We have shown that, if $\Phi$ is perpendicular to any simple function in $\Fcal$, then it is also perpendicular to $((h_k)_{k \in [\kmax]},(g_m)_{m\in\N_0})$. Since $\Fcal$ is a closed linear subspace of $\mathcal A_\pm$, it follows that it contains this element. That is, there are three functions $\widetilde A,\widetilde B,\widetilde C$ on $W$ such that, for any $k$ respectively $m$,
$$
h_k(x_0,\dots,x_{k-1})=\widetilde A(x_0)+\sum_{l=1}^{k-1}\widetilde C(x_l)\qquad \mbox{and}\qquad g_m(x)=\widetilde B(x)-m\widetilde C(x),\qquad x,x_0,\dots,x_{k-1}\in W.
$$
Using an obvious substitution, this is equivalent to the existence of three positive functions $A,B,C$ such that
\begin{eqnarray}
 \nu_k(\d x_0,\ldots, \d x_{k-1} ) &=& \mu(\d x_0)\, A(x_0)\prod_{l=1}^{k-1} \big(C(x_{l})M(\d x_l)\big) \e^{-\gamma \widetilde f_k(x_0,\ldots,x_{k-1})},\qquad k \in [\kmax],\label{nukminimizer}\\
 \mu_m(\d x) &=& \mu(\d x)\, B(x)\frac{(C(x)\mu(W))^{-m} }{m!}\e^{-\beta \eta(m)},\qquad m\in\N_0.\label{mumminimizer}
\end{eqnarray}
From (i) and (ii) in \eqref{constraints}, we can identify $A$ and $B$ as
\begin{eqnarray}\label{ABident}
\frac1{A(x_0)}&=&\sum_{k\in[\kmax]} \int_{W^{k-1}}\prod_{l=1}^{k-1} \big(C(x_{l})M(\d x_l)\big)\e^{-\gamma \widetilde f_k(x_0,\ldots,x_{k-1})},\\
\frac1{B(x)}&=&\sum_{m\in\N_0}\frac{(C(x)\mu(W))^{-m} }{m!}\e^{-\beta \eta(m)}.
\end{eqnarray}
Furthermore, condition (iii) says that
\begin{equation}\label{Ccondition}
\frac 1{C(x)}=\frac 1{C(x)}\frac{\mu(\d x)}{M(\d x)}\varphi\Big(\frac{1}{C(x)\mu(W)}\Big)=\Gamma(C\,\d M,x),\qquad x\in W,
\end{equation}
where $ \varphi(\alpha)=\sum_{m\in\N_0}m\frac{\alpha^{m} }{m!}\e^{-\beta \eta(m)}/\sum_{m\in\N_0}\frac{\alpha^{m} }{m!}\e^{-\beta \eta(m)}$ for $\alpha\in[0,\infty)$ and
\begin{equation}\label{Gammadef}
\Gamma(\d \widetilde M,x)=\int_W \mu(\d x_0)\,\frac{\sum_{k\in[\kmax]} \int_{W^{k-2}}\prod_{l=1}^{k-2} \widetilde M(\d x_l) \, F_k(x_0,x_1,\dots,x_{k-2}, x)} 
{\sum_{k\in[\kmax]} \int_{W^{k-1}}\prod_{l=1}^{k-1}\widetilde M(\d x_l)\,\e^{-\gamma \widetilde f_k(x_0,\dots,x_{k-1})}},
\end{equation}
where
\begin{equation}\label{Fdefinition}
F_k(x_0,x_1,\dots,x_{k-2}, x)=\sum_{l=1}^{k-1} \e^{-\gamma \widetilde f_k(x_0,y^{l})},
\end{equation}
$y^{l}$ is the vector of length $k-1$, consisting of $x_1,\dots,x_{k-2}$; augmented by $x$ at the $l$-th place, and $\widetilde M(\d x)=C(x)M(\d x)$. This ends our derivation of the Euler--Lagrange equations for any minimizer $\Psi$ of \eqref{variation}.

This description of $C$ and $M$ is rather implicit and involved, therefore we cannot offer any simple criterion for the uniqueness of the minimizers of \eqref{variation}. Also, the question of continuity of the tilting functions $A$, $B$ and $C$ is open.

Since $\mathrm I+\gamma \mathrm S + \beta \mathrm M$ is convex, it follows that any admissible trajectory setting $\Psi$ satisfying \eqref{nukminimizer}--\eqref{Fdefinition} is a minimizer of \eqref{variation}. 

\section{Proof of Proposition \ref{prop-variationbeta=0}}\label{sec-beta=0proof}

We proceed analogously to Sections~\ref{sec-proofingredients} and \ref{sec-prooflargelambda}, and thus we start with part \eqref{freeenergy-variationbeta=0}, i.e., with verifying \eqref{beta0variation}. We use the discretization argument from Section~\ref{sec-discretization} again. We now provide the definition of a \emph{transmission setting}, the analogue of Definition~\ref{defn-standardsetting} of a standard setting with no reference to users receiving given numbers of incoming hops.

\begin{defn} \label{defn-transmissionsetting}
A \emph{transmission setting} is a collection of measures  
\begin{equation}
\begin{aligned}
\underline{\Sigma}&=\Big( \Sigma=(\nu_k)_{k=1}^{\kmax}, ((\nu_k^{\delta})_{k=1}^{\kmax})_{\delta \in \mathbb B}, ((\nu_k^{\delta,\lambda})_{k=1}^{\kmax})_{\delta \in \mathbb B,\lambda>0}, (\mu^{\delta,\lambda})_{\delta \in \mathbb B,\lambda>0} \Big)
\end{aligned}
\end{equation}
such that for any $\delta,\delta'\in\mathbb B$, $\lambda>0$, $k\in [\kmax]$ and $s,s_0,\ldots,s_{k-1}=1,\ldots,\delta^{-d}$, respectively, $\nu_k \in \Mcal(W^k)$, and parts \eqref{muadmissible-standardsetting}, 
\eqref{i-standardsetting}, \eqref{nukconsistency-standardsetting} and \eqref{nukconv-standardsetting} 
of Definition \ref{defn-standardsetting} hold.
\end{defn}
Recall that Definition~\ref{defn-transmissionsetting} implies parts \eqref{muconsistency-standardsetting}, \eqref{muconv-standardsetting}, \eqref{mudelta-standardsetting} and \eqref{nukdelta-standardsetting} of Remark~\ref{remark-standardsetting}. Further, it is easy to see that for any transmission setting $\underline\Sigma$, $\Sigma$ is an asymptotic routing strategy.

The following lemma describes the combinatorics of the choices of message trajectories in the system. We recall the empirical measures
$(R_{\lambda,k}(s))_{k\in[\kmax]}$ from \eqref{Rdef}. 

\begin{lemma}
Let $\underline \Sigma$ be a transmission setting. For $\delta \in \mathbb B$ and $\lambda>0$ let
\[ K^{\delta,\lambda}(\underline\Sigma)=\big\{ s \in \mathcal{S}_{\kmax}(X^\lambda): R_{\lambda,k}^\delta(s)=\nu_k^{\delta,\lambda} ~\forall k=1,\ldots,\kmax \big\}. \]
Then we have
$\# K^{\delta,\lambda}(\underline\Sigma)=N^1_{\delta,\lambda}(\underline\Sigma) \times N^4_{\delta,\lambda}(\underline\Sigma)$,
where $N^1_{\delta,\lambda}(\underline\Sigma)$ equals $N^1_{\delta,\lambda}(\underline\Psi)$ from \eqref{comb-multinomial} for any standard setting $\underline \Psi$ containing $\underline \Sigma$, and
\[ N^4_{\delta,\lambda}(\underline\Sigma)=\prod_{j=1}^{\delta^{-d}} (\lambda\mu^{\delta,\lambda}(W^\delta_j))^{\lambda \sum_{k=1}^{\kmax} \sum_{l=1}^{k-1}\pi_l\nu_k^{\delta,\lambda}(W^\delta_j)}. \]
\end{lemma}
\begin{proof}
We proceed in two steps by counting first the trajectories, registering only the partition sets $W^\delta_i$ that they travel through, second, the choices of the relays for each hop in each partition set.
Since every choice in the two steps can be freely combined with the other one, the product of the two cardinalities is equal to the number of all trajectory configurations with the prescribed coarsened empirical measures.

\begin{enumerate}[(A)]
\item {\em Number of the transmitters of trajectories passing through given sequences of $\delta$-subcubes.} This is equal to the corresponding quantity in the proof of Lemma~\ref{lem-cardinality}, hence it equals $N^1_{\delta,\lambda}(\underline\Sigma)$.
\item {\em Number of assignments of the hops to the relays.} For each $i=1,\ldots,\delta^{-d}$, there are $\lambda \sum_{k=1}^{\kmax}\sum_{l=1}^{k-1} \pi_l\nu_k(W^\delta_i)$ incoming hops arriving to the relays in $W^\delta_i$ in total. Each incoming hop arriving at $W^\delta_i$ can choose any of the $\lambda \mu^{\delta,\lambda}(W^\delta_i)$ users as the corresponding relay. Such choices between different hops in $W^\delta_i$ are independent, moreover all the choices in $W^\delta_i$ are independent from all the choices in $W^\delta_j$ for $j \neq i$. It follows that the number of assignments equals $N^4_{\delta,\lambda}(\underline\Sigma)$.
\end{enumerate}
We also see that all the choices in the two parts are independent of each other, i.e., they can be freely combined with each other and yield different combinations. Hence, we arrived at the assertion.
\end{proof}
Using the arguments of the proof of Proposition \ref{prop-combinatorics}, the next lemma immediately follows.

\begin{lemma} \label{lemma-beta0combinatorics} Let $\underline \Sigma$ be a transmission setting. Then 
\[ \lim_{\delta\downarrow0}\lim_{\lambda\to\infty}\frac{1}{\lambda}\log\frac{\#
K^{\delta,\lambda}(\underline\Sigma)}{N^0_{\delta,\lambda}(\underline\Sigma)} = -\mathrm J(\Sigma)
\in (-\infty,\infty] , \numberthis\label{newcombinatorics} \]
where  $N^0_{\delta,\lambda}(\underline\Sigma)$ equals $N^0_{\delta,\lambda}(\underline\Psi)$ from \eqref{N0def} for any standard setting $\underline \Psi$ containing $\underline \Sigma$. 
Moreover, if the r.h.s.~of \eqref{newcombinatorics} is finite, then
\[ \lim_{\delta\downarrow0}\lim_{\lambda\to\infty}\frac{1}{\lambda\log\lambda}\log\# K^{\delta,\lambda}(\underline \Sigma) = M(W)=\sum_{k=1}^{\kmax} (k-1)\nu_k(W). \] 
\end{lemma}

Now, the identity in \eqref{beta0variation} follows from the proof of Theorem~\ref{theorem-variation1}, using transmission settings instead of standard settings and replacing Proposition~\ref{prop-combinatorics} by our Lemma \ref{lemma-beta0combinatorics}. There is one more major change in the proof. Indeed, instead of the compactness of $\lbrace \Psi \colon \Psi ~\text{adm.~trajectory setting},~ \mathrm M(\Psi) \leq y \rbrace$ for all $y \geq 0$ in Section~\ref{sec-varproof} and the fact that any level set of $\mathrm I+ \gamma \mathrm S + \beta \mathrm M$ is contained in a larger level set of $\mathrm M$, one shall use the following argument. Using that $\mathrm S$ is continuous on the set of asymptotic routeing strategies, and that $\mathrm J$ is lower semicontinuous, bounded from below and it has compact level sets \cite[Section 6.2]{DZ98}, it follows that each level set of $\mathrm J+\gamma \mathrm S$ is included in a larger level set of $\mathrm S$. Now, for all $y \in \mathbb R$, the set $\lbrace \Sigma \colon \Sigma\text{ asymptotic routeing strategy, } \mathrm S(\Sigma) \leq y \rbrace$ is compact, because it is closed and contained in the set $\lbrace \Sigma \colon \Sigma~ \text{asymptotic routeing strategy},~ \sum_{k=1}^{\kmax} k \nu_k(W^k) \leq y' \rbrace$ for all sufficiently large $y' \in \mathbb R$, and such sets are compact by Prohorov's theorem. These together allow us to conclude \eqref{beta0variation}.


From this, parts \eqref{LDP-variationbeta=0} and \eqref{convergence-variationbeta=0} of Proposition \ref{prop-variationbeta=0} can be derived analogously to how Theorem~\ref{thm-LDP} was derived from Theorem~\ref{theorem-variation1} in Section~\ref{sec-LDPproof}. The additional fact that the rate function $\mathrm J+\mu(W) \log \kmax$ has compact level sets holds because relative entropies with respect to fixed reference measures have compact level sets \cite[Section 6.2]{DZ98}.

Lastly, we verify \eqref{minimizer-variationbeta=0}, i.e., we prove that \eqref{nukminimizerbeta=0} is the unique minimizer of \eqref{beta0variation}. The fact that the set of minimizers of the variational formula on the right-hand side of \eqref{beta0variation} is non-empty, compact and convex follows similarly to Lemma~\ref{lemma-existence}, again by Prohorov's theorem and the compactness of the sets $\lbrace \Sigma \colon \Sigma ~\text{asymptotic routeing strategy}, {\mathrm S}(\Psi) \leq y \rbrace$, $y>0$. Further, an argument analogous to Lemmas~\ref{lemma-everywherepositive} and \ref{lemma-densitybounds} shows that for all minimizers $\Sigma=(\nu_k)_{k \in[\kmax]}$, we have that $\nu_k \ll \mu^{\tensor k} \ll \nu_k$ and $\smfrac{\d \nu_k}{\d \mu^{\tensor k}}$ is bounded away from zero, for all $k \in [\kmax]$. Deriving the Euler--Lagrange equations similarly to Section~\ref{sec-minimizerproof}, it follows that \eqref{nukminimizerbeta=0}--\eqref{Adefnew} hold for any minimizer $\Sigma=(\nu_k)_{k \in [\kmax]}
$ of \eqref{beta0variation}. This also implies that the minimizer $\Sigma$ is unique. Thus, we 
conclude Proposition \ref{prop-variationbeta=0}. \ProofEnde

\appendix
\section{Representations of the entropy term}\label{sec-entropyrepresentation}
We defined the entropy term $\Psi \mapsto \mathrm I(\Psi)$ via the formula \eqref{quenchedentropy}, which we interpreted in Section~\ref{sec-Discussion-varformula}. It is easy to see that \eqref{quenchedentropy} is equivalent to the representation in \eqref{I-minimizer}, which we used for analytical investigations. Now we show that \eqref{quenchedentropy} is equivalent to \eqref{I-proof}, which arises from the combinatorics in Section~\ref{sec-asymptotics}.

Recall that for $k \in \N$ and $\xi,\eta \in \Mcal(W^k)$, we have $\mathcal H_{W^k}(\xi | \eta)=H_{W^k}(\xi |\eta) -\xi(W^k)+\eta(W^k)$. Further, for an admissible trajectory setting $\Psi= ((\nu_k)_{k \in[\kmax]},(\mu_m)_{m \in \N_0})$, recall the measure $M=\sum_{k=1}^{\kmax} \sum_{l=1}^{k-1}\pi_l\nu_k=\sum_{m=0}^{\infty} m \mu_m$.  Starting from the definition of $\mathrm I(\cdot)$, in \eqref{quenchedentropy}, we compute
\begin{align*}
 \mathrm I(\Psi)&=\sum_{k=1}^{\kmax} \mathcal H_{W^k}\big(\nu_k \mid \mu \tensor M^{\tensor (k-1)}\big)+ \sum_{m=0}^{\infty} \mathcal H_{W}(\mu_m \mid \mu c_m)+\mu(W) \Big( 1-\sum_{k=1}^{\kmax} M(W)^{k-1}\Big)-\frac1\e \\ 
&=  \sum_{k=1}^{\kmax} H_{W^k}\big(\nu_k \mid \mu \tensor M^{\tensor (k-1)}\big) - \sum_{k=1}^{\kmax} \nu_k(W^k) + \mu(W) \sum_{k=1}^{\kmax} M(W)^{k-1}  + \sum_{m=0}^{\infty} \mathcal H_{W}(\mu_m \mid \mu c_m) \\ & \qquad +\mu(W) \Big( 1-\sum_{k=1}^{\kmax} M(W)^{k-1}\Big)-\frac1\e \\ 
&=  \sum_{k=1}^{\kmax} H_{W^k}\big(\nu_k \mid \mu \tensor M^{\tensor (k-1)}\big) + \sum_{m=0}^{\infty} H_{W}(\mu_m \mid \mu c_m) -\frac1\e,
\end{align*}
where we used \eqref{mumentropymassidentity}, and the fact that $\sum_{k=1}^{\kmax} \nu_k(W^k)=\mu(W)$ by (i) in \eqref{constraints}. By the definition of the measure $M$, it suffices to show that 
\[ \sum_{k=1}^{\kmax} H_{W^k}\big(\nu_k \mid \mu \tensor M^{\tensor (k-1)}\big)  = \sum_{k=1}^{\kmax} H_{W^k} (\nu_k | \mu^{\tensor k})-H_W(M|\mu). \numberthis\label{2entropies} \]
Clearly, if any of the sides of \eqref{2entropies} is infinite, then so is the other side. Else, we verify \eqref{2entropies} as follows
\begin{align*}
\sum_{k=1}^{\kmax} &H_{W^k}\big(\nu_k \mid \mu \tensor M^{\tensor (k-1)}\big) 
\\ = & \sum_{k=1}^{\kmax} \int_{W^k} \d \nu_k(x_0,\ldots,x_{k-1}) \Big[ \log \frac{\d \nu_k}{\d \mu^{\tensor k}}(x_0,\ldots,x_{k-1}) - \log \frac{\d (\mu \tensor M^{\tensor (k-1)})}{\d \mu^{\tensor k}}(x_0,\ldots,x_{k-1}) \Big]  
\\ = &  \sum_{k=1}^{\kmax} \int_{W^k} \d \nu_k(x_0,\ldots,x_{k-1}) \Big[ \log \frac{\d \nu_k}{\d \mu^{\tensor k}}(x_0,\ldots,x_{k-1}) - \log \Big( \prod_{l=1}^{k-1} \frac{\d M}{\d \mu}(x_l) \Big) \Big] 
\\ = & \sum_{k=1}^{\kmax} \int_{W^k} \d \nu_k(x_0,\ldots,x_{k-1})  \log \frac{\d \nu_k}{\d \mu^{\tensor k}}(x_0,\ldots,x_{k-1}) - \sum_{k=1}^{\kmax} \sum_{l=1}^{k-1} \int_{W} \pi_l\nu_k(\d x_l) \log \frac{\d M}{\d \mu}(x_l) 
\\ = & \sum_{k=1}^{\kmax} H_{W^k}(\nu_k | \mu^{\tensor k}) - H_W(M|\mu).
\end{align*}

\subsection*{Acknowledgements} The authors thank B. Jahnel, C. Hirsch, M. Klimm, S. Morgenstern and M. Renger for interesting discussions and comments. The second author acknowledges support by the Berlin Mathematical School (BMS) and the DFG RTG 1845 \lq Stochastic Analysis with Applications in Biology, Finance and Physics\rq.


\begin{thebibliography}{WWW98}

\bibitem[AK08]{AK08}
\newblock {\sc S. Adams} and {\sc W. König,}
\newblock {Large deviations for many Brownian bridges with symmetrised initial-terminal condition.}
\newblock {\em Probab. Theory Relat. Fields} {\bf 142:1},  79–124 (2008).


\bibitem[BB09]{BB09}
\newblock {\sc F. Baccelli} and {\sc B. Błaszczyszyn,} 
\newblock {\em Stochastic Geometry and Wireless Networks: Volume I: Theory.} 
\newblock Now Publishers Inc. (2009).

\bibitem[BC12]{BC12}
\newblock {\sc F. Baccelli} and {\sc C.~S. Chen},
\newblock Self-Optimization in Mobile Cellular Networks: Power
Control and User Association,
\newblock \emph{IEEE International Conference on Communications}, Cape Town (2010).

\bibitem[CCS16]{CCS16} 
\newblock {\sc R. Colini-Baldeschi}, {\sc Roberto Cominetti} and 
{\sc Marco Scarsini}, 
\newblock On the Price of Anarchy of Highly Congested
Nonatomic Network Games,
\newblock {\em 	arXiv:1605.03081} (2016).


\bibitem[Ch09]{Ch09}
\newblock {\sc X.~Chen},
\textit{Random Walk Intersections: Large Deviations and Related Topics}.
Mathematical Surveys and Monographs, AMS, Vol. 157, Providence, RI.
(2009).

\bibitem[DZ98]{DZ98}
{\sc A.~Dembo} and {\sc O.~Zeitouni},
{\it Large Deviations Techniques and Applications\/}, 
2nd edition, Springer,  Berlin (1998).

\bibitem[GK00]{GK00}
\newblock {\sc P. Gupta} and {\sc P.~R.~Kumar},
\newblock The Capacity of Wireless Networks,
\newblock \emph{IEEE Trans. Inform. Theory} {\bf 46:2}, (2000).

\bibitem[GT08]{GT08}
\newblock {\sc A.~J.~Ganesh} and {\sc G.~L.~Torrisi},
\newblock Large Deviations of the Interference in a Wireless Communication Model, 
\newblock \emph{IEEE Trans. Inform. Theory} {\bf 54:8}, (2008).

\bibitem[G11]{G11}
\newblock {\sc H-O.~Georgii},
\newblock {\em Gibbs Measures and Phase Transitions},
\newblock De Gruyter, 2nd edition (2011).

\bibitem[GZ93]{GZ93}
\newblock {\sc H-O.~Georgii} and {\sc H.~Zessin},
\newblock Large deviations and the maximum entropy principle for marked point random fields,
\newblock {\em Probab. Theory Relat. Fields} {\bf 96:2}, 177--204 (1993).

\bibitem[HJKP16]{HJKP15a}
\newblock {\sc C.~Hirsch}, {\sc B.~Jahnel}, {\sc P.~Keeler} and {\sc R.~Patterson},
\newblock Large-deviation principles for connectable receivers in wireless networks,
\newblock {\em Advances in Applied Probability}, {\bf 48}, 1061--1094 (2016). {\em arXiv:1506.00576}.

\bibitem[HJKP18]{HJKP15}
\newblock {\sc C.~Hirsch}, {\sc B.~Jahnel}, {\sc P.~Keeler} and {\sc R.~Patterson},
\newblock Large deviations in relay-augmented wireless networks,
\newblock \emph{Queueing Systems} \textbf{88}, 349--387 (2018). {\em arXiv:1510.04146}.

\bibitem[HJP16]{HJP16}
\newblock {\sc C.~Hirsch}, {\sc B.~Jahnel} and {\sc R.~Patterson},
\newblock Space-time large deviations in capacity-constrained relay networks,
\newblock \emph{Latin American Journal of Probability and Statistics} {\textbf 15}. 587--615 (2018), {\em arXiv:1609.06856}.


\bibitem[HJ17]{HJ17}
\newblock {\sc C.~Hirsch} and {\sc B.~Jahnel},
\newblock  Large deviations for the capacity in dynamic spatial relay networks,
\newblock  {\em arXiv:1712.03763} (2017).

\bibitem[K93]{K93}
\newblock {\sc  J. F. C. Kingman},
\newblock {\it Poisson Processes},
\newblock Oxford University Press, New York (1993).


\bibitem[KM02]{KM02}
\newblock{\sc W.~K\"onig} and  {\sc P.~M\"orters},
\newblock Brownian intersection local times: upper tail asymptotics and thick points.
\textit {Ann. Probab.} {\textbf {30}}, 1605--1656 (2002).


\bibitem[KM13]{KM13}
\newblock {\sc W.~K\"onig} and  {\sc C.~Mukherjee },
\newblock Large deviations for Brownian intersection measures,
\newblock {\it Commun. Pure Appl. Math.} {\bf 66:2}, 263-306 (2013).

\bibitem[KT18]{KT17b}
\newblock {\sc W.~K\"onig} and {\sc A.~Tóbiás},
\newblock Routeing properties in a Gibbsian model for highly dense multihop networks,
\newblock {\em arXiv:1801.04985v2} (2018).

\bibitem[M18]{M18}
\newblock {\sc S.~Morgenstern},
\newblock \emph{Markov Chain Monte Carlo for
Message Routing},
\newblock master's thesis, TU Berlin (2018).

\bibitem[SPW07]{SPW07}
\newblock {\sc H. Song}, {\sc M. Peng} and {\sc W. Wang},
\newblock Node Selection in Relay-based Cellular Networks,
\newblock {\em IEEE 2007 International
Symposium on Microwave, Antenna, Propagation, and EMC Technologies For Wireless Communications}.

\bibitem[T08]{T08}
\newblock {\sc J. Trashorras},
\newblock Large Deviations for Symmetrised Empirical Measures.
\newblock {\em J. Theoret. Probab.} {\bf 21}, 397--412 (2008).


\end{thebibliography}
\end{document}